\documentclass[11pt]{article}

\usepackage{amsmath,amssymb,amsthm,hyperref,color}
\usepackage{enumerate}
\usepackage{verbatim}
\usepackage{bookmark}
\usepackage{graphicx,tikz}

\usepackage{thmtools}
\usepackage{thm-restate}
\usepackage[noadjust]{cite}

\usepackage[T1]{fontenc}
\usepackage{libertine} 
\usepackage[libertine]{newtxmath}

\hypersetup{colorlinks=true,citecolor=blue, linkcolor=blue, urlcolor=blue}{\tiny }
\usepackage[margin = 1in]{geometry}

\newtheorem{lemma}{Lemma}[section]
\newtheorem{theorem}[lemma]{Theorem}

\newtheorem{fact}[lemma]{Fact}
\newtheorem{observation}[lemma]{Observation}

\newtheorem{corollary}[lemma]{Corollary}

\theoremstyle{remark}
\newtheorem{remark}{Remark}
\theoremstyle{definition}
\newtheorem{definition}[lemma]{Definition}

\DeclareMathOperator{\Pp}{{\mathbf{P}}}

\usepackage{bm}
\usepackage{bbm}

\def\E{\mathbb{E}}

\def\var{\text{Var}}

\newcommand{\ord}{{\ensuremath{\mathsf{ord}}}}
\newcommand{\dis}{{\ensuremath{\mathsf{dis}}}}

\renewcommand{\epsilon}{\varepsilon}

\newcommand{\betau}{\beta_{\mathsf{u}}}
\newcommand{\betac}{\beta_{\mathsf{c}}}
\newcommand{\betas}{\beta_{\mathsf{s}}}

\newcommand{\TV}{\textsc{tv}}

\renewcommand{\Pr}{\mathbb{P}}

\newcommand{\thetas}{\theta_{\mathsf{s}}}
\newcommand{\thetar}{\theta_{\mathsf{r}}}

\newcommand{\mr}{m_{\mathsf{r}}}

\newcommand{\ka}{k_{\mathsf{a}}}
\newcommand{\kia}{k_{\mathsf{ia}}}
\newcommand{\Wa}{W_{\mathsf{a}}}
\newcommand{\Wia}{W_{\mathsf{ia}}}

\title{Mean-field Potts and random-cluster dynamics from high-entropy initializations}
\date{}

\author{Antonio Blanca \thanks{Department of CSE, Pennsylvania State University, ablanca@cse.psu.edu} \and Reza Gheissari \thanks{Department of Mathematics, Northwestern University, gheissari@northwestern.edu} \and Xusheng Zhang \thanks{Department of CSE, Pennsylvania State University, xushengz@psu.edu}}

\begin{document}

\maketitle

\begin{abstract}
    A common obstruction to efficient sampling from high-dimensional distributions with Markov chains is the multimodality of the target distribution because they may get trapped far from stationarity. Still, one hopes that this is only a barrier to the mixing of Markov chains from \emph{worst-case} initializations and can be overcome by choosing high-entropy initializations, e.g., a product or weakly correlated distribution. Ideally, from such initializations, the dynamics would escape from the saddle points separating modes quickly and spread its mass between the dominant modes with the correct probabilities.   
    
    In this paper, we study convergence from high-entropy initializations for the random-cluster and Potts models on the complete graph---two extensively studied high-dimensional landscapes that pose many complexities like discontinuous phase transitions and asymmetric metastable modes. 
    We study the Chayes--Machta and Swendsen--Wang dynamics for the mean-field random-cluster model and the Glauber dynamics for the Potts model. We sharply characterize the set of product measure initializations from which these Markov chains mix rapidly, even though their mixing times from worst-case initializations are exponentially slow. Our proofs require careful approximations of projections of high-dimensional Markov chains (which are not themselves Markovian) by tractable 1-dimensional random processes, followed by analysis of the latter's escape from saddle points separating stable modes. 
\end{abstract}

\thispagestyle{empty}

\setcounter{page}{1}

\section{Introduction}
A ubiquitous and generically hard computational task is to minimize a high-dimensional function $f$ over a discrete space $\{1,...,q\}^n$;
closely connected is the problem of 
sampling from the probability distribution with mass proportional to $e^{ - \beta f}$ for $\beta$ large. 
 The function $f$ is often viewed as an energy landscape in statistical physics or as a loss function in machine learning, and the $\beta$-large setting is referred to as the low-temperature regime. 
The difficulty is induced by the possibility of $f$ having several minima with large basins of attraction (or, equivalently, by the multimodality of the induced distribution); this poses a barrier to traditional optimization/sampling algorithms like gradient descent and Markov chain sampling, at least when initialized from a worst-case state, e.g., in a sub-optimal mode.

Still, a black box approach to these tasks would select the initial state randomly from a product measure or more general high-entropy distribution and would hope that this can circumvent (some of) the bottlenecks in the space. 
When we say high-entropy initialization, we mean distributions that are well-spread over the probability space, in contrast to, e.g., worst-case initializations, or initializations in some extremal energy state. 
Ideally, a high-entropy initialization distributes its mass across the space in such a way that dynamics are primarily driven by diffusion away from the saddles separating dominant modes, 
picking the basins to fall into with the correct probability. (Here and throughout we are using the terminology ``saddle point" informally, by analogy to critical points of landscapes in continuous spaces.) However, rigorous studies of convergence rates to stationarity from high-entropy initializations are severely lacking and generically difficult.

We focus on the theoretical study of high-entropy initializations in the context of Markov chains for spin system models, such as the Ising, Potts, and random-cluster models. In the statistical physics literature, questions of dynamics from high-entropy initializations have a long (empirical) history. A paradigmatic version of this is in the Ising model on the integer lattice graph $\mathbb Z^d$ at low temperatures, where it is widely expected (though entirely open to prove) that the Glauber dynamics (the natural reversible local Markov chain) initialized from a uniform-at-random assignment of $\pm 1$ mixes in polynomial time: see, e.g., the review paper~\cite{Bray-phase-ordered-kinetics} for the rich physics literature about this process. 
Indeed, this question requires understanding the motion of interfaces separating regions where plus and minus respectively dominate, and is likely even harder than the notorious problem of showing a polynomial bound for the worst-case initialization mixing time of Ising Glauber dynamics in a box with plus boundary conditions~\cite{Martinelli-phase-coexistence,MaTo,LMST}.

Even in other geometries, for instance on trees and random graphs, bounding the speed of convergence of the low-temperature Ising Glauber dynamics initialized 
uniformly at random from $\{\pm 1\}^n$
seems to be mathematically very challenging. (See, e.g.,~\cite{KaMo11,SahSawhney-majority-dynamics,Fountoulakis-majority-dynamics,Tran-Vu-majority-dynamics,GNS-zero-temp-Ising-RG} for recent progress on zero-temperature $\beta = \infty$ versions of this problem, and~\cite{CaMa06} for analysis on the tree initialized from a biased product measure.) In the special case of the Ising Glauber dynamics on the complete graph, known as the mean-field model, the process reduces perfectly to a 1-dimensional birth-death Markov chain. Here~\cite{LLP,DLP-censored-Glauber} showed 
the escape from the saddle corresponding to balanced configurations at low temperatures is fast, implying $O(n \log n)$ mixing when initialized from a product of fair coin tosses, despite slow mixing from worst-case. 

When one generalizes from the Ising model to the $q$-state Potts model, the above questions become significantly more complicated due to the presence of a higher-dimensional space of spin counts, the possibility of slow mixing at intermediate temperatures, and the lack of symmetry between the modes. The closely related random-cluster model faces similar difficulties, as well as the non-locality of its interactions and update rules. Together, these yield a rich class of models for investigating the above-described expected benefits and possible limitations for high-entropy (e.g., product) initializations to overcome the slow mixing of standard Markov chains.

Let us note that purely from an approximate sampling perspective, it has been known that sampling from the ferromagnetic Ising model is tractable since~\cite{JSIsing,RanWil}, but general sampling in the ferromagnetic Potts model is \#BIS hard~\cite{GoldbergJerrum,GSVY}. This has led to much recent work towards finding general criteria (on $q$, the underlying graph, and the temperature) under which sampling is algorithmically tractable, e.g.,~\cite{BhRa04-simulated-tempering,HPR-Algorithmic-Pirogov-Sinai,BCHPT-Potts-all-temp,carlson,JKP,CGGPS,HeJePe20,Huijben-Patel-Regts,BlGh23}. 
Of particular relevance to our work is a series of recent results using special initializations of Markov chains to overcome bottlenecks, specifically using a priori knowledge of the global minimizers of the energy landscape to initialize near them (e.g., initializing an Ising model in the all-plus or all-minus states with probability $1/2$ each), and showing that mixing locally in those modes is rapid: see~\cite{GS22,GS23,GGS23,BDE-ERGM-24} for recent works in this direction for the Ising, random-cluster, and exponential random graph models. 

By contrast, the interest in initializations from product measures (and similar high-entropy distributions), is towards understanding success/failure of black-box approaches that do not require prior knowledge of the target distribution's modes. In particular, such initializations do not rely on, nor hope for, local convexity of the landscape near the initialization, instead relying on the non-convexity to allow access to multiple modes, making the analysis significantly more challenging. We also note the close connections between  high-entropy initializations and simulated annealing schemes, for which we obtain new results: see Section~\ref{subsec:simulated}. 

In this paper, we study the convergence of Markov chains in spin systems in the presence of metastability and phase coexistence from high-entropy initializations (predominantly, product measures). By
phase coexistence, we mean the presence of multiple modes of roughly equal weights, and by metastability, we mean the presence of multiple modes of vastly different weights; in both situations, Markov chains get trapped for exponential time scales far from stationarity.
We study two canonical Markov chains,
the Chayes--Machta (CM) dynamics for the mean-field random-cluster model and the Glauber dynamics for the mean-field Potts model; our results also extend to the well-studied Swendsen-Wang (SW) dynamics. These chains are known to be slow mixing from worst-case initializations for ranges of intermediate, near-critical, and low temperatures. We study their equilibration times in their slow-mixing regimes from product measure initializations and sharply characterize the families of product initializations that lead to optimally fast mixing. 

Already on the complete graph, the understanding of mixing from such initializations poses significant challenges.
Our analysis requires a careful understanding of low-dimensional projections of the Markov chains near unstable saddle points of the landscape separating its dominant modes; namely, in regions where the fluctuations of the projected chain compete on the same scale with its drift as well as regions of strict non-convexity.  
Unlike the case of the mean-field Ising Glauber dynamics where the projection onto the count of plus spins is itself a 1-dimensional birth-and-death Markov chain, in our setting, these projections are not themselves Markov chains. Instead, we approximate the projections of truly high-dimensional Markov chains on a ``good'' set of configurations by more tractable 1-dimensional chains whose diffusion away from saddle points separating modes we then study. 
Altogether, this amounts to a significantly more refined analytic control of the Markov chains than previous works on these dynamics.
See Section~\ref{subsec:proof-outlines} for more on the proof ideas.

\subsection{Main results}

The (ferromagnetic) Potts model on the $n$-vertex complete graph $G = ([n],\binom{[n]}{2})$ at inverse temperature $\beta>0$ is the probability distribution $\pi_\beta$ over spin assignments $\sigma \in \{1,\dots, q\}^{n}$ to vertices of $G$ such that 
\begin{align}\label{eq:Potts-measure}
    \pi_{\beta}(\sigma) = \frac{1}{Z_{\beta,q}} \cdot \exp\Big( \frac{\beta}{n} \sum_{i \neq j} \mathbf{1}\{\sigma_i = \sigma_j\}\Big)\,,    
\end{align}
where $Z_{\beta,q}$ is the normalizing constant known as the partition function and $\sigma_i \in \{1,\dots ,q\}$ denotes the spin of vertex $i$. 

Closely related is the random-cluster model on $G$ with parameters $q$, now allowed to be in $(0,\infty)$, and $p\in [0,1]$. This is a model 
that assigns to edge-subsets $A$ of $G$ probability
\begin{align}\label{eq:RC-dynamics}
    \mu_{p}(A) = \frac{1}{Z_{p,q}} \cdot \Big(\frac{p}{1-p}\Big)^{|A|} q^{k([n],A)}\,,
\end{align}
where $k([n],A)$ is the number of connected components in the subgraph induced by $A$, and $Z_{p,q}$ is the corresponding partition function. Note that when $q=1$, the mean-field random-cluster model corresponds exactly to the Erd\H{o}s--R\'{e}nyi random graph model, but when $q\ne 1$ the cluster weighting can change the phenomenology significantly. When $q\ge 2$ is integer and $p=1-e^{-\beta/n}$, the random-cluster model is closely linked to the Potts model; in particular, if one assigns spins to the components of $A \sim \mu_\beta$ independently and uniformly at random among $\{1,...,q\}$, the result is a sample from the Potts distribution at inverse temperature $\beta$. To unify the discussion, via the reparametrization $p=1-e^{-\beta/n} \approx \beta/n$, we henceforth only work with a temperature parameter $\beta$, even when discussing the random-cluster model, and we write $\mu_{\beta}$ for $\mu_{p}$.

The two canonical Markov chains we consider are the Glauber dynamics for the Potts model and the Chayes--Machta (CM) dynamics for the random-cluster model. 
The Glauber dynamics is the Markov chain which at each time-step picks a vertex $i$ uniformly at random among $[n]$ and resamples its spin conditionally on the remainder of the configuration; namely, resamples it to take spin $k\in \{1,...,q\}$ with probability proportional to $\exp(\frac{\beta}{n} \sum_{j} \mathbf{1}\{\sigma_j = k\})$. 
The Potts Glauber dynamics is typically exponentially slow to equilibrate at low temperatures due to bottlenecks between spin configurations where each of the $q$ colors dominates. 
The CM dynamics is an appealing alternative that uses the connection between the Potts and random-cluster models
to overcome the low-temperature bottlenecks of the Glauber dynamics. More precisely, the CM dynamics is the following Markov chain defined for general $q \ge 1$ as follows: from an edge-subset $X_t$ generate $X_{t+1}$ via 
\begin{enumerate}
    \item \emph{Activation step}: Independently for each connected component $\mathcal C$ of $X_t$, with probability $\frac{1}{q}$ label all vertices of $\mathcal C$ \emph{active} and otherwise label all vertices of $\mathcal C$ \emph{inactive}. 
    \item \emph{Percolation step}: Independently for each edge $e$ both of whose endpoints are active, include $e$ in $X_{t+1}$ with probability $\beta/n$. For all other edges $e$, let $X_{t+1}(e) = X_t(e)$. 
\end{enumerate}
The CM dynamics is a generalization to non-integer values of $q$ of the famous Swendsen--Wang (SW) dynamics for the Potts model. Indeed for integer $q$, if the activation step of the CM dynamics is performed by coloring the components of $X_t$ independently among $[q]$, and activating one of the color classes, then this produces a Markov chain on Potts configurations which is basically equivalent to the SW dynamics. (Technically, the SW chain treats each color class as an activated set and does percolation steps within each of them simultaneously before recoloring.) 
As far as speeds of convergence are concerned, the CM and SW are thus closely related~\cite{BS15, FGW23}; our results all also apply to the SW dynamics without significant modifications; see~Remark~\ref{rem:SW}.

The standard way to quantify the speed of convergence of a Markov chain is the mixing time, i.e., the time it takes to reach total-variation distance $\varepsilon$ to stationarity, either from a prescribed initialization, or from a worst-case initialization. For the worst-case mixing, as soon as $\varepsilon<1/2$, the TV-distance decays exponentially fast, so one usually takes as a convention $\varepsilon = 1/4$. In the context of this paper, we are considering mixing times from a specified initial distribution, and often in settings where the worst-case initialization mixing time is exponentially large. This does not satisfy the same boosting property (i.e., exponential decay of TV-distance once it is $<1/2$) as the worst-case mixing time does, so when we refer to bounds on mixing time from an initial distribution, we mean that in that time, we can attain any TV-distance $\epsilon>0$ independent of $n$. Note that when we discuss mixing from an initial distribution, the law of the process is over both the initialization and the run of the Markov chain. 

To describe what is known about the mixing time behavior of the CM and Glauber dynamics, we first recall the static phase diagram of the mean-field Potts and random-cluster models. 
These have been extensively analyzed 
and
are controlled by three threshold values $\betau\le  \betac \le \betas$ of $\beta$. The middle one, $\betac$, is the critical point marking the distinction between order (a dominant spin class in the Potts setting or, equivalently, a ``giant'' connected component in the random-cluster) and disorder (balanced spin classes or small connected components) in a typical sample from the distribution. 
When $q>2$, $\betau$ and $\betas$
are two additional critical points marking the onset of metastability for the ordered phases and the end of metastability of the disordered phase, respectively (see Figure~\ref{fig:random-cluster-phase_structure}, Figure~\ref{fig:Potts-phase_structure} and~\cite{BGJ-mean-field-random-cluster,costeniuc2005complete, LL06}.)

\begin{figure}[t]
    \centering
    \begin{tikzpicture}[scale=0.8]
    \draw (0, 0) node[inner sep=0] {\includegraphics[scale=0.4]{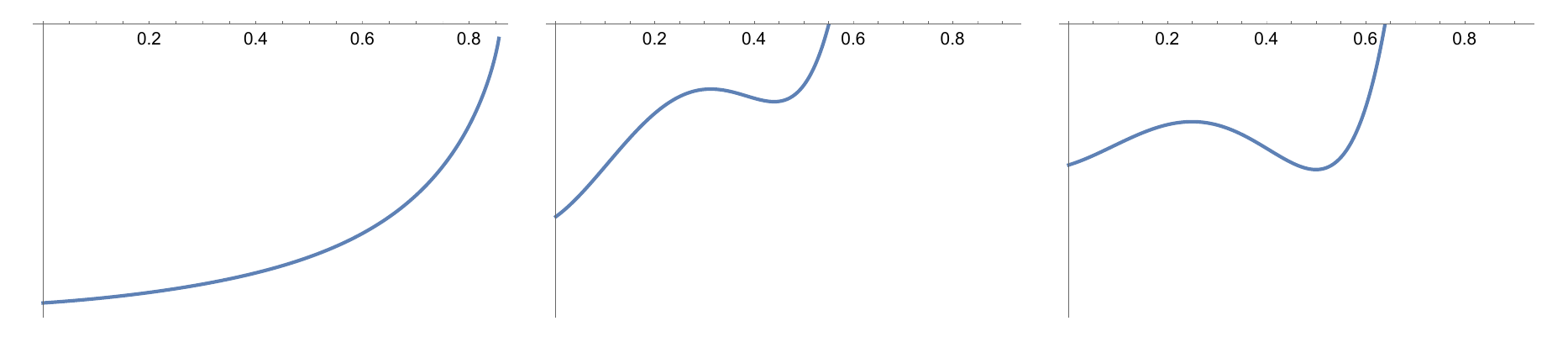}};
    \draw (-4.8, -1.7) node {\tiny{$\beta < \betau$}};
    \draw (0.31, -1.7) node {\tiny{$\beta \in (\betau,\betac)$}};
    \draw (5.36, -1.7) node {\tiny{$\beta = \betac$}};
    \draw (0, -4) node[inner sep=0] {\includegraphics[scale=0.4]{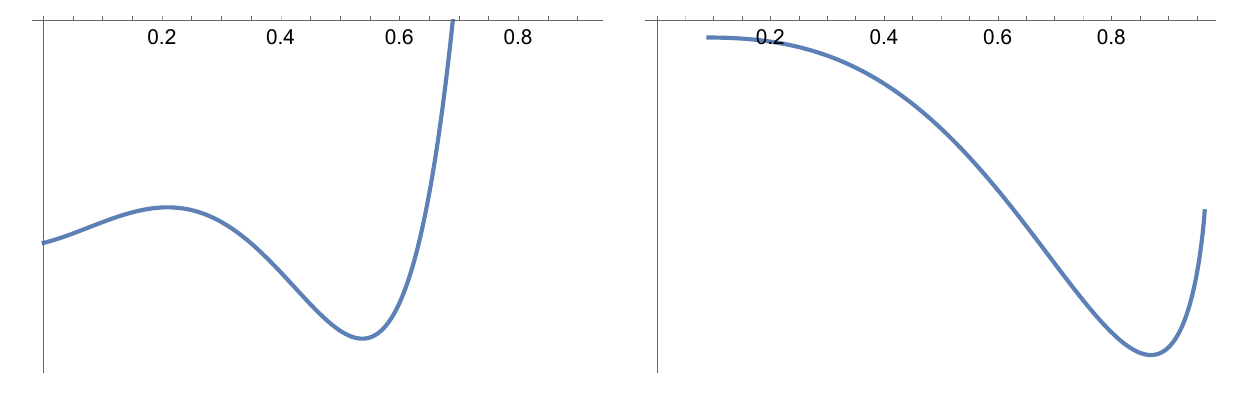}};
    \draw (-2.35, -5.7) node {\tiny{$\beta \in (\betac,\betas)$}};
    \draw (2.8, -5.7) node {\tiny{$\beta > \betas$}};

     \filldraw[color=orange, fill=orange](-0.825,0.82) circle (0.04);
      \filldraw[color=orange, fill=orange](4.13,0.502) circle (0.04);
       \filldraw[color=orange, fill=orange](-3.83,-4.09) circle (0.04);
    
    \end{tikzpicture}
    \caption{The negative log probability (i.e., restricted free energy) of the largest connected component's size in the $n$-vertex mean-field random-cluster model with $q=3$ (for large $n$). The $x$-axes are the fraction of vertices in the largest component, and the $y$-axes represent the logarithm of the total weight of the configurations with such largest component sizes divided by $n$. We emphasize that the landscape is a function of all the component sizes and can exhibit more complexity in the other directions, but our results show this 1-dimensional projection approximately governs the dynamics from product initializations. The orange-marked dots are the saddle points with respect to which the fast initializations for Chayes--Machta dynamics are characterized in Theorem~\ref{thm:SWCM-dynamics}}
    \label{fig:random-cluster-phase_structure}
\end{figure}
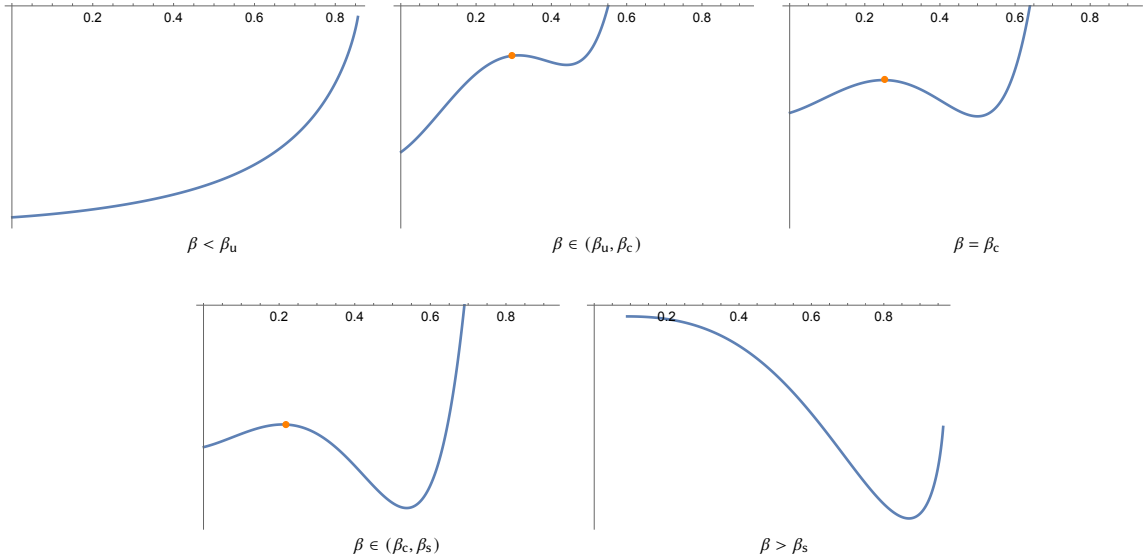

The above static description of these mean-field models is closely related to the worst-case initialization mixing time of 
its dynamics. For instance, the Potts Glauber dynamics transitions from fast $O(n\log n)$ convergence when $\beta<\betau$ to exponentially slow as soon as $\beta>\betau$; first, when $\beta \in (\betau,\betac)$, due to the metastability of $q$ basins corresponding to ordered configurations (one for each spin) preventing equilibration to the disordered phase, then when $\beta \ge \betac$ due to the coexistence of the $q$ ordered phases: see~\cite{CDLLPS-mean-field-Potts-Glauber}. The CM dynamics is similarly fast when $\beta<\betau$, and still exponentially slow in the window $(\betau,\betas)$. This is because, even though the $q$ ordered phases are equivalent as far as the CM dynamics is concerned, there is still metastability between the ordered or disordered phases. The benefit of CM dynamics is seen once $\beta>\betas$, where it
becomes fast again: see~\cite{LNNP11, BS15,GSV,ABthesis,GLP,BSZ} which together give the worst-case mixing behavior of the mean-field CM and SW dynamics.

We start by stating our main results for the CM dynamics demonstrating the sharp families of product initializations  
that circumvent the exponential slowdown throughout the metastability regime $\beta \in (\betau,\betas)$. 
Let $\bigotimes\text{Ber}(\lambda_0/n)$ be the product distribution where each edge is included with probability $\lambda_0/n$. (Note that $\bigotimes\text{Ber}(\lambda_0/n)$ corresponds to the classical Erd\H{o}s--R\'enyi $G(n,\lambda_0/n)$ random graph model.)

\begin{theorem}\label{thm:SWCM-dynamics}
	For every $q>2$ and $\beta\in (\betau,\betas)$, there exist $\lambda_*(\beta,q)$ and $c_*(q)$ such that  
 the CM dynamics initialized from $\bigotimes\text{Ber}(\lambda_0/n)$ mixes in $O(\log n)$ steps 
 whenever 
	\begin{enumerate}
	\item $\beta\in (\betau,\betac)$ and $\lambda_0 < \lambda_*(\beta,q) - \omega(n^{-1/2})$, \label{item:thm-CM-1}
	\item $\beta = \betac$ and $\lambda_0 = \lambda_*(\beta,q) + c_*(q) n^{-1/2} + o(n^{-1/2})$, \label{item:thm-CM-2}
	\item $\beta \in (\betac,\betas)$ and $\lambda_0 > \lambda_*(\beta,q) + \omega(n^{-1/2})$. \label{item:thm-CM-3}
\end{enumerate}
Conversely, if $\beta \in (\betau,\betas)$ and the CM dynamics is initialized from $\bigotimes\text{Ber}(\lambda_0/n)$ for $\lambda_0$ outside the corresponding regime of~\ref{item:thm-CM-1}--\ref{item:thm-CM-3} (meaning, e.g., for item 1., that $\lambda_0 \ge \lambda_*(\beta,q) - O(n^{-1/2})$), then it takes $\exp(\Omega(n))$ time to attain any $o(1)$ total-variation distance.\footnote{By this we mean that there exists an $\epsilon_0$ and a $C_0$ such that attaining any $\epsilon<\epsilon_0$ total-variation distance takes time $\exp( C_0 n)$.}
\end{theorem}
Recall that the CM dynamics is fast from every initialization when $\beta\notin (\betau,\betas)$ per~\cite{BS15,ABthesis}, so this covers all temperatures where the worst-case mixing time is slow. The constant $\lambda_*$ is explicit and is such that the expected size of the giant in an Erd\H{o}s--R\'enyi$(n,\lambda_*/n)$ is $\theta_*n$ where $\theta_*$ is the saddle point in the landscape of the random-cluster model projected onto its giant component size, separating the ordered and disordered phases: see Figure~\ref{fig:random-cluster-phase_structure}. This is the reason for the sharpness of the thresholds of Theorem~\ref{thm:SWCM-dynamics}, as our proof shows that for nice initializations, whichever side of the saddle at $\lambda_*$ the giant size starts on, it quasi-equilibrates to the corresponding phase, and then takes exponential time to leave that phase. E.g., if $\beta \in (\betac,\betas)$ and $\lambda_0< \lambda_* + O(n^{-1/2})$ then the CM dynamics has an $\Omega(1)$ chance of quasi-equilibrating to the disordered phase. This implies slow mixing when combined with the fact that for $\beta\in (\betac,\betas)$, initialized from the disordered phase, mixing is exponentially slow. The slowness from bad choices of $\lambda_0$ demonstrates that even from the perspective of product measure initializations, it is still important that some parameter of the product measure should be tuned according to the model from which one is sampling.

The most delicate part of the analysis is the $\beta=\beta_c$ case, where the coexistence of ordered and disordered phases necessitates pinning down the relative exit probabilities to the right and left of the saddle from initializations  whose giant component size is within order of the standard deviation from (in particular, microscopically near) the saddle point. 
This is the source of the constant $c_*$, which is defined in terms of left vs.\ right exit probabilities for an explicit approximating 1-dimensional Gaussian process. This is also in contrast to sampling schemes that use mixtures of extremal initializations~\cite{GS22,GGS23,GS23} where one starts a priori by matching the initialization weights with the weights of the different phases at stationarity.
This is fleshed out in a detailed proof overview in Section~\ref{subsec:proof-outlines}.

\begin{remark}\label{rem:SW}
    When $q\ge 3$ is integer, the results of Theorem~\ref{thm:SWCM-dynamics} apply without change to the SW dynamics, in which instead of one activated set per step, there are $q$ activated sets corresponding to each of the $q$ color classes and independent percolation steps within each. Because of this distinction, all components are updated in every step, and this strictly simplifies the proofs in the SW case, not having to distinguish the analyses between situations where the giant component is activated or not. Interestingly, the $c_*$ parameter in the initialization at $\beta = \beta_c$ might differ between the SW and CM dynamics. This is because the 1-dimensional Gaussian process approximations have different variances.  
\end{remark}

We now turn to our results for the Potts Glauber dynamics that show the sharp families of high-entropy initializations that overcome the exponential bottlenecks in the metastable and phase-coexistence regimes. Let $\hat{\nu}^{\otimes}(m_0)$ denote the distribution over $\{1,...,q\}^{[n]}$, where first, a dominant spin $s$ is selected uniformly at random from $\{1,\dots,q\}$, and then, independently, each vertex is assigned spin $s$ with probability $m_0 \ge 1/q$ and each of the remaining $q-1$ spins with probability $\frac{1-m_0}{q-1}$.

\begin{theorem}\label{thm:Potts-Glauber}
For every $\beta>\betau$, there exist $m_*(\beta,q)$ and $\hat c_*(q)$ such that  
the Potts Glauber dynamics initialized from $\hat{\nu}^{\otimes}(m_0)$ mixes in $O( n \log n)$ steps whenever 
\begin{enumerate}
	\item $\beta\in (\betau,\betac)$ and $m_0 < m_{*}(\beta,q) - \omega(n^{-1/2})$, \label{item:thm-Potts-1}
	\item $\beta = \betac$ and $m_0 = m_{*}(\beta,q) + \hat c_*(q)n^{-1/2} + o(n^{-1/2})$,\label{item:thm-Potts-2}
	\item $\beta \in (\betac,\betas]$ and $m_0 > m_{*}(\beta,q) + \omega(n^{-1/2})$,\label{item:thm-Potts-3}
        \item $\beta >\betas$ and $m_0 \ge 1/q$ arbitrary. 
\end{enumerate}
Conversely, if $\beta \in (\betau,\betas)$ and the Potts Glauber dynamics is initialized from $\hat{\nu}^\otimes(m_0)$ for $m_0$ outside the corresponding regime of \ref{item:thm-Potts-1}--\ref{item:thm-Potts-3} (meaning, e.g., for item 1., that $m_0 \ge m_*(\beta,q) - O(n^{-1/2})$), then it takes $\exp(\Omega(n))$ time to attain any $o(1)$ total-variation distance.
\end{theorem}

When $\beta\le \betau$ the mixing time is fast from arbitrary initializations per~\cite{CDLLPS-mean-field-Potts-Glauber}, so the above covers all temperature regimes where the worst-case mixing time is slow. The constant $m_*$ is explicit and corresponds to initializing exactly at the saddle point(s) separating the $q$ ordered phases from the disordered phase (see Figure~\ref{fig:Potts-phase_structure}). This again is the source of the sharpness of the classification. 
We note that in the low-temperature $\beta>\betas$ regime, our result includes the most delicate and important case $m_0 = 1/q$, which is the fully uniform-at-random assignment of spins to vertices; this requires understanding the subtle competition between multiple simultaneously unstable directions around the saddle. 

\begin{remark}
     The critical initialization parameters $c_*, \hat c_*$ are defined by the points at which the approximating 1D Markov chains have certain (explicit in terms of $q$) left and right exit probabilities. If one is interested in computationally finding $c_*, \hat c_*$ this can be done up to the requisite precision $o(1/\sqrt{n})$ as follows. Since the exit probabilities are smooth and monotone in the initialization (see proofs of Lemma~\ref{lem:monotonicity-of-Z-process} and  Lemma~\ref{lem:existcstar}), perform binary search on the potential value for $c_*, \hat c_*$, with $\tilde O(n)$ many runs of the 1D Markov chain to estimate the exit probabilities at each search point. This leads to a complexity of $\tilde O(n)$ for estimating the parameters.
\end{remark}
\begin{remark}
    A lower bound of $\Omega(n\log n)$ on the mixing time of the Glauber Potts dynamics will generally hold, except for the potentially delicate case, corresponding to the one special value of $m_0$ where the initialization puts the fractional size of the largest color class exactly at the stable fixed point $m_r$. For other values of $m_0$, an $\Omega(n\log n)$ lower bound follows from the concentration estimates and bounds on the derivative of the drift function as it approaches the fixed point found in Section~\ref{sec:Potts-dynamics}. We predict that the lower bound exactly at the $m_r$ initialization is instead $\Omega(n)$, but it is somewhat subtle. 
(Similarly, an $\Omega(\log n)$ lower bound holds for the CM dynamics except potentially at the one special value of $\lambda_0$ where it places the fractional giant size exactly at $\theta_r$.)   
\end{remark}

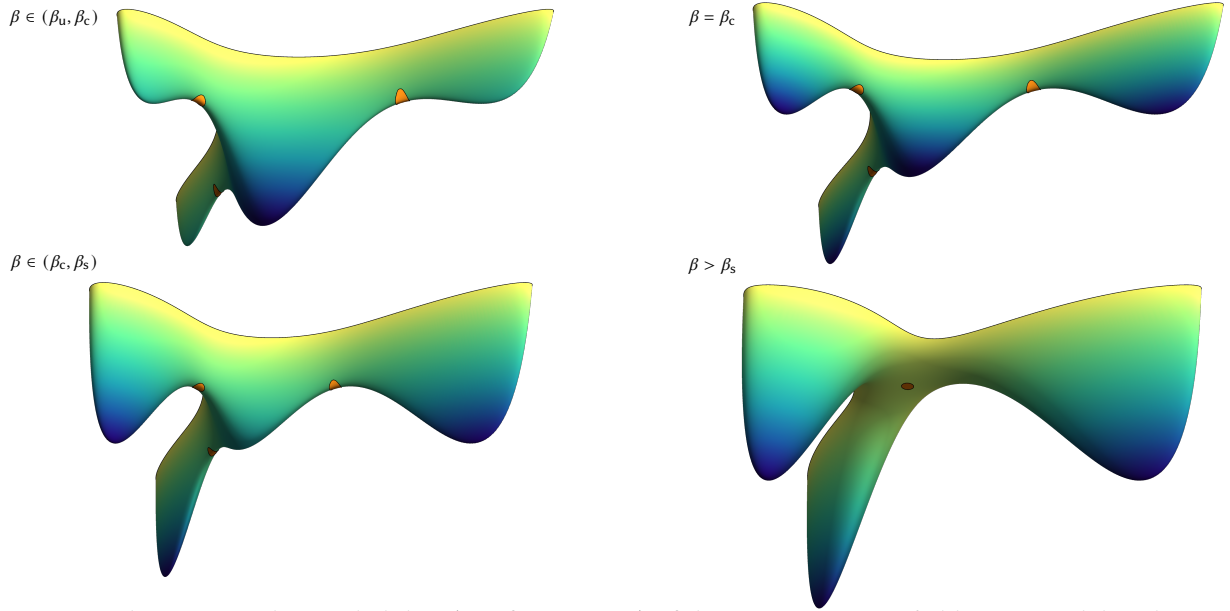
\begin{figure}[t]
    \centering

     \begin{tikzpicture}[scale=0.75]
    \draw (0, 0) node[inner sep=0] {\includegraphics[scale=0.29,trim={0 8cm 0 4cm}]{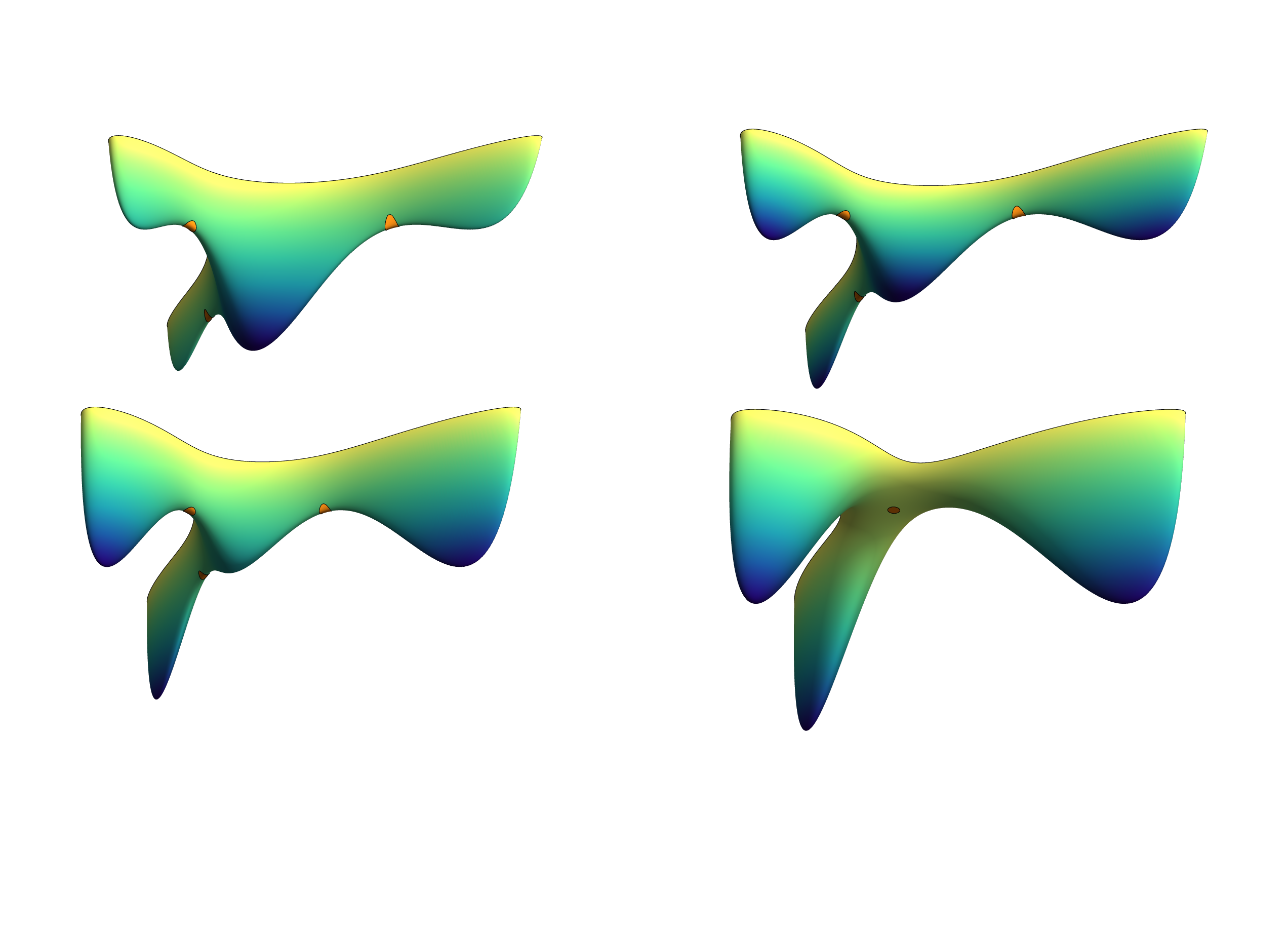}};
    \draw (-8, 3.5) node {\tiny{$\beta \in (\betau,\betac)$}};
    \draw (0.75, 3.5) node {\tiny{$\beta = \betac$}};
    \draw (-8, 0.25) node {\tiny{$\beta \in (\betac,\betas)$}};
    \draw (0.75, 0.2) node {\tiny{$\beta > \betas$}};
    \end{tikzpicture}

    \caption{The negative log probability (i.e., free energy) of the $n$-vertex mean-field Potts model with $q=3$, projected down onto the fractional spin counts; i.e., the $xy$-axes are the fraction of vertices assigned spin $1$ and $2$ respectively (which also determines the fraction of vertices with spin $3$), and the $z$-axis is the negative logarithm of the configuration weights divided by $n$. 
    The orange-marked dots are the saddle points with respect to which the fast initializations for Glauber dynamics are characterized in Theorem~\ref{thm:Potts-Glauber}.
    }
    \label{fig:Potts-phase_structure}
\end{figure}

\subsection{Initializations at other temperatures and relations to simulated annealing}\label{subsec:simulated}

Another central purpose of understanding mixing times of high-entropy initializations for Markov chains is that it is closely related to understanding simulated annealing schemes. Simulated annealing, introduced in~\cite{Kirkpatrick-simulated-annealing} (and variants like simulated tempering~\cite{Marinari-simulated-tempering}) lowers the temperature over the run of the Markov chain, so that when run at low temperatures, one is effectively initializing from the stationary measure at a nearby, but a higher-temperature (and therefore higher-entropy) distribution. This is one of the most commonly implemented modifications to vanilla MCMC to make it better able to sample from multimodal low-temperature landscapes. 
Simulated annealing and simulated tempering have been analyzed in the context of mean-field Potts dynamics by e.g.,~\cite{BhRa04-simulated-tempering,Woodard-Schmidler-Huber-tempering}, and it is known that the discontinuous phase transition at $\betac$ presents a serious obstruction to fast mixing from such schemes, for instance because the high-temperature initialization is metastable when $\beta \in (\betac, \betas)$. 

Still, our results lend insight into the approach by precisely classifying the set of temperatures $\beta_0$ for which initialization from $\mu_{\beta_0}$ or $\pi_{\beta_0}$ would be fast for the CM and Potts Glauber dynamics.

\begin{theorem}\label{thm:simulated-annealing-CM}
    For every $\beta\in (\betau,\betas)$, there exists explicit $b_*(\beta,q)\ge \betac(q)$ such that:
    \begin{enumerate}
    \item If $\beta\in (\betau, \betac)$, the CM dynamics initialized from $\mu_{\beta_0}$ for $\beta_0<b_*(\beta,q)$ will mix in $O(\log n)$ steps, but will take at least $\exp(\Omega(n))$ steps to attain $o(1)$ TV-distance for any $\beta_0 >b_*(\beta,q)$. 
    \item If $\beta = \betac$, the CM dynamics initialized from $\mu_{\beta_0}$ will take at least $\exp(\Omega(n))$ steps to attain $o(1)$ TV-distance for any $\beta_0 \ne \betac$.
    \item If $\beta \in (\betac, \betas)$, the CM dynamics initialized from $\mu_{\beta_0}$ mixes in $O(\log n)$ steps for every $\beta_0 >b_*(\beta,q)$ but will take at least $\exp(\Omega(n))$ steps to attain $o(1)$ TV-distance for any $\beta_0<b_*(\beta,q)$. 
\end{enumerate}
\end{theorem}
\begin{theorem}\label{thm:simulated-annealing-Potts}
For every $\beta\ge \betau$, there exists explicit $d_*(\beta,q)\ge \betac(q)$ such that:
\begin{enumerate}
    \item If $\beta\in (\betau, \betac)$, the Potts Glauber dynamics initialized from $\pi_{\beta_0}$ for $\beta_0<d_*(\beta,q)$ will mix in $O( n \log n)$ steps, but take at least $\exp(\Omega(n))$ steps to attain $o(1)$ TV-distance for any $\beta_0>d_*(\beta,q)$. 
    \item If $\beta = \betac$, the Potts Glauber dynamics initialized from $\pi_{\beta_0}$ will take at least $\exp(\Omega(n))$ steps to attain any $o(1)$ TV-distance for every $\beta_0 \ne \betac$.
    \item If $\beta \in (\betac, \betas)$, the Potts Glauber dynamics initialized from $\pi_{\beta_0}$ mixes in $O(n \log n)$ steps for every $\beta_0 >d_*(\beta,q)$ but will take at least $\exp(\Omega(n))$ steps to attain $o(1)$ TV-distance for any $\beta_0<d_*(\beta,q)$. 
    \item If $\beta>\betas$, the Potts Glauber dynamics initialized from $\pi_{\beta_0}$ mixes in $O(n\log n)$ steps for \emph{every} $\beta_0\ge 0$.
\end{enumerate}

\end{theorem}

\subsection{Proof outlines}\label{subsec:proof-outlines}
We focus our proof overview on the CM dynamics at the critical point $\betac$, where the equilibrium measure is roughly a $(1-\xi,\xi)$-mixture of an ordered and a disordered phase. This regime is the most technically involved part of the proof and contains most of the ideas used in other parameter regimes as well. 

In what follows, let $f(\theta) = f_{\beta,q}(\theta)$ be the expected drift for the size of the giant component in one step of CM dynamics. When $\beta \in (\betau,\betas)$, let $\theta_*$ be its unstable fixed point separating the two stable fixed points of $0$ (the disordered phase) and $\thetar$ (the ordered phase); see Figure~\ref{fig:random-cluster-drift-function} and Section~\ref{subsec:CM-drift-function} for the precise definitions. For intuition, when comparing to Figure~\ref{fig:random-cluster-phase_structure}, the drift is roughly the derivative of the log-probability landscape: a giant component of size $\theta_* n$ corresponds to the saddle point separating the modes whose minima are at no giant component, and at a giant component of size $\thetar n$.  

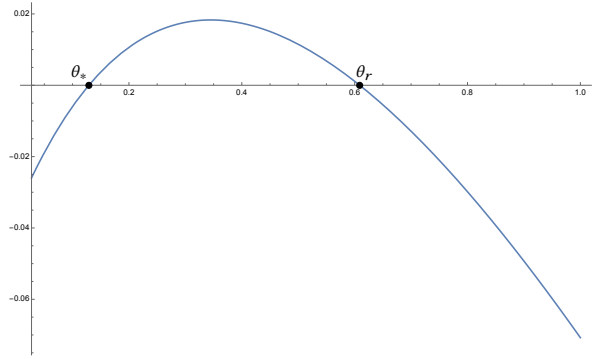
\begin{figure}[t]
    \centering

     \begin{tikzpicture}
    \draw (0, 0) node[inner sep=0] {\includegraphics[scale=0.3]{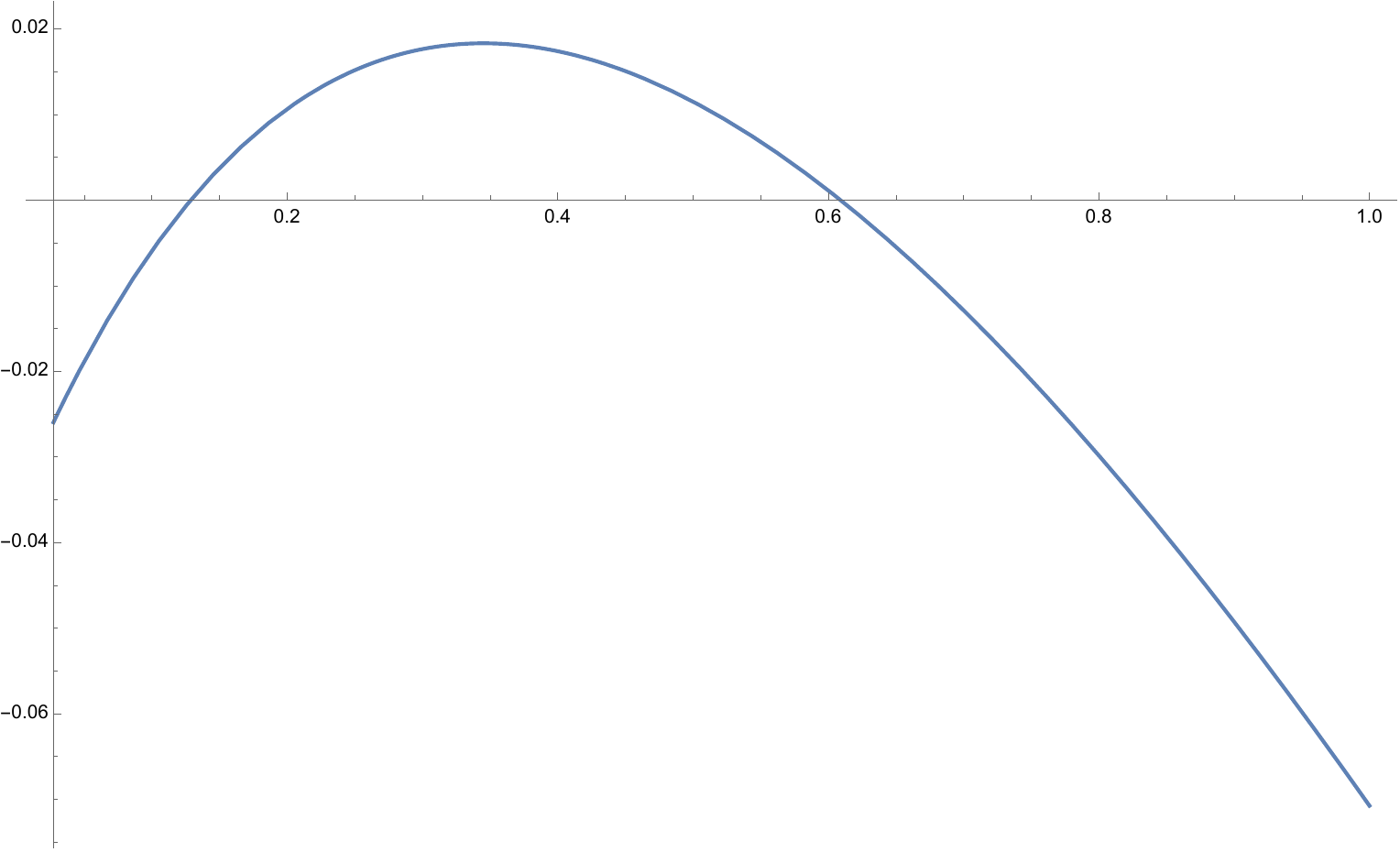}};



    \draw (0.85,1.42) node {\tiny{$\theta_r$}};
    \draw (-2.945,1.42) node {\tiny{$\theta_*$}};

    \filldraw[color=black, fill=black](-2.82,1.25) circle (0.04);
    \filldraw[color=black, fill=black](0.78,1.25) circle (0.04);


    
    \end{tikzpicture}

    \caption{The drift function $f(\theta)$ for the fractional size of the giant component under the mean-field CM dynamics in $\beta \in (\betau,\betas)$ with its 
    repulsive and attractive fixed points $\theta_*$ and $\thetar$, respectively, marked.}
    \label{fig:random-cluster-drift-function}
\end{figure}

The overall strategy of the proof can be described in the following two step manner. 

\begin{enumerate}
    \item \emph{Escape from a local neighborhood of the unstable fixed point}: We show that in an $O(\sqrt n)$ window around the unstable fixed point $\theta_*n$, the evolution of the size of its giant component, though not literally a Markov process, is well-approximated by a 1-dimensional Gaussian Markov chain (so long as the dynamics maintains a certain typical structure). 
    
    The 1-dimensional Gaussian Markov chain is monotone in the initialization, and for each $\xi$ has a unique initialization $c_\xi$ such that in $O(1)$ time, it exits the unstable fixed point to the right (towards the ordered phase) with probability $1-\xi$, and to the left (towards the disordered phase) with probability $\xi$.
    This behavior transfers to the CM dynamics by the approximation, and by picking the initialization's edge-probability $\lambda_0/n$ so that the initial giant component size has mean $\theta_* n + c_\xi \sqrt{n} + o(\sqrt{n})$. 
    
    \item \emph{Convergence to the stable fixed point}: Once the giant component size is $\omega(\sqrt{n})$ away from $\theta_* n$ (say to the right), its drift towards the stable fixed point $\thetar$, dominates the fluctuations. By a dyadic argument, we show that in a further $O(\log n)$ steps, the CM dynamics gets $\Omega(n)$ away from $\theta_* n$. From there, the landscape is effectively strictly convex and there is a macroscopic drift towards the stable fixed point. At this point, up to small modifications to handle the unlikely event that the dynamics leaves the region of convexity, the fast quasi-equilibration to the ordered phase follows from arguments very similar to those employed in~\cite{BS15}. 
\end{enumerate}
We outline the above two items in more detail in what follows. 

\subsubsection{Escaping the unstable fixed point separating phases}
Typical fast mixing arguments in the analysis of SW and CM chains rely on a uniform drift (dominating the fluctuations) to take them to a local neighborhood of a stable fixed point. In our setting, near the unstable fixed point at $\theta_* n$ for the giant component, the drift and fluctuations are on the same order, and together give constant probabilities to exiting the neighborhood of $\theta_* n$ to the right (and subsequently quasi-equilibrating to the ordered phase), or to the left (and subsequently quasi-equilibrating to the disordered phase). Understanding this competition to pinpoint the right initialization for fast mixing to occur requires pinpointing not just the order, but the variance, and even the exact distribution of the fluctuations. 

Towards this, we show that the evolution of the giant component, which we denote by $|\mathcal L_1(X_t)|$, is well approximated near its unstable fixed point $\theta_* n$, by an explicit 1-dimensional Gaussian Markov chain. To be precise, there exist explicit constants $a, (b^2_{t})_{t\ge 0} > 0$ such that if we define 
\begin{align}\label{eq:Zbar-process-intro}
    \bar Z_{t+1} - \bar Z_{t}= \begin{cases}  a \bar Z_t + \mathcal N(0, b^2_t) & \text{with \ prob., $1/q$} \\ 
    0 & \text{otherwise}
    \end{cases}\,,
\end{align}
then the giant component, centered by $\theta_* n$ and divided by $\sqrt{n}$, tracks this process closely.  I.e., if we let $\bar L_t = n^{-1/2}(|\mathcal L_1(X_t)| - \theta_* n)$, then by Theorem~\ref{thm:1dapproxZ} there exists a coupling $\mathbf{P}$ such that for every $T = O(1)$, 
\begin{align}\label{eq:1D-approximation}
    \mathbf{P} \Big( \sup_{t\le T} \big|\bar L_t - \bar Z_t\big|\ge O\big(\frac{\log n}{\sqrt{n}}\big)\Big) = o(1)\,,
\end{align}
where the initializations are $X_0 \sim G(n, \lambda_0/n)$ for $\lambda_0 = \lambda_* + {c_*}n^{-1/2} + o(n^{-1/2})$, and $\bar Z_0 \sim \mathcal N(0, F(c_*))$ for an explicit increasing function $F$. 

To show this, we first argue that $(X_t)_{t\in [0,T]}$ stays in some good set of configurations $\mathcal G_T$, consisting of having a specific (time-dependent) sum-of-squares of its non-giant component sizes, and a sufficient number of singleton components. Lemma~\ref{lem:Gt} shows that with high probability, $(X_t)_t \in \mathcal G_T$ for $T= O(1)$ times when initialized from a supercritical Erd\H{o}s--R\'enyi, but we emphasize that if these a priori regularity estimates do not hold for the configuration, the approximation by the 1-dimensional chain can fail completely. 
Once on the set $\mathcal G_T$, the approximation of the evolution of the giant by a Gaussian Markov chain boils down to concentration estimates for the set of vertices not activated in the past $k$ steps for each $k$, and local limit theorems both for the activation and percolation steps of the CM dynamics. 

With~\eqref{eq:1D-approximation} in hand, we translate exit probabilities for $\bar Z_t$ to the right and left of an interval $[-\gamma,\gamma]$ with large $\gamma = O(1)$, to exit probabilities on $\bar L_t$ with only $o(1)$ error. A subtle technical point is that, although the variances of $\bar Z_t - \bar Z_{t-1}$'s steps are time-dependent, they are notably neither $n$-dependent, nor $\bar Z_{t-1}$-dependent, so this is a tractable and monotone chain on $\mathbb R$. 

\subsubsection{Mixing within a phase away from the fixed-point}
From the above we deduce that for a well-chosen initialization parameter $\lambda_0$, after $O(1)$ many steps (depending on $\gamma$), $|\mathcal L_1(X_t)|$ has exited $[\theta_* n - \gamma \sqrt{n}, \theta_* n + \gamma \sqrt{n}]$ to the right with probability $1-\xi+o(1)$ and to the left with probability $\xi-o(1)$. The next step is to show that if it exited to the right, with probability $1-O(\gamma^{-2})$, in a further $O(\log n)$ steps, it quasi-equilibrates from there to the ordered phase (the random-cluster measure conditioned on having a giant), and if it exited to the left similarly to the disordered phase. Without loss of generality, let us discuss the first setting of exiting to the right. 

This step can be broken into three parts: 
\begin{enumerate}[(i)]
    \item initially, the drift away from the saddle for $|\mathcal L_1(X_t)|$ is proportional to its distance to $\theta_* n$, and this can be used to get it from $\theta_* n +  \gamma \sqrt{n}$ to $\theta_* n + \Omega(n)$ in $O(\log n)$ steps, except with probability $1-O(\gamma^{-2})$, which covers the possibility it goes back to the saddle in which case quasi-equilibration fails. 
    \item from $\theta_* n + \Omega(n)$, the distance to $\thetar n$ contracts exponentially fast, and in  $O(\log n)$ steps, we have $|\mathcal L_1(X_t) - \thetar n|  = O(\sqrt{n})$; 
    \item any two configurations having $|\mathcal L_1(X_t) - \thetar n|  = O(\sqrt{n})$ and a few other easy-to-maintain properties can be coupled with probability $1-\epsilon$ in $O_\epsilon(\log n)$ further steps; since the ordered phase measure also satisfies these properties, we can couple $X_t$ to a sample from the ordered phase with high probability. 
    \end{enumerate}
    The last two stages here are not so different from those arguments carried out in~\cite{BS15}. However, we emphasize a point of care in all these arguments is that unlike mixing guarantees from worst-case initializations, couplings cannot be restarted arbitrarily if they fail, and a single bad move (which can be correlated with failure of the coupling) could destroy mixing. 

\subsubsection{The Potts Glauber dynamics proof}
At a very high level, the proofs of the Potts Glauber dynamics results in Theorem~\ref{thm:Potts-Glauber} for $\beta \in (\betau,\betas)$ follow a similar strategy to the CM dynamics proof. We first argue a priori that from the initializations considered in the theorem, all color counts except the one dominant color stay within $\tilde O(\sqrt{n})$ of one another for long times: this plays the analogue of the good set of configurations ${\mathcal G}_T$ on which a 1-dimensional chain approximating the evolution of the dominant color class, is sufficient. Within this good set of configurations, we first bound the time to escape the local neighborhood of the unstable fixed point, and then show that whichever side the dynamics exits out of, in $O(n \log n)$ steps, it quasi-equilibrates to the corresponding phase (disordered, or dominated by one of the $q$ colors). 
The technical tools in these proofs,
in addition to those appearing in the CM proofs, are martingale concentration estimates adapted from~\cite{CDLLPS-mean-field-Potts-Glauber}. 
This allows us to roughly treat every $n$ steps of the Potts Glauber updates like a single step of the CM dynamics. 

Some notable further challenges arise in the low-temperature regime of $\beta>\betas$ that are not present in the CM analysis, due to the fact that the saddle at $(1/q,...,1/q)$ proportions vector is unstable in all $q$ directions simultaneously, while when $\beta <\betas$, the saddles we encounter are stable in all but one direction. To handle these difficulties, we obtain a better quantitative understanding of the $q$-dimensional drift landscape (for general $q$), and use refined approximations of the behavior of the Glauber dynamics' proportions vector away from its fixed points by a $q$-dimensional gradient dynamical system. For the reader only interested in the Potts Glauber dynamics, this is found in Section~\ref{sec:Potts-dynamics}, and its presentation is effectively self-contained.

\subsection*{Acknowledgments}
The authors thank anonymous referees for their careful reading and helpful comments. 
The research of R.G.\ is supported in part by NSF grant DMS-2246780. The research of A.B.\ and X.Z.\ is supported in part by NSF grant CCF-2143762.

\section{Preliminaries}
\label{sec:prelim}

In this section, we introduce notations and gather known facts about Erd\H{o}s--R\'enyi random graphs, the mean-field random-cluster measure, and the mean-field Potts measure that we will appeal to in our proofs.

Let us begin with some discussion of the notation used throughout the paper. We will always think of $\beta,q$ fixed, and therefore may drop their dependencies from subscripts. All our results should be thought of as applying for $n$ sufficiently large. When we use big-$O$, little-$o$, and $\Theta$ notation, the implicit constants may depend on $\beta,q$ but not on $n$. 
The threshold $\gamma$ will 
ultimately be taken to be a sufficiently large constant depending on the $\varepsilon$ total variation distance to the stationary measure we are aiming for. Therefore, we mostly treat $\gamma=O(1)$ except when it is important to ensure that large enough $\gamma$ suffices to achieve $\varepsilon$ total variation distance.
We use the notation $\tilde{O}$ to suppress logarithmic factors.

Throughout, for a random-cluster configuration (i.e., a subset of the edges of the complete graph) $X$, we use $\mathcal{L}_i(X)$ to denote the $i$-the largest component (breaking ties according to an arbitrary ordering on components). We write $|\mathcal{L}_i(X)|$ for the number of vertices in  $\mathcal{L}_i(X)$. Furthermore, we define the following: 
\begin{itemize}
    \item  $R_k(X) := \sum_{i\ge 1} \mathcal |\mathcal L_i(X)|^k$ is the sum of $k$'th powers of component sizes; and
    \item  $R_k^- (X) := \sum_{i\ge 2} \mathcal |\mathcal L_i(X)|^k$ is like $R_k(X)$ but excludes the largest component of $X$. 
\end{itemize}

\subsection{Structure of the mean-field Potts and random-cluster equilibrium measures}
We begin by recapping the typical behavior of the mean-field Potts and random-cluster measures, describing their phase transitions as the inverse-temperature parameter $\beta$ varies. As described in the introduction, when $q>2$, 
there are three critical points $\betau<\betac<\betas$ of relevance, given by 
\[
\betau:=\sup\left\{\beta\ge 0: \Big(1+(q-1)e^{\beta\cdot \frac{1-qx}{q-1}}\Big)^{-1} - x\neq 0,~~ \forall x\in (1/q, 1)\right\},
\]
\[
\betac:=\frac{2(q-1)}{q-2} \log (q-1), \qquad\text{and} \qquad \betas:=q.
\]

To characterize the giant component appearing at $\beta \ge \betac$, we define
$\thetar=\thetar(\beta,q)$ as the largest $x>0$ satisfying the equation
\begin{align}\label{eq:thetar-implicit-equation}
e^{\beta x} = \frac{1-x}{1+(q-1)x}.
\end{align}
The following lemma characterizes the mean-field random-cluster distribution at each~$\beta$.
To state the lemma, including the form of coexistence at $\beta = \betac$, we define
    the constant $\xi = \xi(q)$, as 
    \begin{align}\label{eq:xi}
        \xi &:= \frac{1}{1+\xi'}, \qquad \text{where}\\ 
        \xi' &:= \frac{1}{q-1}  \Big( \frac{q- \betac/(q-1)}{q-\betac} \Big)^{(2-q)/2} 
        \exp{\Big( \frac{\betac^2(q-2)(q^2-4q+2)}{4q(q-1)^2} \Big)}. \notag
    \end{align}
\begin{lemma}\label{lem:equilibrium-estimates}
Let $q>2$ and consider the mean-field random-cluster model at fixed $\beta$. 
 Let $X\sim \mu_{\beta}$.
 \begin{itemize}
     \item If $\beta <\betac$, with probability $1-o(1)$, it has 
     \begin{align*}
         |\mathcal L_1(X)| = O(\log n)\,, \qquad \text{and} \qquad R_2(X)  = O(n).
     \end{align*}
     \item If $\beta  = \betac$, then for $\xi$ defined as in \eqref{eq:xi}, with probability $\xi-o(1)$, it has 
     \begin{align*}
         |\mathcal L_1(X)| = O(\log n) \qquad \text{and} \qquad R_2(X) = O(n)\,,
     \end{align*}
     and with probability $(1-\xi) - o(1)$, it has 
     \begin{align*}
         ||\mathcal L_1(X)| -  \thetar n| = o(n)\,, \qquad   |\mathcal L_2(X)| = O(\log n)\,, \qquad   \text{and} \qquad R_2^-(X) = O(n)\,.
     \end{align*}
     \item If $\beta > \betac$, with probability $1-o(1)$, it has 
     \begin{align*}
                  ||\mathcal L_1(X)| -  \thetar n| = o(n)\,, \qquad   |\mathcal L_2(X)| = O(\log n)\,,  \qquad \text{and} \qquad R_2^-(X) = O(n)\,.
     \end{align*}
 \end{itemize}
\end{lemma}

The bounds on $|\mathcal L_1(X)|$ in Lemma~\ref{lem:equilibrium-estimates} appeared in \cite{LL06} and \cite{BGJ-mean-field-random-cluster}, whereas the bounds for $|\mathcal L_2(X)|$ and 
$R_2^-(X)$ follow from 
the analogous bound for sub-critical random graphs (see, e.g.,~\cite{JL08}) and the machinery from \cite{BGJ-mean-field-random-cluster} to transfer such results to the random-cluster model.

To describe the corresponding phase transition for the Potts measure, we let $S(\sigma)$ be its proportions vector, i.e., 
\begin{align*}
    S(\sigma) = \Big(\frac{1}{n} \sum_{v\in [n]}\mathbf{1}\{\sigma(v)=1\},...,\frac{1}{n} \sum_{v\in [n]}\mathbf{1}\{\sigma(v)=q\}\Big)\,.
\end{align*}
We sometimes consider $S$ as close to another vector even if it is only so up to permutation of the $q$ coordinates in the vector; in this case we say ``up to a permutation of the $q$ spins'', noting that the measure, and the law of the Glauber dynamics are invariant under such permutations. In the analysis of the Potts Glauber dynamics, also, we typically consider $S$ permuted so that its largest count is always its first coordinate. 
By Lemma~\ref{lem:equilibrium-estimates} and the coupling between the random-cluster model and the ferromagnetic Potts model (coloring components independently), we obtain the following. 

\begin{corollary}
    \label{cor:equilibrium-estimates-Potts}
     Let $q>2$ and consider the mean-field Potts model at fixed $\beta$. 
     Let $\sigma \sim \pi_{G,\beta}$.
 \begin{itemize}
     \item If $\beta <\betac$, with probability $1-o(1)$, it has 
     \begin{align*}
        \big\|S(\sigma) - (\tfrac{1}{q}, \dots, \tfrac{1}{q} ) \big\|_1 =O(\sqrt{n})\,.
     \end{align*}
     \item If $\beta  = \betac$, then for $\xi$ defined as in \eqref{eq:xi}, with probability $\xi-o(1)$, it has 
     \begin{align*}
        \big\|S(\sigma) - (\tfrac{1}{q}, \dots, \tfrac{1}{q} ) \big\|_1 =O(\sqrt{n})\,.
    \end{align*}
     With probability $(1-\xi) - o(1)$, (up to a permutation of the $q$ spins) it has 
     \begin{align*}
        \big\|S(\sigma) - (\tfrac{\thetar(q-1) + 1}{q},\tfrac{1-\thetar }{q},  \dots, \tfrac{1-\thetar }{q} ) \big\|_1 =O(\sqrt{n})\,.
     \end{align*}
     \item If $\beta > \betac$, with probability $1-o(1)$, (up to a permutation of the $q$ spins) it has 
     \begin{align*}
        \big\|S(\sigma) - (\tfrac{\thetar(q-1) + 1}{q}, \tfrac{1-\thetar }{q}, \dots, \tfrac{1-\thetar }{q} ) \big\|_1 =O(\sqrt{n})\,.
     \end{align*}
 \end{itemize}
\end{corollary}

\subsection{Random graph preliminaries}
Due to the percolation step in the definition of the CM dynamics, precise random graph estimates are essential to careful understanding of the CM dynamics. We begin with central limit and local limit theorems for the size of the giant in a super-critical Erd\H{o}s--Renyi random graph. 
(Recall that a $G(n,\lambda/n)$ random graph is said to be \emph{sub-critical} when $\lambda < 1$ and is called \emph{super-critical} when $\lambda >1$.)

\subsubsection{Limit theorems for the size of the giant component}
The expected size of the giant component in a $G(n,\lambda/n)$ random graph is roughly $\alpha(\lambda) n$ where
\begin{equation}
    \label{eq:beta}
   \alpha(\lambda) := \begin{cases} \max\{x>0:  e^{-\lambda x} = 1-x\} & \lambda>1 \\ 0 & \lambda \le 1 \end{cases}\,.
\end{equation}
In particular, Theorem 5 of~\cite{CMS04} showed the following: if $\lambda>1$ uniformly in $n$, and $G\sim G(n,\lambda/n)$, then 
\begin{equation}
    \label{eq:expectedgiantgnp}
        |\mathbb E[|\mathcal L_1(G)|] - \alpha(\lambda)n|\le \tilde{O}(1)\,.
\end{equation}

It is easy to check that the $e^{-\lambda x} - 1+x$ has strictly positive $x$-derivative at $(\lambda,\alpha(\lambda))$ as long as $\lambda >1$, and moreover it is analytic in $x$. By the analytic implicit function theorem, this implies $\alpha(\lambda)$ is analytic in $\lambda$ for all $\lambda>1$.  
Moreover, the size is concentrated about $\alpha(\lambda) n$ with variance approximately $\sigma^2 (\lambda) n$, where 
\begin{equation}
    \label{eq:sigmag}
    \sigma^2(\lambda) := \frac{\alpha(\lambda)(1-\alpha(\lambda))}{\left( 1 - \lambda(1-\alpha(\lambda)) \right)^2}.
\end{equation}
In fact, it is known to satisfy the following central and local limit theorems.

\begin{theorem}
[\cite{PW05}]
    \label{thm:giantLLT}
    Let $G(n,\lambda/n)$ with $\lambda>1$. Let $\alpha = \alpha(\lambda)$ and $\sigma =\sigma(\lambda)$.
    For any compact interval $J\subset \mathbbm{R}$ and any $\delta > 0$, for all large $n$ and any integer $k\in \mathbbm{N}$ satisfying $\frac{k-\alpha n}{\sigma \sqrt{n}} \in J$, we have
    \begin{equation}
        \frac{1-\delta}{\sigma\sqrt{2\pi n}} \exp{\left( -\frac{(k-\alpha n)^2}{2\sigma^2 n} \right)} 
        \le \Pr(|\mathcal L_1(G)| = k) 
        \le \frac{1+\delta}{\sigma\sqrt{2\pi n}} \exp{\left( -\frac{(k-\alpha n)^2}{2\sigma^2 n} \right)}.
    \end{equation}
\end{theorem}

\begin{theorem}[\cite{BR12}]
\label{thm:CLT}
 Let $G(n,\lambda/n)$ with a fixed $\lambda = O(1)$ with $(\lambda-1)^3 n \to \infty$ as $n\to\infty$.
    Let $|\mathcal L_1(G)|$ denote the number of vertices in the largest component of $~G\sim G(n,\lambda/n)$.
    We have as $n\to\infty$, that 
    \begin{equation}
        \frac{|\mathcal L_1(G)| - \alpha(\lambda) n}{\sigma(\lambda)\sqrt{n}} \xrightarrow[D]{} \mathcal N(0,1),
    \end{equation}
    where $\xrightarrow[D]{}$ denotes convergence in distribution, and $\mathcal N(0,1)$ is a standard normal.
\end{theorem}

\subsubsection{Discrete duality}
We also want sharp understanding of other statistics of the component counts for $G(n,\lambda/n)$ random graphs in different regimes of $\lambda$. A well-known tool to translate bounds between sub-critical and super-critical $\lambda$  is the discrete duality principle.  We start with an observation using the definition of total-variation distance. 

\begin{observation}
    \label{obs:tv-distance-conditioning}
    If $\mu(A) \ge 1-\epsilon$, and $\nu = \mu (\cdot \mid A)$, then $\|\mu - \nu\|_{\TV} \le \frac{2\epsilon}{1-\epsilon}$.
\end{observation}

\begin{lemma}\label{lemma:duality}
    Suppose $\lambda>1$ uniformly in $n$ and consider $X\sim G(n,{\lambda}/{n})$. For any $U: \left| |U|- \alpha(\lambda) n \right| = o(n)$, conditional on $\{\mathcal L_1(X)=U\}$, the law of the graph induced by $X$ on $U^c$ has total-variation distance at most $e^{- \Omega(n)}$ to $G(n-|U|,\lambda/n)$.
\end{lemma}
\begin{proof}
    Reveal the edge-set on $U$. Conditioning on $\mathcal L_1(X) = U$ is equivalent to conditioning on all edges between $U$ and $U^c$ being absent, the vertices of $U$ being connected within $E(U)$, and the event $\mathcal E_{|U|}$ that the subgraph induced on $U^c$ has no component larger $|U|$.   
    The first two parts are measurable with respect to edges in $E\setminus E(U^c)$, so the law induced on $E(U^c)$ is exactly that of $G(n-|U|,\lambda/n)$ conditioned on $\mathcal E_{|U|}$.

    By a standard calculation, if $|U| = \alpha(\lambda) n +o(n)$, then $\frac{n-|U|}{n} \lambda <1$ is uniformly bounded away from $1$, so that the resulting graph is uniformly sub-critical and the probability of $\mathcal E_{|U|}$ 
    is $1-\exp(-\Omega(n))$. The result follows from Observation~\ref{obs:tv-distance-conditioning}.  
\end{proof}

\subsubsection{Refined random graph statistics}
We now summarize estimates for $R_2, R_3$ and $R_2^-$ and $R_3^-$ for sub and super-critical random graphs. 

\begin{lemma}
    [\cite{JL08}]
    \label{lemma:Skmoment}
    Let $G\sim G(n,\lambda/n)$.
    For $n\ge 1$ and $\lambda < 1$, 
    \begin{align}
        \label{eq:ER2R3}
        \E [R_2(G)] &= \frac{n}{1-\lambda} \cdot \Big( 1 + O( n^{-1}(1-\lambda)^{-3}) \Big)\,,
        \\
        \E [R_3(G)] &= \frac{n}{(1-\lambda)^3} \cdot\Big( 1 + O( n^{-1}(1-\lambda)^{-3}) \Big)\,, \notag
    \end{align}
    and 
    \begin{equation}
        \var (R_2(G)) = O\Big( \frac{n}{(1-\lambda)^{5}} \Big), \qquad         \var (R_3(G)) = O\Big( \frac{n}{(1-\lambda)^{9}} \Big).
    \end{equation}
\end{lemma}

The following is a list of typical structural properties of sub-critical random graphs that we deduce. In what follows, $I_1(X):=|\{i: |\mathcal{L}_i(X)|=1 \}|$ denotes the number of isolated vertices in $X$.

\begin{lemma}
    \label{lem:subcriticalgstructure}
    Let $G\sim G(n,\lambda/n)$ where $\lambda<1$.
    Then $G$ satisfies
    \begin{enumerate}
        \item  $   \left| R_2(G) - \E[R_2(G)]  \right| \le  \sqrt{n} \log^2 n$,
        with probability $1-O((\log n)^{-4})$;
        \item $I_1(G) = \Omega(n)$  with probability $1-O(n^{-1})$;
        \item $|\mathcal{L}_1(G)| = O(\log n)$  with probability $1-O(n^{-1})$;
        \item $R_3(G)  = O(n)$, $R_2(G) = O(n)$ with probability $1-O(n^{-1})$.
    \end{enumerate}
\end{lemma}
\begin{proof}
    Items (1) and (4) follow from Chebyshev's inequality and Lemma~\ref{lemma:Skmoment};
    items (2) and (3) can be found in standard literature of random graphs (see, e.g.,  Lemma 5.7 in \cite{LNNP11} and Lemma 7 in \cite{CF00}). 
\end{proof}

Given Lemma~\ref{lemma:duality}, we can also deduce the
analogous properties for super-critical random graphs. 

\begin{corollary}
    \label{cor:R2moment}
    Let $G\sim G(n,\lambda/n)$, where $\lambda > 1$ uniformly in $n$. Then $\mathbb E[R_2^-(G)]$, $\mathbb E[R_3^-(G)]$, $\var(R_2^-(G))$ and $\var(R_3^-(G))$ are all $O(n)$. 
\end{corollary}

Finally, the following concerns stability of the size of the giant under perturbations of the vertex count. 

\begin{lemma} 
    [Lemmas 2.7 \& 2.8 in \cite{ABthesis}]\label{lem:rgdeviation}
    Let $G_{d_n}$ be distributed as a $G(n+m, d_n/n)$ random graph where $|m|=o(n)$ and $\lim_{n\rightarrow \infty} d_n = d$. Assume $1<d_n$ and $d_n$ is bounded away from $1$ for all $n$. Then,
    \begin{enumerate}
        \item $|\mathcal{L}_2(G_{d_n})| = O(\log n)$ with probability $1-O(n^{-1})$, 
        \item $\var(|\mathcal{L}_1(G_{d_n})|) = \Theta(n)$ and 
        \item For $A=o(\log n)$ and sufficiently large $n$, there exists a constant $c$ such that
        \[
            \Pr(||\mathcal{L}_1(G_{d_n})| - \alpha(d)n| > |m| + A\sqrt{n}) \le e^{-cA^2}.
        \]
    \end{enumerate}
\end{lemma}

We arrive at the following by combining Lemmas~\ref{lemma:duality},~\ref{lem:subcriticalgstructure}, and~\ref{lem:rgdeviation}.

\begin{lemma}
\label{lem:supercriticalgstructure}
    Let $G\sim G(n,\lambda/n)$ where $\lambda>1$.
    Then with high probability, $G$ satisfies
    \begin{enumerate}
        \item  $   \left| R_2^-(G) - \E[R_2^-(G)]  \right| \le  \sqrt{n} \log^2 n$,   with probability $1-O((\log n)^{-4})$;
        \item $I_1(G) = \Omega(n)$ with probability $1-O(n^{-1})$;
        \item $R_3^-(G)  = O(n)$, $R_2^-(G) = O(n)$ with probability $1-O(n^{-1})$.
    \end{enumerate}
\end{lemma}

\subsection{Drift of the giant component in the CM dynamics}\label{subsec:CM-drift-function}
We end this preliminaries section by describing properties of the drift function for the size of the giant component in the CM dynamics when $\beta \in (\betau,\betas)$. This function will govern the expected change to the (fractional) size of the giant component in the CM dynamics on the event that the giant is activated. 
To start, for $\theta\in [0,1]$, let $\ka(\theta)$ be the expected fraction of activated vertices if a giant of fractional size $\theta$ is activated, i.e., 
$$
\ka(\theta) = \theta + \frac{1}{q} (1-\theta)\,, \qquad \text{and let} \qquad   \phi(\theta):=\alpha(\beta \ka(\theta)) \cdot \ka(\theta)\,,
$$
so that $\phi(\theta)$ is the expected fractional size of the giant component from a configuration with a giant of fractional size $\theta$.
Define the drift function
$$f(\theta) :=\phi(\theta) - \theta\,.$$
Note that $\alpha(\lambda)$ represents the expected fraction of the giant component, as  characterized in \eqref{eq:beta}. When $\lambda = \beta \ka(\theta)$, the regime $\lambda > 1$ corresponds to $\theta\in (\thetas,1]$, where $\thetas= \inf\{ \theta: \beta k_a(\theta) > 1\} = \frac{q-\beta}{\beta(q-1)}$.

We compile a set of useful facts about $f$, and refer to Figure~\ref{fig:random-cluster-drift-function} for a visual aid.
\begin{lemma} [\cite{ABthesis}]
    \label{lem:factsaboutdrift}
    For every $\beta>0$, the following properties hold for $f$:
    \begin{enumerate}
        \item The function $f$ is continuous, differentiable and strictly concave in $(\thetas, 1]$.
        \item The function $\phi$ is continuous, differentiable, and strictly increasing in $(\thetas, 1]$.
        \item Let $f(\thetas^+):=\lim_{\theta\rightarrow\thetas +} f(\theta)$. Then $sgn(f(\thetas^+)) = sgn(\beta -q)$.
        \item If $\beta \in (\betau,\betas)$, then $f$ has exactly 2 roots in $(\thetas, 1]$; $\theta_*$ is the first root and $\thetar$ is the second root.
    \end{enumerate}
\end{lemma}
Notice that $\thetar$ defined in this way, matches the solution to~\eqref{eq:thetar-implicit-equation} when $\beta \ge \betac$.

The next lemma is a minor extension of Lemma 3.9 in \cite{ABthesis}.

\begin{lemma}
    \label{lem:phifunctionsupercritical}
    Let $\beta\in (\betau,\betas)$ and $q > 2$.
    \begin{enumerate}
        \item For $\theta\in (\theta_*,\thetar)$, we have $\theta\le \phi(\theta) \le \thetar$.
        \item For a fixed $\vartheta\in (\theta_*,\thetar]$, there exists $\delta = \delta(\vartheta) \in (0,1)$ such that 
        $\delta |\theta - \thetar| \le |\phi(\theta) - \theta|$ for all $\theta\in [\vartheta, 1]$.
    \end{enumerate}
\end{lemma}
\begin{proof}
    First, let $\theta_*< \theta < \thetar$. By Lemma~\ref{lem:factsaboutdrift} (2) and (4) we have
    $
        \phi(\theta) < \phi(\thetar) = \thetar
    $.
    Moreover, by Lemma~\ref{lem:factsaboutdrift} (1) and (4), 
    $f$ is strictly concave and $f(\thetar) = 0 = f(\theta_*)$.
    Hence $f(\theta) > 0$, or equivalently $\theta < \phi(\theta)$.
    This establish the first part, and in particular, shows that $f(\vartheta) > 0$ for a fixed $\vartheta \in (\theta_*, \theta_r)$.
    The second part follows from the fact $f(\vartheta) >0$ and the concavity of $f$ as in the proof in Lemma 3.9 in \cite{ABthesis} without modification. 
\end{proof}

We will also use the following lemma, which extends Lemma 3.7 in \cite{ABthesis}.

\begin{lemma}
    \label{lem:phifunctionsubcritical}
    Let $s\in (0,1)$ be a fixed constant.
    If $\beta \in (\betau,\betas)$, then for all $\theta\in (\thetas, \theta_*-s]$ there exists a $\delta >0$ such that $f(\theta) \le -\delta$.
 \end{lemma}

\begin{proof}
    By Lemma~\ref{lem:factsaboutdrift}, we have $f(\thetas^+) < 0$ and $f(\theta) < 0$ for $\theta \in (\thetas, \theta_* - s]$. 
    Hence, $f$ as a continuous function must attain a negative maximum on $ (\thetas, \theta_* - s]$ and the lemma follows.
\end{proof}

\section{Mixing from product initializations for the CM dynamics}\label{sec:CM-dynamics}
Our aim in this section is to establish the mixing time results of Theorem~\ref{thm:SWCM-dynamics}. Recall that $f$ is the drift function for the giant component in the CM dynamics, i.e., if a configuration $X_0$ has a giant component of size $\theta n$, on the event of activation of $X_0$, the expected size of the giant in $X_1$ is roughly $(\theta + f(\theta))n$. For $\beta\in (\betau,\betas)$, this function has an unstable fixed point $\theta_*$, i.e., $f(\theta_*) = 0$ and $f'(\theta_*) >0$.    

Section~\ref{subsec:mixing-away-from-saddle-SWCM} will focus on quasi-equilibration to the  measure constrained to the ordered (resp., disordered) phase when initialized with a giant that is $\omega(n^{-1/2})$ away from $\theta_*$. This is the core of all the results of Theorem~\ref{thm:SWCM-dynamics} except the $\beta = \betac$ case where we need to study the small-drift diffusion in the $O(n^{-1/2})$ neighborhood of $\theta_*$ to pick the ordered vs.\ disordered phases with the correct relative probabilities; this latter step is done in Section~\ref{subsec:escaping-saddle-SWCM}. Finally, the results are combined to prove Theorem~\ref{thm:SWCM-dynamics} in Section~\ref{subsec:proof-of-main-SWCM}.

\subsection{Quasi-equilibrating away from the unstable fixed point}\label{subsec:mixing-away-from-saddle-SWCM}

Our main goal for this subsection is to show that initialized from a random graph with 
a giant component of fractional size $\omega(n^{-1/2})$ away from $\theta_*$, and a ``reasonable'' structure on the complement of the giant, the CM dynamics quasi-equilibrates to the random-cluster measure conditioned on the phase corresponding to the side of $\theta_*$ it initializes. To be more precise, for fixed $\beta\in (\betau,\betas)$, define 
\begin{align*}
    \mu^\ord = \mu_\beta( \cdot \mid \Omega^{\ord})\,, \qquad \text{and} \qquad  \mu^\dis = \mu_\beta(\cdot \mid \Omega^\dis)\,,
\end{align*}
where 
\begin{align*}
    \Omega^\ord = \{A : |\mathcal L_1(A)| \ge \theta_* n \}\,, \qquad \text{and} \qquad \Omega^\dis = \{A: |\mathcal L_1(A)| < \theta_* n \}\,.
\end{align*}
Note that $\mu^\ord$ has $|\mathcal L_1(A)|$ concentrated around $\thetar>\theta_*$, where $\thetar$ is the (stable) fixed point of $f$ to the right of $\theta_*$. On the other hand $\mu^\dis$ has $|\mathcal L_1(A)|$ that is concentrated around $0$. 

Throughout the paper, we use $\Pr_{X_0}(\cdot)$ for the law of the Markov chain initialized at $X_0$ and $\Pr_\nu$ when it is initialized from the distribution $\nu$. The main result of this subsection is the following: 
\begin{lemma}\label{lem:quasi-equilibration-away-from-saddle-SWCM}
    Suppose that $q> 2$, $\gamma >0$ and $\beta \in (\betau, \betas)$. If $X_0$ is such that $|\mathcal L_1(X_0)| \ge \theta_* n + \gamma \sqrt{n}$, $|\mathcal L_2(X_0)| = O(\log n)$, and $R_2^-(X_0) = O(n)$, then there exists $C>0$ such that if $T= C\log n$, we have 
    \begin{align*}
        \|\Pr_{X_0} (X_T \in \cdot) - \mu^{\ord}\|_{\TV} = O(\gamma^{-2})\,.
    \end{align*}
    An analogous statement holds w.r.t.\ $\mu^\dis$ if ~$|\mathcal L_1(X_0)|\le \theta_* n - \gamma \sqrt{n}$. 
\end{lemma}

The main part of the proof of Lemma~\ref{lem:quasi-equilibration-away-from-saddle-SWCM} is showing that any initialization whose giant is $\omega(n^{1/2})$ away from the unstable critical point $\theta_* n$ gets to $\Omega(n)$ away from it in $O(\log n)$ steps. Let us introduce a few more notational shorthands for this section: we let 
\begin{itemize}
    \item $L_t:=|\mathcal{L}_1(X_t)|$ be the process tracking the giant component size of $X_t$; 
    \item $\Lambda_t$ be the event that $\mathcal{L}_1(X_{t-1})$ is activated in the $t$'th step.
    \item $A_t$ be the number of activated vertices at the $t$'th step, and let $A_t^- = A_t - L_t \mathbf{1}\{\Lambda_t\}$.
\end{itemize}

\subsubsection*{Quasi-equilibrating to the ordered phase}
We start with the case where we are equilibrating towards the ordered phase.

\begin{lemma}
    \label{lem:hittimesupercritical2}
    Let $q>2$, $\gamma>0$ and $\beta \in (\betau,\betas)$.
    Let $X_0$ be any configuration such that $L_0 \ge \theta_*n + \gamma\sqrt{n}$, 
    $|\mathcal{L}_2(X_0)| = O(\log n)$ 
    and $R_2^-(X_0) = O(n)$. Let $T_s$ be the first $t$ such that $L_t \ge (\theta_* + s)n$, $|\mathcal{L}_2({X}_t)| = O(\log n)$, 
    and $R_2^-({X}_t) = O(n)$ all hold. Then, there exist constants $s\in(0,\theta_r-\theta_*)$ and $C >0$ such that $T_s \le C \log n$
    with probability $1 - O(\gamma^{-2})$.
    
\end{lemma}
To prove Lemma~\ref{lem:hittimesupercritical2}, we use the following lemma to estimate the drift after 1 step.
\begin{lemma}
    [Lemma 3.19 and (3.11) of its proof, \cite{ABthesis}] 
    \label{lem:L1decayexpectation}
    Suppose $0<\beta <\betas$ and that $X_t$ has at most one large component whose size is at least $2n^{11/12}$. Let $\epsilon>0$ be a small constant.
    If $L_t/n \ge \thetas +\epsilon$, then 
    \[
     L_t + f(L_t/n)n - 3n^{1/4} \le 
    \E[L_{t+1} \mid X_t, \Lambda_{t+1}] \le L_t + f(L_t/n)n + 3n^{1/4},
    \]
    and 
    \[
        \frac{f(L_t/n) n}{q} - 2n^{1/4} \le \E[L_{t+1} - L_t \mid X_t] \le \frac{f(L_t/n)n}{q} + 2n^{1/4}.
    \]
    If $L_t/n\in (\thetas - \epsilon, \thetas +\epsilon)$, then
    \[
        \E[L_{t+1} - L_t \mid X_t] \le \frac{f(\thetas + \epsilon)n}{q} +\frac{2\epsilon n}{q} + 2n^{1/4}.
    \]
\end{lemma}

We will now move on to prove Lemma~\ref{lem:hittimesupercritical2}.
\begin{proof}[Proof of Lemma~\ref{lem:hittimesupercritical2}] Let $\tau$ be the first time $t$ when $L_t \ge (\theta_* +s) n$ where $s>0$ is a small constant that will be decided. Let   $(t_i)_{i\ge 0}$ be the subset of times in $[T]= \{1,...,T\}$ at which the largest component is activated. We can condition on this sequence, generating $\sigma$-algebra $\mathscr{T}$, and notice that this can be generated by a sequence of $T$ independent $\text{Ber}(1/q)$ random variables which are also independent of the remaining randomness of the dynamics (say by reserving the first activation coin to always be used for the largest component). 

Let $c_* = f'(\theta_*)/2 >0$, and let $M$, $M_0$ and $M_1$ be sufficiently large constants. 
Consider the event: 
$$
\mathcal A_k = \big\{
                (L_k - \theta_* n) \ge \Big(1+ \frac{c_*}{4}\Big)^{\mathbf 1\{\Lambda_k\}} (L_{k-1} -\theta_* n), R_2^-( X_k) \le M n, |\mathcal{L}_2(X_k)| \le M_1 \log n
            \big\}.
$$
Our goal in this proof is to show the following by induction on $t$: 
\begin{align}\label{eq:wts-induction-escaping-saddle}
    \Pr \bigg(
        \bigcap_{k=1}^{t\wedge \tau} 
            \mathcal A_k \mid \mathscr{T} 
        \bigg) \ge \prod_{i: t_i \le t\wedge \tau} \Big(1- \frac{16 M_0^2}{\gamma^2(1+\frac{c_*}{4})^{2i-2} c_*^2}\Big)
    \cdot \Big( 1- O(T/n) \Big) \,.
\end{align}
In words, \eqref{eq:wts-induction-escaping-saddle} shows that with the stated probability, $L_k$ drifts away from $\theta_* n$ while preserving the control on $R_2^-(X_k)$ and $ |\mathcal{L}_2(X_k)|$.
Let us first conclude the proof assuming~\eqref{eq:wts-induction-escaping-saddle}. 
For any given $s>0$,
if $T=C \log n$ with sufficiently large $C >0$, then standard binomial tail bounds yield that $|\{t_i: t_i\le T\}| \ge T/2q$
with probability at least $1 - n^{-1}$, and 
\[
    \Big(1+\frac{c_*}{4}\Big)^{|\{t_i: t_i\le T\}|} (L_0 - \theta_* n) \ge 
    \Big(1+\frac{c_*}{4}\Big)^{(C/2q) \log n} \gamma \sqrt{n} \ge 
    sn\,.
\]
Thus, on any such realization of $\mathscr{T}$ where $|\{t_i: t_i\le T\}| \ge T/2q$, the event in~\eqref{eq:wts-induction-escaping-saddle} implies 
    $\{\tau\le T\} \cap \{R_2^-(X_\tau) \le M n\} \cap \{ |\mathcal{L}_2(X_\tau)| \le M_1 \log n\}$. 
As such, we can conclude that  
\begin{align*}
     \Pr\big(\{\tau\le T\} &\cap \{R_2^-(X_\tau) \le M n\} \cap \{ |\mathcal{L}_2(X_\tau)| \le M_1 \log n\} \big) \\
     &\ge 
     \Big[
         \prod_{i=1}^{T} 
            \Big(
                1- \frac{16 M_0^2}{\gamma^2c_*^2} \frac{1}{(1+\frac{c_*}{4})^{2i-2} }
            \Big)  
    \Big] \cdot 
    \Big[ 1- O\big(\frac{T}{n}\big) \Big] - \frac{1}{n} \ge 1- \frac{A}{\gamma^2},
\end{align*}
for some $A(c_*, M_0, \gamma, s)>0$.

We now proceed to prove~\eqref{eq:wts-induction-escaping-saddle} inductively. The event in the probability holds at initialization deterministically; now suppose~\eqref{eq:wts-induction-escaping-saddle} holds at time $t$ and let us show it holds at $t+1$. If $\tau = t$, then the event in the probability is unchanged at $t+1$, and the right-hand side is at least its value at $t$. Now suppose that $t<\tau$, and fix $(X_k)_{k\le t}$ such that the events in~\eqref{eq:wts-induction-escaping-saddle} hold for it.

If $t+1\notin \{t_i\}$ and $\Lambda_{t+1}^c$ occurs, then $L_{t+1} \ge L_{t}$ deterministically.
In this case, it suffices to consider only the probability of $\{R_2^-(X_{t+1}) \le Mn,  |\mathcal{L}_2(X_{t+1})| \le M_1 \log n\}$.
On $\Lambda_{t+1}^c$, the decrease in $R_2^-$ as a result of the dissolution of active components is $R_2^-(X_t)/q$ in expectation, and at least $R_2^-(X_t)/q - o(n)$ with probability $1- O(n^{-1})$ by Hoeffding's inequality and the fact that  $R_2^-( X_t) \le Mn$ and $|\mathcal{L}_2(X_{t})| \le M_1 \log n$.
Moreover, 
with probability $1-O(n^{-1})$, $A_{t+1}\in [\frac{n-L_t}{q} - \sqrt{n \log n}, \frac{n-L_t}{q} + \sqrt{n \log n} ]$. 
Thus, $ A_t \cdot \frac{\beta}{n} < 1$ is bounded away from $1$ uniformly in $n$.
Lemma~\ref{lem:subcriticalgstructure} implies that the increase in $R_2^-$ as a result of the creation of new components in the sub-critical percolation step is at most $C_1n$ with probability $1-O(n^{-1})$, for some constant $C_1>0$;
besides, with probability  $1-O(n^{-1})$,
the sizes of new components in the sub-critical percolation step is at most $M_1 \log n$. 
By a union bound,
\begin{align*}
    \Pr \big(  &  
        |\mathcal{L}_2(X_{t+1})| \le M_1 \log n, 
        R_2^-(X_{t+1}) \le     R_2^-(X_{t}) \big(1- \tfrac{1}{q} \big) + C_1 n + o(n) 
        \mid X_t, \Lambda_{t+1}^c
    \big) \\ &= 1 - O(n^{-1}).
\end{align*}
When $ R_2^-(X_{t}) > 8C_1 q n$, we have $ R_2^-(X_{t+1}) \le  R_2^-(X_{t})$ with probability $1-O(n^{-1})$ and when  $ R_2^-(X_{t}) \le 8C_1 q n$, we have $ R_2^-(X_{t+1}) \le (8q+1) C_1 n$ with probability $1-O(n^{-1})$.
Therefore, in both cases we have $ R_2^-(X_{t+1}) \le Mn$ by setting $M\ge (8q+1) C_1$.

Now suppose $t+1 = t_i$ for some $i\ge 1$ so that the events $\Lambda_{t+1}$, 
$R_2^-(X_t) \le Mn$, $|\mathcal{L}_2(X_{t})| \le M_1 \log n$ and 
\begin{equation}
    \label{eq:ihshift}
L_{t} - \theta_*n \ge \big[ 1 + \tfrac{c_*}{4}\big]^{i-1}(L_0 - \theta_* n)\,, 
\end{equation}
occur.
Following a similar argument, this time the activated component being with high probability supercritical, we derive that 
\begin{equation}
    \label{eq:R2Active}
    \Pr(\{R_2^-(X_{t+1}) \le M n, |\mathcal{L}_2(X_{t+1})| \le M_1 \log n\} \mid X_t, \Lambda_{t+1})= 1 - O(n^{-1}).
\end{equation}
    Next, we focus on the change of $L_{t+1}$ on $\Lambda_{t+1}$.
    By Taylor expansion, for small enough $s>0$ we obtain that for $\Theta \in [\theta_* + \frac{\gamma}{\sqrt{n}}, \theta_* + s]$,
    \begin{equation}
        \label{eq:Taylorf}
        f(\Theta) =  f(\theta_*)+ f'(\theta_*) (\Theta-\theta_*) + O((\Theta-\theta_*)^2) \ge  c_* \cdot (\Theta - \theta_*).  
    \end{equation}
    By Lemma~\ref{lem:L1decayexpectation}, we have
    \begin{equation}
    \label{eq:expgrowthEt}
               \E[L_{t+1} \mid X_t, \Lambda_{t+1}] \ge  L_t + f(L_t/n) n - 3n^{1/4}.
    \end{equation}
    Set $s_t = f(\frac{L_t}{n}) \cdot \frac{n}{4}$.
    For $L_t > \theta_* n+ \gamma\sqrt{n}$, 
        we have $f(\frac{L_t}{n}) \ge  \frac{c_* \gamma}{\sqrt{n}}$, so for large enough $n$,
    \begin{equation}
    \label{eq:defc*}
        f\big(\frac{L_t}{n}\big)n - 3n^\frac{1}{4} - s_t \ge f\big(\frac{L_t}{n}\big) \cdot \frac{n}{2}.
    \end{equation}
    Then by \eqref{eq:expgrowthEt}, \eqref{eq:defc*} and Chebyshev's inequality, we obtain that
    \begin{align*}
        \Pr\big(L_{t+1} \le L_t + f(L_t/n) \cdot \frac{n}{2} \mid X_t, \Lambda_{t+1} \big)
    &\le  \Pr\big( \big|L_{t+1} - \E[L_{t+1} \mid X_t, \Lambda_{t+1}] \big| \ge s_t \mid X_t, \Lambda_{t+1} \big) \\
    &\le \frac{\var(L_{t+1} \mid X_t, \Lambda_{t+1})}{s^2_t} 
    = \frac{16 M_0^2 n}{f(L_t/n)^2 n^2},
\end{align*}
where the last equality follows from the Fact~\ref{lem:uniformvarbound} that there exist constants $M_0 > 0$ and $s\in (0,1)$ such that if $ \theta_*n - sn\le L_t \le \theta_*n + sn$, then
    \[
     \var(L_{t+1} \mid X_t, \Lambda_{t+1}) \le M_0^2 n\,.
    \]
By \eqref{eq:ihshift} and \eqref{eq:Taylorf}, we get 
\[
    \frac{16 M_0^2 n}{f(L_t/n)^2 n^2}
    \le \frac{16 M_0^2 n}{c_*^2 \cdot (L_t - \theta_* n)^2}
    \le \frac{16 M_0^2}{c_*^2 \gamma^2 (1+c_*/4)^{2(i-1)}}.
\]
Hence, by combining the inequalities above, 
\begin{align*}
    \Pr\big( L_{t+1} - \theta_* n \ge (1 + \frac{c_*}{4}) & \cdot (L_t - \theta_* n) \mid X_t, \Lambda_{t+1} \big) \\
    &\ge \Pr\big(L_{t+1} \ge L_t + f(L_t/n) \cdot \frac{n}{2} \mid X_t, \Lambda_{t+1} \big)\\
    &\ge 1 - \frac{16M_0^2}{c_*^2 \gamma^2 (1+c_*/4)^{2(i-1)}}.
\end{align*}
By a union bound, this inequality and \eqref{eq:R2Active} conclude the induction step of \eqref{eq:wts-induction-escaping-saddle} for the case of $\Lambda_{t+1}$. 
\end{proof}

We now proceed to show that once the giant's size is macroscopically away from the unstable fixed point at $\theta_*$, in a further $O(\log n)$ steps, it equilibrates (to the corresponding phase) quickly. This part of the proof follows closely those of~\cite{BS15} for $\beta>\betas$, with a little care due to the rare event that the dynamics crosses to the other side of the unstable fixed point. The argument goes in two stages, the first getting the giant to within $O(\sqrt{n})$ of the stable fixed point at $\thetar$, and the second quasi-equilibrating from there. 

\begin{lemma}
    \label{lem:hittingtimecritical3}
    Let $q > 2$ and $\beta > \betau$. 
    Let $s \in(0,\theta_r-\theta_*)$ be a fixed constant.
    Suppose ${X}_0$ satisfies that $L_0 \ge (\theta_* + s)n$, 
    $|\mathcal{L}_2({X}_0)| = O(\log n)$, and
    $R_2^-({X}_0) = O(n)$.
    Then, for any $\varepsilon>0$ there exists $T=O(\log n)$ such that ${X}_T$ satisfies all of the following properties with probability $1 - \varepsilon$:
    \begin{enumerate}
        \item $| L_T - \thetar n | = O(\sqrt{n}$);
        \item $I_1({X}_T) = \Omega(n)$;
        \item $|\mathcal{L}_2({X}_T)| = O(\log n)$;
        \item $R_2^-({X}_T) = O(n)$.
    \end{enumerate}
\end{lemma}
\begin{proof}
This lemma was established in \cite{ABthesis} (see Lemma 3.26 there) for the case when $\beta \ge q$; the same argument can be carried over to the more general setting where $\beta > \betau$ with only minor modifications, as we detailed next.
Let $\Delta_t = |L_t -\thetar n|$. 
For $\beta \ge q$
and a configuration $X_t$ 
such that $L_t \ge (\theta_* + s)n$, $|\mathcal{L}_2({X}_t)| = O(\log n)$ and $R_2^-({X}_t) = O(n)$,
it was shown in~\cite[Eq.~(3.22)]{ABthesis} that 
   \begin{align}
   \label{eq:delta-bound}
    \E[\Delta_{t+1} \mid X_t] \le \big(1-\frac{1}{q} \big) \Delta_t + \frac{|\thetar - \phi(L_t/n)|}{q} n + O(\sqrt{n}).
    \end{align}
It can be readily checked that the same inequality holds under the weaker condition that $q>2$ and $\beta > \betau$. Specifically, Fact 3.28 from~\cite{ABthesis} holds in this setting since under the assumption that $L_t \ge (\theta_* + s)n$, the percolation step of the CM dynamics is subcritical (resp. supercritical) when the largest component of $X_t$ is inactive (resp., active), and this is essentially all that is required to establish \eqref{eq:delta-bound}. 

Lemma~\ref{lem:phifunctionsupercritical} then implies that there exists a constant $\delta\in (0,1)$ such that 
    \[
    |\thetar - \phi(L_t/n)| = \frac{1}{n} |\thetar n - L_t| - \frac{1}{n} |L_t - \phi(L_t/n) n| \le (1-\delta)  |\thetar - (L_t/n)|;
    \]
(note that Lemma~\ref{lem:phifunctionsupercritical} extends Lemma 3.9 from~\cite{ABthesis} to the $\beta > \betau$ regime). Plugging this bound into~\eqref{eq:delta-bound}, we obtain
   \begin{align}
   \label{eq:delta-bound-1}
    \E[\Delta_{t+1} \mid X_t] \le \big(1-\tfrac{\delta}{q} \big) \Delta_t + O(\sqrt{n}).
    \end{align}
Now, let $\Omega_{\mathsf{good}}$ be the set of all random-cluster configurations $X$ such that $|\mathcal{L}_1(X)| \ge (\theta_* + s)n$, $|\mathcal{L}_2({X})| = O(\log n)$ and $R_2^-({X}) = O(n)$.    
Lemma 3.25 from~\cite{ABthesis} shows that if $X_t \in \Omega_{\mathsf{good}}$, then $X_{t+1}\in \Omega_{\mathsf{good}}$ with probability $1-O(n^{-1})$. (Again, this result from~\cite{ABthesis} is stated for $\beta \ge q$, but it extends to the $q>2$ and $\beta > \betau$ setting by the same observations made above about the subcriticallity/supercriticality of the percolation step; in addition, Lemma 3.25 only states a $1-o(1)$ but its proofs yields a $1-O(n^{-1})$ bound on the probability.)
By averaging over all the configurations on $\Omega_{\mathsf{good}}$, we get from~\eqref{eq:delta-bound-1} that
 \begin{align}
   \label{eq:delta-bound-2}
    \E[\Delta_{t+1}] \le \big(1-\tfrac{\delta}{q} \big) \E[\Delta_t] + O(\sqrt{n}) + n \cdot \Pr(X_t \not\in \Omega_{\mathsf{good}}) \le \big(1-\tfrac{\delta}{q} \big) \E[\Delta_t] + O(\sqrt{n}).
    \end{align}
  Iterating this bound, we obtain
    \[
     \E[\Delta_{t+1}] \le  \big(1-\tfrac{\delta}{q} \big)^t \Delta_0+ O(\sqrt{n}).
    \]
    Since $\Delta_0 = O(n)$, there exists some $T=O(\log n)$ so that $\E[\Delta_T] \le  C\sqrt{n}$, and by Markov's inequality 
    we have $\Delta_T \le 2C \sqrt{n}/\epsilon$ with probability $1 - \epsilon/2$ for any fixed $\epsilon>0$.
    Finally, note that $X_T \in \Omega_{\mathsf{good}}$ with probability $1-o(1)$ and
    in the percolation step of the last step,  
    Lemma~\ref{lem:subcriticalgstructure} and \ref{lem:supercriticalgstructure} imply that $I_1(X_{T}) = \Omega(n)$ with probability $1-o(1)$.
    The result then follows from a union bound.
\end{proof}

The next lemma shows that once the giant's size is within $O(\sqrt{n})$ of $\thetar n$, mixing happens in at most $O(\log n)$ further steps. 

\begin{lemma}
\label{lemma:couplestagesupercritical}
    Let  $q > 2$ and $\beta > \betau$. Suppose ${X_0}$ is a configuration satisfying all the following conditions
    \begin{enumerate}
        \item $||\mathcal L_1(X_0)| - \thetar n | = O(\sqrt{n})$;
        \item $I_1({X}_0) = \Omega(n)$;
        \item $|\mathcal{L}_2({X}_0)| = O(\log n)$;
        \item $R_2^-({X}_0) = O(n)$.
    \end{enumerate}
    Suppose ${Y}_0$ also satisfies all these conditions. 
    Then for any constant $\varepsilon>0$, 
    there exists  $T=O(\log n)$ and a coupling of $({X}_t, {Y}_t)$ such that ${Y}_T = {X}_T$ with probability at least $1-\varepsilon$.
\end{lemma}

\begin{proof}
This lemma essentially follows from 
Lemmas 3.16 and 3.27 in~\cite{ABthesis} but requires a slight generalization of the latter. Specifically, Lemma 3.27 from~\cite{ABthesis} 
provides a coupling 
from two configurations satisfying conditions 1 to 4 in the lemma statement
to two configurations with the same component structure but
assumes that $\beta \ge q$ and only provides an $\Omega(1)$ bound on the probability of success of the coupling.
Lemma 3.16 from~\cite{ABthesis} provides a coupling from two configurations with the same component structure to the same configurations with high probability and holds for any $q>1$ and $\beta >0$.

Our first observation is that Lemma 3.27 and Corollary 3.33 from~\cite{ABthesis} hold when $q>2$ and $\beta > \betau$. (This is a byproduct of the percolation step of the CM dynamics being subcritical (resp. supercritical) when the largest component of the configuration is inactive (resp., active).)
To boost the probability of success of the coupling, we note that at a suitable $T_0=O(\log n)$, either the coupling to the same component structure succeeds with probability at least $\alpha = \Omega(1)$, or, by Corollary 3.33 from~\cite{ABthesis}, both $X_{T_0}$ and $Y_{T_0}$ satisfy conditions 2, 3 and 4 from the lemma statement with probability $1-o(1)$ and also $| |\mathcal{L}_1(X_{T_0})| - \thetar n | = O(\sqrt{n} \log^2 n)$, $| |\mathcal{L}_1(Y_{T_0})| - \thetar n | = O(\sqrt{n} \log^2 n)$.
Then, by Lemma~\ref{lem:hittingtimecritical3}, at time $T_1 = T_0 + O(\log n)$, we have that all four conditions from the lemma statement hold with probability $1-1/A$ for any desired constant $A > 0$.
Iterating this reasoning,  
we obtain a coupling for which
${X}_{kT_1}={Y}_{kT_1}$ with probability at least $1-(1-\alpha)^k - k/A - o(1)$ for any constant $k$. Letting $T=kT_1$ with $k$ and $A$ sufficiently large,
we obtain a coupling under which $X_T$
and $Y_T$ have the same component structure with probability at least $1-\epsilon/2$ for any $\epsilon>0$. 
The result then follows from Lemma 3.16 in~\cite{ABthesis} and a union bound.
\end{proof}

The above lemmas are the key ingredients to establish the first part of Lemma~\ref{lem:quasi-equilibration-away-from-saddle-SWCM}. 

\subsubsection*{Quasi-equilibrating to the disordered phase}
We require analogous lemmas to establish the second part, namely the quasi-equilibration to the disordered phase if initialized with a largest component of size at most $\theta_* n - \gamma \sqrt{n}$. 

\begin{lemma}
\label{lem:hittingtimesubcritical}
    Let $q>2$, $\gamma>0$ and $\beta \in (\betau, \betas)$. If $X_0$ is a configuration such that $L_0 \le \theta_*n - \gamma\sqrt{n}$,  
    $|\mathcal{L}_2(X_0)| = O(\log n)$ and $R_2^-(X_0) = O(n)$,
    then there exist constants $s\in(0,\theta_* - \thetas) ,C>0$ such that for $T = C \log n$ we have
     $L_T \le (\theta_* - s)n$, $|\mathcal L_2(X_T)|= O(\log n)$, and $R_2^-(X_T)= O(n)$, with probability $1 - O(\gamma^{-2})$.
\end{lemma}
The proof of Lemma~\ref{lem:hittingtimesubcritical} is essentially identical to that of Lemma~\ref{lem:hittimesupercritical2} and is thus omitted. The next lemma shows how the giant's size goes from $(\theta_* - s)n$ to $O(\log n)$; a little care is needed here compared to the ordered side because on its way the giant may approach $\thetas$ where it can take one step as a critical random graph. To deal with this, we recall the following two lemmas from~\cite{ABthesis}. 

\begin{lemma}[Fact 3.18, \cite{ABthesis}]
\label{lem:nosecondlarge}
Let $0<\beta<\betas$ and $X_0$ has a unique component that is of size at least $2n^{11/12}$, then $|\mathcal{L}_2(X_t)|<2n^{11/12}$ for all $0\le t\le T$ with probability $1 - O(T\cdot n^{-1/12})$ for any $T=O(\log n)$.
\end{lemma}

\begin{lemma}[Fact 3.20, \cite{ABthesis}]
    \label{lem:decaysubphase3}
    Let $0<\beta < \betas$. If $X_0$ is a configuration such that $L_0 \le (\thetas - \epsilon)n$ and  $X_0$ has at most one large component whose size is at least $2n^{11/12}$, then there exists $T=O(\log n)$ such that $L_T = O(\log n)$ with high probability.
\end{lemma}
Note that although originally Fact 3.20 in \cite{ABthesis} guarantees only $\Omega(1)$ probability for the statement to hold, 
it can be improved to, for example $1-n^{-1/2}$, with a careful look.

With the above lemmas recalled, we show how the size of the giant goes from at most $(\theta_* - s) n$ to $o(n)$. 

\begin{lemma}
\label{lem:subcriticalhitting3}
    Let $q>2$, $s>0$ and $\beta \in (\betau, \betas)$. 
    If  $X_0$ is a configuration such that $L_0 \le (\theta_* - s)n$,  $|\mathcal{L}_2(X_0)| = O(\log n)$ and $R_2^-(X_0) = O(n)$, then there exists $T=O(\log n)$ 
    such that $L_T=O(\log n)$ and $I_1(X_T) = \Omega(n)$ with probability $1-o(1)$.
\end{lemma}

\begin{proof}
    Recall we set $\thetas:=\frac{q-\beta}{\beta(q-1)}$. Fix a small $\varepsilon>0$, and define $\tau = \min\{\tau_\epsilon,\tau_2\}$ where 
    \begin{align*}
             \tau_\varepsilon:= \min\{t>0: L_t/n \notin [\thetas - \varepsilon, \theta_* - s/2]\} \quad \text{and} \quad \tau_2 = \min\{t>0: |\mathcal L_2(X_t)| > 2n^{11/12}\}\,.
    \end{align*}
    
    We bound the drift of the giant in two cases of $L_t$: (i) $L_t/n \in  [ \thetas + \epsilon, \theta_* -s/2]$, and (ii) $L_t/n \in [\thetas - \epsilon, \thetas+\epsilon]$ to upper bound $\tau_\varepsilon$.
   Lemmas~\ref{lem:L1decayexpectation} and \ref{lem:phifunctionsubcritical} imply that 
    in case (i) with $X_t: |\mathcal{L}_2(X_t)| < 2n^{11/12}$, 
    there exists a constant $\delta>0$ such that 
    \begin{equation}\label{eq:decaysubphase1}
        \E[L_{t+1} - L_t \mid X_t] \le \frac{f(L_t/n)n}{q} + O(n^{1/4}) \le - \frac{\delta n}{q} + O(n^{1/4}).
    \end{equation}
    Similarly, in case (ii), Lemmas~\ref{lem:L1decayexpectation} and \ref{lem:phifunctionsubcritical} imply that
    \begin{equation}
        \label{eq:decaysubphase2}
        \E[L_{t+1} - L_t \mid X_t] \le \frac{f(\thetas + \epsilon)n}{q} + \frac{2\epsilon n}{q} + O(n^{1/4}) 
        \le - \frac{\delta n}{q} + \frac{2\epsilon n}{q} + O(n^{1/4}).
    \end{equation}
    By choosing $\epsilon$ small enough, we see from \eqref{eq:decaysubphase1} and \eqref{eq:decaysubphase2} that there exists $\eta >0$ such that if $t<\tau$, then 
    \[
        \E[L_{t+1} - L_t \mid X_t] \le -\eta n.
    \]
    By a standard application of the optional stopping theorem (see, e.g., Lemma 2.20 in~\cite{ABthesis}), we have
    $\E[\tau] \le 4/\eta$. By Markov's inequality, $\Pr(\tau > \frac{4 \log n}{\eta}) \le 1/\log n$, so for $T_1 = \frac{4\log n}{\eta}$, $\tau<T_1$ with high probability.
    
    At the same time, by Lemma~\ref{lem:nosecondlarge}, $|\mathcal{L}_2(X_t)| < 2n^{11/12}$ holds for all $t\le T$ with probability $1 - O(Tn^{-1/12})$, so with high probability it is $\tau_\epsilon$ that is attained and $\tau_\epsilon <T_1$. 
    Moreover, while $L_t/n <\theta_* - s/2$ and $|\mathcal L_2(X_t)|<2n^{11/12}$, the conditional variance of $L_{t+1} - L_t$ is at most $O(n^{23/24})$ so by Chebyshev's inequality and a union bound, the probability that $\tau_\epsilon<T_1$ is attained by $L_t/n > \theta_* - s/2$ is at most $T_1 n^{-1/24}= o(1)$.   
    Altogether, with high probability $L_{T_1} \le (\thetas - \epsilon)n$ and $|\mathcal L_2(X_{T_1})|\le 2n^{11/12}$.

    At this point, Lemma~\ref{lem:decaysubphase3} implies that after $T_2 = O(\log n)$ additional steps, 
    the largest component in the configuration has size $O(\log n)$ with high probability.
    Finally, in the percolation step of the very last step,  
    Lemma~\ref{lem:subcriticalgstructure} and \ref{lem:supercriticalgstructure} imply that $I_1(X_{T_1 + T_2}) = \Omega(n)$ with high probability.
    The result follows from a union bound.
\end{proof}

To quasi-equilibrate to $\mu^\dis$ from here, we appeal to the following lemma which lower bounds the probability of coupling two shattered configurations in $O(\log n)$ steps. 

\begin{lemma}[{\cite[Lemmas 3.15--3.16 and Fact 3.17]{ABthesis}}]
\label{lemma:couplingstagesubcritifcal}
    Let $q>1$ and $0<\beta < \betas$. Let $X_0$ be a random-cluster configuration such that $L_0 = O(\log n)$ and $I_1(X_0) = \Omega(n)$.
Suppose $Y_0$ also satisfies these conditions. 
    Then for any $\epsilon>0$ there exist $T=O(\log n)$ and a coupling of $(X_t, Y_t)$ such that $X_T = Y_T$ with probability at least $1-\epsilon$.
\end{lemma}

We are now in position to put all the above ingredients together to establish Lemma~\ref{lem:quasi-equilibration-away-from-saddle-SWCM}. 

\begin{proof}[\textbf{\emph{Proof of Lemma~\ref{lem:quasi-equilibration-away-from-saddle-SWCM}}}]
 Suppose $(X_t)_{t\ge 0}$ is initialized from $X_0$ with $L_0 \ge \theta_* n + \gamma  \sqrt{n}$, $|\mathcal L_2(X_0)| = O(\log n)$, and $R_2^-(X_0) = O(n)$. Let $(Y_t)_{t\ge 0}$ be initialized from $\mu^\ord$ and restricted to $\Omega^\ord$ by rejecting any update that would take it out of $\Omega^\ord$; by Lemma~\ref{lem:equilibrium-estimates} and a union bound, with probability $1-o(1)$, $(Y_t)_{t\ge 0}$ does not feel the restriction to $\Omega^\ord$ (i.e., doesn't attempt to leave $\Omega^\ord$) for exponential in $n$ many steps, so for all $e^{o(n)}$ time steps, we can treat $Y_t$ as the (unrestricted) CM chain initialized from $\mu^\ord$. We claim that there is $T= O(\log n)$ such that with probability $1-O(\gamma^{-2})$, we have coupled $X_T = Y_T$, which would imply the claim. In order to see this, notice that Lemma~\ref{lem:hittimesupercritical2} can be stitched with Lemma~\ref{lem:hittingtimecritical3} to get, with probability $1-O(\gamma^{-2})$, a configuration $X_{T_1}$ on which Lemma~\ref{lemma:couplestagesupercritical} can be applied. At the same time, $Y_{T_1}$ satisfies the conditions of Lemma~\ref{lemma:couplestagesupercritical}, being a sample from $\mu^\ord$ (which satisfies those conditions by Lemma~\ref{lem:equilibrium-estimates}).   

 The proof for $L_0 \le \theta_* n - \gamma  \sqrt{n}$ is analogous (Lemmas~\ref{lem:hittimesupercritical2}--\ref{lemma:couplestagesupercritical} are replaced by Lemmas~\ref{lem:hittingtimesubcritical}--\ref{lemma:couplingstagesubcritifcal}). 
\end{proof}

\subsection{Getting away from the unstable fixed point at criticality}\label{subsec:escaping-saddle-SWCM}

We now focus on the critical point $\beta = \betac$ where it is essential to understand the diffusion away from the fixed point $\theta_* n$; here, the drift and fluctuations of the giant component process compete on the same scale.

\subsubsection{Staying in a nice set of configurations} 

We begin with certain a priori estimates guaranteeing that for a sufficiently long period of time ($O(1)$ times will suffice), the near-saddle dynamics has largest component size, and sum of squares of other components, that are concentrated around explicit quantities. Moreover, the sum of cubes of components, and isolated vertices (things needed for sharp local limit theorems in updates) stay on the right order. 

We begin by defining the following good set that captures all the a priori estimates except the concentration of the sum of squares of component sizes, which will come subsequently. 

\begin{definition}
    For a constant $K>0$,  
    let $\mathcal G_K$ be all random-cluster configurations $U$ satisfying: 
    \begin{enumerate}
        \item The largest component has $| |\mathcal{L}_1(U)|  - \theta_* n| \le K n^{1/2} \log n$,
        \item The second largest component has $|\mathcal{L}_2(U)| \le K\log n$,
        \item The sum of squares of non-giant components has $R_2^-(U) \le K n$,
        \item The sum of cubes of non-giant components has $R_3^-(U) \le K n$,
        \item The singletons have $I_1(U)  \ge \frac{n}{K}$.
    \end{enumerate}
\end{definition}

We introduce some notations that will be useful in what follows.
Recall for $\lambda>1$, that $\alpha(\lambda)$ is defined to be the largest positive root of the equation~\eqref{eq:beta} (i.e., $\alpha(\lambda) n$ is approximately the expected size of the giant in a $G(n,\lambda/n)$), and recall $\sigma^2(\lambda)$ as defined in~\eqref{eq:sigmag}
which approximates the variance of the giant component in $G(n,\lambda/n)$. 
Also, for $\theta\in [0,1]$, let $\ka(\theta)$ be the expected fraction of activated vertices if a giant of fractional size $\theta$ is activated, and $\kia(\theta)$ be the same quantity if it is not activated, i.e., 
\begin{align}\label{eq:activated-fractions}
    \ka(\theta) = \theta + \frac{1}{q} (1-\theta)\,, \qquad \text{and} \qquad \kia(\theta) = \frac{1}{q} (1-\theta)\,.
\end{align}
Moreover, for $i\ge 1$ we use $\Lambda_i$ to denote the event that the largest component is activated in step $i$ of the CM dynamics; we use $A_i$ to denote the total number of activated vertices in step $i$, and on $\Lambda_i$ use $A_i^-$ to denote the activated number minus the giant.  
Lastly, let $G_0 \sim G(n,p_0)$ and denote by $G_i$ the random graph resampled in step $i$ of $\{{X}_t\}$, 
that is, $G_{i} \sim G(A_{i}, \beta/n)$. 

Let $\lambda_*$ be the solution to $\alpha(\lambda_*) = \theta_*$.

\begin{lemma}
    \label{lem:basicproperties}
    Let $\beta = \betac$.
    For all $T\ge 0$ fixed independent of $n$, if~ $X_0 \sim G(n, \frac{p_0}{n})$ with $p_{0}  =\lambda_*+ O(n^{-1/2}) $, then for $K=K(T)$, $X_T \in \mathcal G_K$ with probability $1-o(1)$. 
\end{lemma}

\begin{proof}[\textbf{\emph{Proof of Lemma~\ref{lem:basicproperties}}}]
    Let ${X}_0\sim G(n, \frac{p_0}{n})$. For some $K_0 >0$, Lemma~\ref{lem:supercriticalgstructure} implies that, ${X}_0 \in \mathcal{G}_{K_0}$ with $1-o(1)$ probability.
        Let $K_t = t K_0$.
    Now suppose ${X}_t \in \mathcal{G}_{K_t}$ for some $t\ge 0$. 
    We will inductively show that 
    ${X}_{t+1} \in \mathcal{G}_{K_{t+1}}$ with probability $1-o(1)$.
    Then a union bound over $t\in[0,T]$ implies ${X}_T \in \mathcal{G}_{K_T}$.

     We consider the two cases, $\Lambda_{t+1}$ and $\Lambda^c_{t+1}$.
    First, $\E[A_{t+1} \mid \Lambda_{t+1}, {X}_t] = \ka(L_t/n) n$.
    Moreover, since property 3 holds at time $t$, by Hoeffding's inequality, we have
    \begin{equation}
        \label{eq:activationconcentrate5}
        \Pr\Big(   | A_{t+1} - \ka(L_t/n) n | \ge \sqrt{n} \log n \mid \Lambda_{t+1}, {X}_t \Big) 
        \le 2\exp\Big( - \frac{n \log^2 n}{R_2^-({X}_t)} \Big) 
        = \exp{\left( -C_1 \log^2 n \right)},
    \end{equation}
    where $C_1 >0$ is some constant. 
    In the percolation step, $G_{t+1} \sim G(A_{t+1}, \betac/n)$,
    and when the concentration in \eqref{eq:activationconcentrate5} holds, 
    the random graph $G_{t+1}$ is supercritical.
    By Lemma~\ref{lem:rgdeviation}, $|\mathcal{L}_2(G_{t+1})| = O(\log n)$ with probability $1 - O(n^{-1})$ and
    $| |\mathcal{L}_1(G_{t+1})|-\theta_*n|\le n^{1/2} \log n$ with probability $1-o(n^{-1})$.
    Since $|\mathcal{L}_2(X_t)| \le K_t \log n$ by property 2 and the largest component of ${X}_t$ has been activated, if $|\mathcal{L}_1(G_{t+1})| = \Omega(n)$, then $\mathcal{L}_1({X}_{t+1}) = \mathcal{L}_1(G_{t+1})$.
    Moreover, with probability $1-o(1)$, we have $G_{t+1} \in \mathcal{G}_{K_0}$ by Lemma~\ref{lem:supercriticalgstructure}.
    Consequently, we also establish property 2-5 by noting that
    \begin{enumerate}
        \item    $|\mathcal{L}_2(X_{t+1})| \le \max\{ |\mathcal{L}_2(X_{t})|, |\mathcal{L}_2(G_{t+1})| \} = 
        \max\{K_0, K_t \} \log n$.
        \item $R_2^-(X_{t+1}) \le R_2^-(X_{t}) + R_2^-(G_{t+1}) 
        \le (K_0 + K_t) n$.
        \item $R_3^-(X_{t+1}) \le R_3^-(X_{t}) + R_3^-(G_{t+1}) = (K_0 + K_t)n$.
        \item $I_1(X_{t+1}) \ge I_1(G_{t+1}) \ge  \frac{n}{K_0} \ge \frac{n}{K_0+K_t}$.
    \end{enumerate}
    
    Next, suppose $\Lambda_{t+1}$ does not happen. 
    Then $\E[A_{t+1} \mid \Lambda^c_{t+1}, {X}_t] = \kia(L_t/n) n$.
    Again, by property 3 and Hoeffding's inequality, we have the following concentration for $A_{t+1}$:
    \begin{equation}
        \label{eq:activationconcentrate6}
        \Pr\left(          \left| A_{t+1} -\kia(L_t/n) n \right| \ge \sqrt{n} \log n \mid \Lambda^c_{t+1}, {X}_t \right) 
        \le 2\exp\Big( - \frac{n \log^2 n}{R_2^-({X}_t)} \Big) 
        = \exp{\left( -C_2 \log^2 n \right)},
    \end{equation}
    where $C_2>0$ is a constant.  
    In the percolation step,  $G_{t+1} \sim G(A_{t+1}, \betac/n)$ and 
    if the estimate in \eqref{eq:activationconcentrate6} holds then $G_{t+1}$ is subcritical. 
    In this case, by Lemma~\ref{lem:subcriticalgstructure}, $|\mathcal{L}_1(G_{t+1})| \le K_0 \log n$ with probability $1-o(1)$, so 
    $L_{t+1} = L_t$, satisfying the first property of $\mathcal{G}_{(t+1)K_0}$.
    In addition, with probability $1-o(1)$, $G_{t+1}$ satisfies that $R_2(G_{t+1}) \le K_0 n, R_3(G_{t+1}) \le K_0 n, I_1(G_{t+1}) \ge  n/K_0$. 
    Similar to the case $\Lambda_{t+1}$, these properties combined with the hypothesis ${X}_t \in \mathcal{G}_{K_t}$ imply that ${X}_{t+1} \in \mathcal{G}_{K_{t+1}}$. 
    Therefore, by a union bound over all the above, we have ${X}_{t+1} \in \mathcal{G}_{K_{t+1}}$ with probability $1-o(1)$.
\end{proof}
   
   At this point, we need one  more property to hold throughout the process, which is concentration of $R_2^-( X_t)$ around an explicit {deterministic} quantity. To define that quantity, in addition to $G_0 \sim G(n,p_0/n)$, let $G_* \sim G(n\kia(\theta_*), \betac/n)$. Then define a sequence of variances 
       \begin{align}\label{eq:sigmas}
        \sigma_s^2 := \tfrac{1}{q} (1-\tfrac{1}{q}) \Big( (1-\tfrac{1}{q})^{s} \mathbb E[R_2^-(G_0)] +  \sum_{i=1}^s (1-\tfrac{1}{q})^{s-i} \mathbb E[R_2(G_*)]\Big)\,.
    \end{align}

    We show that $\frac{1}{q}(1-\frac{1}{q})R_2^-(X_t)$ is concentrated around $\sigma_t^2$ for all $O(1)$ times. 
    \begin{definition}\label{def:G_t}
        For $K(t)$ as in Lemma~\ref{lem:basicproperties}, let $\mathcal G_t$ be the event that for every $s\le t$, $ X_s \in \mathcal G_{K(t)}$ and furthermore 
        \begin{align*}
             \Big|\frac{1}{q} (1-\frac{1}{q}) R_2^-(X_s) - \sigma_s^2\Big| \le \sqrt{n} \log^2 n \qquad \text{for all $s\le t$}\,.
        \end{align*}
    \end{definition}    
    \begin{lemma}
        \label{lem:Gt}
    Let $\beta= \betac$.
       For all $T\ge 0$ fixed independent of $n$, 
       if~ $X_0 \sim G(n, \frac{p_0}{n})$ with $p_{0}  =\lambda_*+ O(n^{-1/2}) $,
       then $(X_s)_{s\le T} \in \mathcal G_T$, with probability $1-o(1)$. 
    \end{lemma}
\begin{proof}
    Fix $T \ge 0$.
    We will show that for each $s\le T$, the event in Definition~\ref{def:G_t} holds with probability $1-o(1)$; the lemma follows from a union bound over all $s\le T$. 
    First, Lemma~\ref{lem:basicproperties} shows that $X_0,...,X_T \in \mathcal{G}_{K}$ with probability $1-o(1)$.
    Working on this event, when $X_0, \dots, X_s \in \mathcal{G}_K$ we can express 
    \begin{equation}
        \label{eq:R2decomp}
        R_2^-( X_s) =  R_2^- (G_0^s) + \sum_{i=1}^{s} 
        \big[\mathbf{1}\{\Lambda_i\} R_2^- (G_i^s) + \mathbf{1}\{\Lambda^c_i\} R_2(G_i^s)\big].
    \end{equation}
    where we are using $G_i^t$ to denote the (random) subgraph induced on the vertex set of $G_i$ (the resampled portion of $X_i$), consisting of vertices that are not  re-activated up through time $t$ (i.e., do not belong to $\bigcup_{j=i+1}^{t} G_j$). That is, the vertices of $G_i$ that survive $t-i$ independent activation steps. 
    Eq.~\eqref{eq:R2decomp} is a consequence of the fact that $X_t$ is partitioned by $\bigcup_{i=0}^t G_i^t$ ($G_t^t$ is the activated subgraph at time $t$).

        Similarly, denote by $G_*^t$ the induced subgraph of $G_*$ on the vertices that have not been activated for $t$ activation steps of the CM dynamics starting at $G_*$. 
    We will show that 
    for each $i\ge 1$,
    \begin{equation}
        \label{eq:R2remainderconcentrate7}
        \Pr\left( 
        \left|  \mathbf{1}\{\Lambda_i\} R_2^- (G_i^s) + \mathbf{1}\{\Lambda^c_i\} R_2(G_i^s) -  \E[R_2(G_*^{s-i})]\right| = O(\sqrt{n} \log n)
        \right) = 1-o(1), 
    \end{equation}
    and 
    \begin{equation}
        \label{eq:R2concentrate8}
        \Pr\left( 
        \left| R_2^-(G_0^{s}) -  \E[R_2^-(G_0^{s})]\right| = O(\sqrt{n} \log n) \right) = 1- o(1).
    \end{equation}
    By \eqref{eq:R2decomp} and a union bound of \eqref{eq:R2remainderconcentrate7}--\eqref{eq:R2concentrate8} over $i=1,\dots, s$, together with Observation~\ref{obs:Xr}, we get the claimed bound for $R_2^-(X_s)$ with probability $1-o(1)$.

    Now we fix $i\in [1,s]$ and show \eqref{eq:R2remainderconcentrate7}.
    Up to an error of $o(1)$,  we assume ${X}_{i-1} \in \mathcal{G}_K$.
    Suppose first the giant of ${X}_{i-1}$ is activated in the step $i$, namely $\Lambda_i$.
  We define the activated window, 
  $$\Wa:=\left[n\ka(\theta_*) - 2\sqrt{n} \log n, n\ka(\theta_*) + 2\sqrt{n} \log n\right].$$
  By \eqref{eq:activationconcentrate5}, we have $|A_{i} - \ka(L_{i-1}/n)n| \le \sqrt{n} \log n$ with probability $1-o(1/n)$;
    by the first property of $\mathcal{G}_K$, we also have $|\ka(L_{i-1}/n)n - \ka(\theta_*)n|  \le \sqrt{n} \log n$.
    Let us work on the event that $A_i \in \Wa$.
We will rely on the following lemma which establishes concentration of the number of vertices not yet activated after $s-i$ steps; its proof follows from a simple calculation and Chebyshev's inequality, and is deferred.

\begin{lemma}\label{lem:concentration-of-G^r}
Suppose $X\sim G(n,p)$ with $n p$ bounded away from $1$ uniformly in $n$, and let $X^r$ be the sub-graph of $X$ that does not get activated in $r$ activation steps. 
 If $n p > 1$ then
 \[
 \mathbb \Pr\big(|R_2^-(X^r) - \mathbb E[R_2^-(X^r)]| > \sqrt{n} \log n\big) = O\big( \frac{1}{\log^{2} n} \big)\, ,
 \]
 and if $n p < 1$ then the same concentration holds for $R_2(X^r)$. 
\end{lemma}
We also use the following lemma that compares the expectation of $G_i^s$ to that of $G_*^{s-i}$ (the latter having deterministic parameter, while the former has a random parameter dictated by $A_i$, though close to the deterministic parameter when $A_i \in \Wa$). 

\begin{lemma}
    \label{lem:close2gstar}
    Suppose $G \sim G(n\ka(\theta_*) +m, \frac{\betac}{n})$, where $|m| =O(\sqrt{n} \log n)$, and $G_* \sim G(n\kia(\theta_*), \frac{\betac}{n})$.
    Then for any integer $r\ge 0$,
    \[
        |\E[R_2^-(G^r)]  - \E[R_2(G_*^r)]| = O(\sqrt{n} \log n).
    \]
\end{lemma}

Note that $A_i\in \Wa$ implies that $G(A_i, \betac/n)$ is supercritical (uniformly in $n$). Combining Lemmas~\ref{lem:concentration-of-G^r}--\ref{lem:close2gstar}, the above inequality ensures that if  $A_i \in \Wa$,  we have
\begin{align*}
    \Pr\big( \big|R_2^-(G_i^s)- \mathbb E[R_2(G_*^{s-i})]\big| = O(\sqrt{n} \log n) \mid A_i \in \Wa \big)  = 1- O\big( \frac{1}{\log^{2} n} \big)\,.
\end{align*}

Next suppose that that $\Lambda_i$ does not happen. 
    Set $\Wia:=[n\kia(\theta_*) - \sqrt{n} \log n, n\kia(\theta_*) + \sqrt{n} \log n]$.
    By \eqref{eq:activationconcentrate6} we have $A_i \in \Wia$ with probability $1-o(1/n)$.
    In this case,  $G(A_i, \frac{\betac}{n})$ is subcritical (uniformly in $n$).
    Combining Observation~\ref{obs:Xr} and \eqref{eq:ER2R3} in Lemma~\ref{lemma:Skmoment}, 
    we get that 
    \[
                |\E[R_2(G^s_i)]  - \E[R_2(G_*^{s-i})]| = O(\sqrt{n} \log n).
    \]
    Then by Lemma~\ref{lem:concentration-of-G^r} we have
    \begin{equation}
        \Pr\big(
        \big| R_2(G_i^{s}) - \E[R_2(G_*^{s-i})] \big| = O(\sqrt{n} \log n) \mid A_i \in \Wia
        \big) = 1 - O\big(\frac{1}{\log^{2} n} \big).
    \end{equation}
    Hence, we get \eqref{eq:R2remainderconcentrate7} by conditioning on $\Lambda_i$ and its complement and
    a union bound.
    The argument for~\eqref{eq:R2concentrate8} is only easier than the above, following immediately from Lemma~\ref{lem:concentration-of-G^r} and the super-criticality of the initialization parameter.
\end{proof}

\subsubsection{Upper bound on escape time from the unstable fixed point}
Let $\tau_{\gamma}$ denote the first time $X_t > \theta_*n + \gamma\sqrt{n}$ and let $\tau_{-\gamma}$ denote the first time $X_t < \theta_*n - \gamma\sqrt{n}$. 
We show an upper bound of the exit time from the window $I_\gamma:=[\theta_*n - \gamma\sqrt{n}, \theta_*n + \gamma\sqrt{n}]$, where $\gamma$ is a sufficiently large constant, to be chosen depending on the $\epsilon$-total-variation distance to stationarity we aim for.
\begin{lemma}
    \label{lemma:exittime}
    Let $\tilde{\tau} := \min\{\tau_\gamma, \tau_{-\gamma}\}$.
   If~ $X_0 \sim G(n, \frac{p_0}{n})$ with $p_{0}  =\lambda_*+ O(n^{-1/2}) $
    then there exists a constant $C>0$ such that for every $T\ge 1$ fixed independent of $n$,  
    \[
    \Pr(\Tilde{\tau} \le   2T e^{C \gamma^2}) \ge 1 - \frac{1}{T}\,.
    \]
\end{lemma}
\begin{proof}
    If $L_0 \notin I_\gamma$, then $\Tilde{\tau} = 0$ and the lemma holds trivially. 
    We assume therefore that $L_0 \in I_\gamma$.
    Also we assume $X_0 \in \mathcal{G}_K$, for a large constant $K$, which happens with probability $1- o(1)$ by Lemma~\ref{lem:basicproperties}.
    Let $\tau' := \min\{ t: X_t\notin \mathcal{G}_K \}$ and 
    let $\tau = \Tilde{\tau} \wedge \tau'$.
    Observe that $\tau$ is stochastically dominated from above by 
    a geometric random variable $\text{Geo}(y_\gamma)$, where 
    \[
    y_\gamma:= \min_{X_t \in  \mathcal G_K: L_t\in I_\gamma} \Pr(L_{t+1} > \theta_*n + \gamma\sqrt{n}  \mid X_t )\,.
    \]
    Our aim in what follows is to show existence of $\kappa>0$ such that 
    \begin{align}\label{eq:wts-escape-time-SWCM}
    y_\gamma \ge e^{-\kappa \gamma^2}\,.
    \end{align}
    Indeed, assuming~\eqref{eq:wts-escape-time-SWCM},    $
    \E[\tau \mid X_0 \in I_\gamma] \le \E[\text{Geo}(y_\gamma)] \le e^{\kappa \gamma^2}
    $, and 
    by Markov's inequality, we would have
    $
    \Pr(\tau > 2T e^{ \kappa \gamma^2}) \le \frac{1}{2T}
    $ for every $T$.
    Then we have
    \[
        \Pr\big(\Tilde{\tau} \ge 2T e^{\kappa \gamma^2}\big) 
        \le \Pr\big(\tau > 2T e^{ \kappa \gamma^2}\big) + \Pr\big(\tau' \le 2T e^{\kappa \gamma^2}\big) \le \frac{1}{2T} + 2Te^{\kappa \gamma^2} \cdot o(1)\le  \frac{1}{T}\,,
    \]by Lemma~\ref{lem:basicproperties} and a union bound.

    To show~\eqref{eq:wts-escape-time-SWCM}, note that $\Pr(\Lambda_{t+1}) = 1/q$, and we suppose $\Lambda_{t+1}$ happens.
Then  $L_{t+1}$ is distributed like the giant component of $ G(A_{t+1}, \frac{\betac}{n})$, 
    where $A_{t+1}$ is the number of activated vertices in step $t+1$.
    By property 3 of $\mathcal{G}_K$,   \begin{equation}
         \label{eq:varA1}
         \var(A_{t+1} \mid X_t, \Lambda_{t+1}) = 
         \frac{1}{q} \big(1-\frac{1}{q}\big) R_2^-(X_t) \le Kn.
     \end{equation}
    Suppose $L_t = x_0 \in I_\gamma$.
    By Chebyshev's inequality
    we have
    $$A_{t+1} \ge  x_0 + \frac{n-x_0}{q} - 2\sqrt{Kn} \ge \ka(\theta_* - \gamma n^{-1/2}) n -2\sqrt{Kn} =:m\,,$$ 
    with probability at least $1/4$. 
    On this event, $G(A_{t+1}, \frac{\betac}{n}) \succeq G(m, \frac{\betac}{n})$ so in order to show~\eqref{eq:wts-escape-time-SWCM}, it suffices to show that if $X\sim G(m,\betac/n)$, 
    \begin{align}\label{eq:wts-escape-time-ub}
        \Pr(|\mathcal L_1(X)|\ge \theta_* n + \gamma \sqrt{n}) \ge 4q e^{-\kappa \gamma^2}\,.
    \end{align}
    Let $L= |\mathcal L_1(X)|$ for $X\sim G(m,\betac/n)$. Then, 
    Let $\lambda = m \cdot \frac{\betac}{n}$.
    By subtracting $\alpha(\lambda) m$ on both sides and dividing by $\sigma(\lambda) \sqrt{m}$,  we rewrite the left-hand side of \eqref{eq:wts-escape-time-ub} as 
    \begin{equation}
    \label{eq:alg1}
     \Pr(L \ge \theta_*n + \gamma\sqrt{n}) = 
                \Pr\Big(\frac{L -\alpha(\lambda) m}{\sigma(\lambda) \sqrt{m}} \ge \frac{\theta_*n + \gamma\sqrt{n}  -\alpha(\lambda) m}{\sigma(\lambda) \sqrt{m}} \Big)\,.
    \end{equation}
    By Theorem~\ref{thm:CLT}, we have 
    \begin{equation}
    \label{eq:applyclt}
          \Pr\Big(\frac{L -\alpha(\lambda) m}{\sigma(\lambda) \sqrt{m}} \ge \frac{\theta_*n + \gamma\sqrt{n}  -\alpha(\lambda) m}{\sigma(\lambda) \sqrt{m_1}} \Big)
         = \Pr \Big(Z \ge\frac{\theta_*n + \gamma\sqrt{n} - \alpha(\lambda) m}{\sigma(\lambda) \sqrt{m}} \Big) +o(1)\,,
    \end{equation}
    where $Z$ is a standard Gaussian random variable.
    Also, observe using the facts that $m = \Theta(n)$ and $|\alpha(\lambda) m - \theta_* n - \gamma\sqrt{n}| = O(\gamma\sqrt{n})$ by definition of $\theta_*$, 
    \begin{equation}
    \label{eq:Llowerbound}
        \frac{\theta_*n + \gamma\sqrt{n}  - \alpha(\lambda) m_1}{\sigma(\lambda) \sqrt{m_1}} \le C_1 \gamma,
    \end{equation}
     for some constant $C_1 > 0$.
    By standard Gaussian tail estimates, we get for some $C_2>0$ and $\gamma \ge 1$ say, 
    \begin{align*}
        \Pr(L \ge \theta_*n + \gamma\sqrt{n}) \ge e^{ - C_2 \gamma^2} +o(1)\,.
    \end{align*}
    This implies~\eqref{eq:wts-escape-time-ub} for some $\kappa>0$, for large $n$. 
\end{proof}

\subsubsection{Approximating by a 1-dimensional Markov process}
In order for mixing to be fast from the near-saddle initialization at $\beta= \betac$, we need the diffusion near the unstable fixed point to exit out the right and left with probabilities corresponding to the relative weights of the ordered and disordered phases. Towards capturing this, our goal in this subsection is to approximate the giant component process $(L_t)_{t\ge 0}$ near $\theta_* n$ by a monotone 1-dimensional Markov process. 

    In order to define this process, recall $\alpha(\lambda)$ and $\sigma^2(\lambda)$ from~\eqref{eq:beta}--\eqref{eq:sigmag}. For $\lambda>1$, let 
\begin{align}\label{eq:h1-h2}
h_1(\lambda):=\alpha(\lambda)+\alpha'(\lambda)\cdot \frac{\lambda}{\betac} \qquad \text{and} \qquad h_2(\lambda) := \sigma^2(\lambda)\cdot \frac{\lambda}{\betac} = \frac{\alpha(\lambda)(1-\alpha(\lambda))}{\left( 1 - \lambda(1-\alpha(\lambda)) \right)^2} \cdot \frac{\lambda}{\betac}.\end{align}
Since $\alpha(\lambda)$ is twice differentiable for $\lambda>1$, $h_1$, $h_2$ exist, and are differentiable. 

    \begin{definition}\label{def:Z-t-process}
    Define the 1-dimensional Markov process initialized from $Z_0$ and given $Z_t$, 
        \begin{align}\label{eq:Z-t-process}
            Z_{t+1} \sim Z_t + \varepsilon_{t+1} \Big(f'(\theta_*)Z_{t} + \sqrt{n} \mathcal N\big(0, (h_1(\betac \ka(\theta_*)))^2 \sigma_t^2 + h_2(\betac \ka(\theta_*))\big)\Big)\,,
        \end{align}
        where $\varepsilon_{t+1} \sim \text{Ber}(1/q)$ independently of the normal.  
    \end{definition}
    \begin{theorem}
        \label{thm:1dapproxZ}
        Suppose $X_0 \sim G(n, \frac{p_0}{n})$ with $p_0 = \lambda_* + O(n^{-1/2})$ and let $(Z_t)_{t}$ be the process of Definition~\ref{def:Z-t-process} initialized from $Z_0 \sim \mathcal N((p_0 - \lambda_*)\alpha'(\lambda_*)n, \sigma^2(\lambda_*)n)$.
        For all $T\ge 0$ fixed independent of $n$,
        there exists a coupling such that with probability $1-o(1)$, for all $t\le T$, $$|(L_t-\theta_* n) - Z_t| = O(\log n)\,.$$ 
    \end{theorem}

We consider the following intermediate process that stays within distance $1$ of $L_t$ w.h.p. and is easier to compare to $Z_t$: let $Y_0 \sim \mathcal N(\theta_* n + (p_0 - \lambda_*) \alpha'(\lambda_*) n, \sigma^2(\lambda_*) n)$, and given    
$Y_t \in \mathbbm{R}$ with $t\ge 0$, with probability $1-\frac{1}{q}$ set $Y_{t+1} = Y_t$ and with the remaining probability do the following:
\begin{enumerate}
\label{processY}
    \item generate a random real number $B_{t+1}\sim \mathcal{N}\big(\frac{n-Y_t}{q}, \sigma_t^2 \big)$, 
    \item generate a random number \[Y_{t+1} \sim \mathcal{N}\Big(\alpha\big(\tfrac{\betac}{n} \cdot (Y_t+ B_{t+1})\big)\cdot (Y_t+ B_{t+1}), \sigma^2\big(\tfrac{\betac}{n} \cdot (Y_t+ B_{t+1}) \big) \cdot (Y_t+B_{t+1})\Big).\]
\end{enumerate}
Here, recall $\sigma_t^2$ was defined in~\eqref{eq:sigmas}.  
Besides the centering by $\theta_* n$, the key difference between $Y_t$ and $Z_t$ is that the variance of its increments are functions of $Y_t$ themselves rather than simply functions of time.  

Theorem~\ref{thm:1dapproxZ} will follow from the following two lemmas. Recall that $\tilde{\tau} := \min\{\tau_\gamma, \tau_{-\gamma}\}$.

\begin{lemma}
    \label{lem:1dapproxY}
    Let ~ $X_0 \sim G(n, \frac{p_0}{n})$ with $p_{0}  =\lambda_*+ O(n^{-1/2}) $.
 For all $T\ge 0$ fixed independent of $n$,
    there exists 
    a coupling $\Pp$ of $\{({X}_t, Y_t)\}_{t\ge 0}$ such that for all $t\le T$,
    $\Pp(|L_t-  Y_t| \le 1) =1-o(1)$. 
\end{lemma}

\begin{lemma}
    \label{lem:ZapproxY2}
    There exists a coupling $\{(Y_t, Z_t)\}_{t\ge 0}$ such that for $t \le \Tilde{\tau}$,
    $|Y_t - (Z_t + \theta_*n)| = O(\log n)$ with probability $1-t \cdot n^{-\Omega(1)}$.
\end{lemma}

\begin{proof}[\textbf{\emph{Proof of Lemma~\ref{lem:1dapproxY}}}]
         We will couple $(L_t, Y_t)$ inductively.   
         For the base case, consider $t=0$; by Lemma~\ref{thm:giantLLT}, $(X_0,Y_0)$ can be coupled such that with probability $1-o(1)$, $|L_0 - Y_0|\le 1$. 
        
        Now suppose there exists a coupling of the first $t$ steps of $\{({X}_t, Y_t)\}$
        such that $|L_t - Y_t| \le 1$ with probability $1-o(1)$.
        By Lemma~\ref{lem:Gt}, we have $({X}_s)_{s\le t} \in \mathcal{G}_t$ with probability $1-o(1)$, so in what follows we work on that event.
        In case $({X}_s)_{s\le t} \notin \mathcal{G}_t$, we stop the coupling.
        For the induction step, we construct a coupling $\Pp_{t+1}$ of $({X}_{t+1}, Y_{t+1})$.

    On $\Lambda^c_{t+1}$, we let $Y_{t+1} = Y_t$.
    At the same time, on $\Lambda^c_{t+1}$ and $\mathcal G_t$, with high probability 
    the percolation step is sub-critical and therefore $L_{t+1} = L_t$ with probability $1-o(1)$ (a similar argument was made in the proof of Lemma~\ref{lem:basicproperties}). Hence, we have $|L_{t+1} - Y_{t+1}| = |L_t - Y_t|\le 1$ with probability $1-o(1)$ in this case. 
    
    Next we consider the case when $\Lambda_{t+1}$ occurs.
    Let $A_{t+1}^-$ be the number of activated vertices in step $t+1$ that are not in $\mathcal{L}_1({X}_t)$.
    Note that ${X}_t \in \mathcal{G}_t$ satisfies all the conditions of 
     Lemma~\ref{lemma:llt1}, which implies that there exists a coupling $\Pp$ of $(A_{t+1}^-, N_{t+1})$ such that $\Pp(|A_{t+1}^- - N_{t+1}|> 1) = O(n^{-1/8})$, where $N_{t+1} \sim \mathcal{N}(\E[A^-_{t+1} \mid {X}_t], \var(A^-_{t+1} \mid{X}_t))$.
     In addition, we want to show that $B_{t+1}$ and ${N}_{t+1}$ are typically close, where $B_{t+1}$ is the random variable used in generation of $Y_{t+1}$. 
        Clearly, $\E[A_{t+1}^- \mid{X}_t] = \frac{n-L_t}{q}$, and  
         note that $\var(A_{t+1}^- \mid {X}_t) = \frac{1}{q} \big(1-\frac{1}{q}\big) R_2^-({X}_t)$.  
    Hence, the inductive assumption ensures $\left|\E[A_{t+1}^- \mid {X}_t] - \E[B_{t+1} \mid Y_t] \right| \le 1$, and since we're on the event $\mathcal G_t$, 
    $ | \var(A_{t+1}^- \mid {X}_t) - \sigma_t^2| \le \sqrt{n} \log^2 n$.
   We appeal to the following standard bound on the TV-distance between 1-dimensional Gaussians, which is an easy calculation (see e.g., the univariate case of~\cite{DMR23} which is focused on the more difficult multivariate case):
    if $N_X \sim \mathcal{N}(\mu_X, \sigma_X^2)$ and $N_Y\sim \mathcal{N}(\mu_Y, \sigma_Y^2)$. Then
    \begin{align}\label{eq:tvdifference}
    \|N_X - N_Y \|_{\TV} \le \frac{3|\sigma_X^2 - \sigma_Y^2|}{2\sigma^2_Y} 
        + \frac{|\mu_X - \mu_Y|}{2\sigma_Y}. 
    \end{align}
    Applying~\eqref{eq:tvdifference} 
    \begin{align*}
        \|B_{t+1} - {N}_{t+1}\|_{\TV} & \le \frac{3  | \var(A_{t+1}^- \mid {X}_t) - \sigma_t^2|}{2\var(A^-_{t+1} \mid {X}_t)} + \frac{\left|\E[A^-_{t+1} \mid{X}_t] - \E[B_{t+1} \mid Y_t] \right|}{2\sqrt{\var(A^-_{t+1} \mid {X}_t)}} \\
        & =O\left( \frac{\sqrt{n} \log^2 n}{n} \right) = O\left( \frac{\log^2 n}{\sqrt{n}} \right).
    \end{align*}
    So by the optimal coupling lemma, 
     there exists a coupling $\Pp_{t+1}$ of $(A^-_{t+1}, {N}_{t+1}, B_{t+1})$ such that
     \[
     \Pp_{t+1}\big(\{B_{t+1} \neq {N}_{t+1}\} \cup \{| A^-_{t+1} - {N}_{t+1}| > 1 \} \mid \mathcal{F}_t \big)
     = O(n^{-1/8}).
     \]     
    From now on, we assume $| B_{t+1} - {A}^-_{t+1}| \le 1$ under the coupling $\Pp_{t+1}$.
    The inductive assumption implies 
    \[
        | (L_t + A^-_{t+1}) - (Y_t + B_{t+1}) | \le 2.
    \]
      On the percolation step, Lemma~\ref{lemma:llt2} implies that  
      there exists a coupling of $(|\mathcal{L}_1(G_{t+1})|, Y_{t+1})$ such that 
      $|Y_{t+1} - |\mathcal{L}_1(G_{t+1})|| \le 1$ with probability $1-o(1)$.
    Since $|\mathcal{L}_2({X}_t)| = O(\log n)$ and the giant component has been activated,  with high probability the percolation step is super-critical and, $L_{t+1} = |\mathcal{L}_1(G_{t+1})|$.
      
    Therefore, it follows a union bound over all the probabilistic estimates used so far that the coupling of $({X}_{t+1}, Y_{t+1})$ satisfies all the desired properties with probability $1-o(1)$.
\end{proof}

\begin{proof}[\textbf{\emph{Proof of Lemma~\ref{lem:ZapproxY2}}}]
    Under the identity coupling of the initial normal random variables, with probability one, $Y_0 = Z_0 + \theta_*n$.
    For $t\ge 0$, assume $|Y_t - (Z_t + \theta_* n)| =O(\log n)$ and $t<\tilde{\tau}$. 
    Now we couple the step $(Y_{t+1}, Z_{t+1})$.
    When $\varepsilon_{t+1} =0$, 
    we couple $(Y_{t+1}, Z_{t+1})$ such that both $Y_{t+1}$ and $Z_{t+1}$ stay idle. Now assume $\varepsilon_{t+1} = 1$.
    Let $\Theta_t := Y_t/n$.
    Taking the two sub-steps of $\{Y_t\}$ into one step, we have
    \begin{align*}
    Y_{t+1} &= \alpha\Big(\betac \big( \tfrac{1}{q} + \tfrac{q-1}{q}\Theta_t +\tfrac{W_{t+1}}{n}  \big) \Big)  \cdot 
    \Big(  \tfrac{n}{q} + \tfrac{q-1}{q} Y_t + W_{t+1} \Big) \\
    & \qquad + \mathcal{N}\Big(0, \sigma^2\Big(\betac \big( \tfrac{1}{q} + \tfrac{q-1}{q}\Theta_t +\tfrac{W_{t+1}}{n}  \big) \Big)  \cdot  
    \Big(  \tfrac{n}{q} + \tfrac{q-1}{q} Y_t + W_{t+1} \Big) \Big)\,,
    \end{align*}
    where $W_{t+1} \sim \mathcal{N}(0,\sigma_t^2)$.
    Let $\tilde{k}(\theta) = \ka(\theta) + W_{t+1}/n$. Then 
    \begin{equation}
        \label{eq:Yeq3}
            Y_{t+1} = \alpha(\betac \cdot \tilde{k}(\Theta_t)) \cdot \tilde{k}(\Theta_t)n + \mathcal{N}\big( 0, \sigma^2(\betac \cdot \tilde{k}(\Theta_t)) \cdot \tilde{k}(\Theta_t)n \big).
    \end{equation}
    By Lemma~\ref{lemma:Skmoment} and Observation~\ref{obs:Xr}, 
    $\sigma_t^2 = \Theta(n)$, 
    so for the rest of the proof, we assume $W_{t+1} = O(\sqrt{n \log n})$, which happens with probability $1-n^{-2}$ for every $t$. 
    We first Taylor expand $\alpha(\betac \tilde k(\Theta_t))$ about $\betac \ka(\Theta_t)$. Using twice-differentiability of $\alpha$, we get
    \begin{equation}
        \label{eq:betaTaylor}
        \alpha(\betac \cdot \tilde{k}(\Theta_t)) = \alpha(\betac \cdot \ka(\Theta_t)) 
        + \alpha'(\betac \cdot \ka(\Theta_t))  \cdot \frac{W_{t+1}}{n} + O\left(\frac{\log n}{n}\right).
    \end{equation}
    Multiplying this with $\tilde k(\Theta_t) n = \ka(\Theta_t) n + W_{t+1}$, we get 
    \begin{align*}
        \alpha(\betac  \cdot \tilde{k}(\Theta_t)) \cdot \tilde{k}(\Theta_t)n = \phi(\Theta_t) n + h_1(d_t) W_{t+1} + O(\log n),  
    \end{align*}
where 
\begin{align*}
    \phi(\theta):=\alpha(\betac \ka(\theta))\ka(\theta) \qquad \text{and} \qquad d_t := \betac \ka(\Theta_t).
\end{align*}
    Let $d_*:= \betac \ka(\theta_*)$.
    Since $h_1'$ is a bounded function and $|\Theta_t - \theta_*| \le \frac{\gamma}{\sqrt{n}}$, $|h_1(d_t) - h_1(d_*)| = O(\frac{\gamma}{\sqrt{n}})$.    Hence, by~\eqref{eq:tvdifference}, we can couple $h_1(d_t) \cdot \mathcal{N}(0,\sigma_t^2)$ and $h_1(d_*) \cdot \mathcal{N}(0,\sigma_t^2)$ with probability at least  
    \[
        1 - \| \mathcal{N}(0,h_1(d_t)^2\sigma_t^2) - \mathcal{N}(0,h_1(d_*)^2\sigma_t^2) \|_{\TV}
        \ge 
        1 -\frac{|h_1^2(d_t) - h_1^2(d_*)| \sigma_t^2}{2h_1(d_*)\sigma_t^2} = 1- O\left(\frac{\gamma}{\sqrt{n}}\right).
    \]
    Hence, with probability $1-O(n^{-1/2})$ we can replace $ h_1(d_t) W_{t+1} $ by $h_1(d_*) W_{t+1}$.
    
    We proceed with Taylor expansion on $f(\cdot)$ about $\theta_*$.
    By noting that $f(\theta_*) = 0$, $f'(\theta_*) = O(1), f''(\theta_*) = O(1)$ and $(\Theta_t - \theta_*)^2n \le \gamma^2$,
    we have
    \begin{align*}
        \phi(\Theta_t) n &= [\Theta_t + f(\Theta_t)] n  = \Theta_t n + [f(\theta_*) + f'(\theta_*)(\Theta_t - \theta_*) + f''(\theta_*) 
        (\Theta_t - \theta_*)^2]n \\
        &= \theta_* n + (1 + f'(\theta_*))(\Theta_t - \theta_*)n + O(1).
    \end{align*}
    
    Therefore, the following equation holds with probability $1-O(n^{-1/2})$,
    \begin{equation}
        \label{eq:drifttheta}
        \alpha(\betac \cdot  \tilde{k}(\Theta_t)) \cdot \tilde{k}(\Theta_t)n 
        =  \theta_* n +  (1 + f'(\theta_*))(\Theta_t - \theta_*)n + h_1(d_*) W_{t+1} + O(\log n).
    \end{equation}
    Finally we handle the normal random variable in \eqref{eq:Yeq3}.
    Let $\tilde{d}_*:=\betac \tilde{k}(\theta_*)$ and $\tilde{d}_t := \betac \cdot \tilde{k}(\Theta_t)$.
    Since $h_2'$ is a bounded function and $|\Theta_t - \theta_*| \le \frac{\gamma}{\sqrt{n}}$,
    $|h_2(\tilde{d}_t) - h_2(\tilde{d}_*)| = O(\frac{1}{\sqrt{n}})$.
    Again, since $h_2'$ is a bounded function and $|\tilde{d}_* - d_*| = O\big(\sqrt{\frac{\log n}{n}}\big)$,    
    $
     | h_2(\tilde{d}_*) - h_2(d_*)| = O\big(\sqrt{\frac{\log n}{n}}\big).
    $
    By triangle inequality, 
    \[
         | h_2(\tilde{d}_t) - h_2(d_*)| \le  | h_2(\tilde{d}_*) - h_2(d_*)| + |h_2(\tilde{d}_t) - h_2(\tilde{d}_*)|
         = O\big(\sqrt{\tfrac{\log n}{n}}\big).
    \]
    Hence, it follows from~\eqref{eq:tvdifference} that we can couple $\mathcal{N}( 0,h_2(\tilde{d}_t) n )$ 
    and $\mathcal{N}\left( 0, h_2(d_*)n \right)$ to agree with probability $1-O\big(\sqrt{\frac{\log n}{n}}\big)$. 
    The result follows from \eqref{eq:Yeq3}, \eqref{eq:drifttheta} and this coupling.
\end{proof}

    \subsubsection{Analysis of the limiting 1-dimensional process}
    Now that we have shown the giant component process near the fixed point $\theta_* n$ is well-approximated by the 1-dimensional process $Z_t$ of~\eqref{eq:Z-t-process}, we show  here that this simplified process's exit probabilities to the right and left are monotone, and oscillate on the $\sqrt{n}$ scale. 
   In what follows, let $\tau_\gamma^Z$ be the hitting time of $\gamma \sqrt{n}$ for $Z_t$ from~\eqref{eq:Z-t-process} and let $\tau_{-\gamma}^Z$ be the hitting time of $-\gamma\sqrt{n}$.
    
    \begin{lemma}\label{lem:monotonicity-of-Z-process}
        For every $p\in (0,1)$, there is a unique $c_*\in \mathbb R$ such that if $p_0 = \lambda_* + c_* n^{-1/2} + o(n^{-1/2})$ and  $Z_0 \sim \mathcal N((p_0 - \lambda_*)\alpha'(\lambda_*) n,\sigma^2(\lambda_*) n)$  then $$\Pr(\tau_\gamma^Z <\tau_{-\gamma}^Z) = p + o_{\gamma,n}(1)\,,$$
        where $o_{\gamma,n}(1)$ means it goes to zero either as $n\to\infty$ or as $\gamma\to\infty$. 
    \end{lemma}
    \begin{proof}
    We begin by showing that  the Markov chain of~\eqref{eq:Z-t-process} is monotone in the initialization.
    Namely, we wish to show that for two initializations $Z_0\le Z_0'$, the law of $Z_t$ is stochastically below that of $Z_t'$ for all $t$. Suppose that $Z_{t} \le Z_t'$ and consider $Z_{t+1}$ and $Z_{t+1}'$ generated per~\eqref{eq:Z-t-process} using the same Bernoulli random variable $\varepsilon_{t+1}$ and the same pair of normal random variables. If $\varepsilon_{t+1}=0$ then $Z_{t+1} = Z_{t}\le Z_{t}' = Z_{t+1}'$ and the monotonicity is preserved. If $\varepsilon_{t+1} = 1$, then under this coupling $$Z_{t+1}' - Z_{t+1} = (Z_{t}'- Z_t)(1 + f'(\theta_*))\,.$$ This will be positive because $Z_{t}' - Z_t\ge 0$ and $f'(\theta_*)>0$. The monotonicity in the initialization then carries over to monotonicity in $p_0$ because $\mathcal N(\mu,\sigma^2)\succeq \mathcal N(\mu',\sigma^2)$ if $\mu \ge \mu'$.

    To see the other consequences, notice first that the process $(\bar Z_{t})_{t\ge 0} = (Z_{t}n^{-1/2})_{t\ge 0}$ is $n$-independent (as all of $h_1, h_2, \sigma_t, f', \betau$ are $n$-independent). This implies that $\Pr(\tau_{\gamma}^{Z} <\tau_{-\gamma}^{Z})$ are $n$-independent from $n$-independent initializations, which will be the case if $p_0 = \lambda_* + c_* n^{-1/2}$. 
    As $Z_0 \to \gamma\sqrt{n}$, the probability $\Pr(\tau_{\gamma}^Z<\tau_{-\gamma}^{Z})$ is easily checked to go to $1$, and as $Z_0 \to -\gamma \sqrt{n}$, it goes to zero. This implies existence of a unique $c_*(\gamma)$ such that from the initialization $Z_0$, one has $\mathbb P(\tau_{\gamma}^{Z}<\tau_{-\gamma}^{Z}) = p + o(1)$, so long as we show continuity of $\mathbb P(\tau_{\gamma}^{Z}<\tau_{-\gamma}^{Z})$ in $c_*$. To see that continuity, we show that the total-variation distance between initializations with $p_0$ and $p_0^\delta = p_0 + \delta n^{-1/2}$ goes to zero as $\delta \downarrow 0$, as there would then be a coupling of the initializations (and therefore the future of the chains including the indicators of which direction it exits) that succeeds with probability going to $1$ as $\delta \downarrow 0$. If $Z_0^\delta$ has the Gaussian initialization with $p_0$ replaced by $p_0^\delta$, then a total-variation bound between 1-dimensional Gaussians with the same variance gives 
    \begin{align*}
        \|\mathbb P(Z_0 \in \cdot ) - \mathbb P(Z_0^\delta\in \cdot )\|_{\TV} \le \frac{\delta n^{-1/2} \alpha'(\lambda_*) n }{2 \sigma (\lambda_*) n^{1/2}}= \delta \frac{\alpha'(\lambda_*)}{\sigma(\lambda_*)}\,,
    \end{align*}
    which clearly is going to $0$ as $\delta \downarrow 0$ as desired.

    Finally, to see that $c_*$ is $\gamma$-independent up to the $o_{\gamma}(1)$ error, similarly observe that $$|\Pr(\tau_{\gamma}^{Z}<\tau_{-\gamma}^{Z}) - \Pr(\tau_{(1+f'(\theta_*)/2)\gamma}^{Z} <\tau_{-(1+f'(\theta_*)/2)\gamma}^Z)| = o_\gamma(1)\,.$$
    Together, these yield the lemma. 
    \end{proof}

    We can use the monotonicity above, together with the closeness of the $Z_t$ process with $X_t$ to translate the right initialization from $(Z_t)_{t}$ to $(X_t)_t$. 

    \begin{lemma}\label{lem:Z-X-hitting-probability-comparison}
    Suppose $X_0 \sim G(n,p_0/n)$ for $p_0 = \lambda_* + O(n^{-1/2})$ and let $(Z_t)_{t}$ be the process of Definition~\ref{def:Z-t-process} initialized from $Z_0 \sim \mathcal N((p_0 - \lambda_*)\alpha'(\lambda_*)n, \sigma^2(\lambda_*)n)$.
        Then the hitting probabilities $\Pr(\tau_\gamma^Z<\tau_{-\gamma}^Z)$ and $\Pr(\tau_\gamma^X<\tau_{-\gamma}^X)$ are within $o_{\gamma,n}(1)$ of one another. 
    \end{lemma}
    \begin{proof}
    By Theorem~\ref{thm:1dapproxZ}, there exists a coupling such that the processes $Z_t$ and $L_t - \theta_*n$ are within $O(\log n)$ for all $O(1)$ times, except with $o(1)$ probability. Now fix any $\varepsilon>0$ admissible difference between the probabilities in the lemma. For any such $\varepsilon$, there exists a $\gamma$ such that uniformly over $t$, if $L_t - \theta_*n$ is in $\gamma\sqrt{n} - O(\log n)$ the probability that in the next step it exceeds $\gamma\sqrt{n}$ is at least $1-\varepsilon$; the uniformity over $t$ uses the fact that the variance $\sigma_t^2$ is uniformly bounded, while the drift increases linearly with $\gamma$. This implies that if $Z_t$ has hit $\gamma\sqrt{n}$ then with probability $1-\varepsilon-o(1)$, the process $L_t - \theta_*n$ either has hit $\gamma\sqrt{n}$, or will hit it in the next step. 

    The same holds for hitting $-\gamma\sqrt{n}$, as well as for the converse implications, i.e., that if $X$ has hit one side or the other then so has $Z$ or will in the next step. Putting these together, we deduce that for every $\varepsilon$, there exists $\gamma$ sufficiently large such that  
$|\Pr(\tau_\gamma^Z<\tau_{-\gamma}^Z) - \Pr(\tau_\gamma^X<\tau_{-\gamma}^X)| \le \varepsilon + o(1)$.    \end{proof}

\subsection{Lower bound on the mixing time with different choices of \texorpdfstring{$\lambda_0$}{lambda0}}
By combining the above quasi-equilibration results with slow mixing results of~\cite{GLP} for the CM dynamics, we show that if the initialization is the product measure with parameters not satisfying the conditions of Theorem~\ref{thm:SWCM-dynamics}, then mixing is slow.  

\begin{theorem}\label{thm:slow-mixing-CM}
    For every $q>2$ and $\beta \in (\betau, \betas)$, if $\lambda_*(\beta,q)$ and $c_*(q)$ are as in Theorem~\ref{thm:SWCM-dynamics}, then the CM dynamics initialized from $\bigotimes \text{Ber}(\lambda_0/n)$ with 
    \begin{enumerate}
        \item $\beta \in (\betau,\betac)$ and $\lambda_0  > \lambda_*(\beta,q)  - O(n^{-1/2})$,
        \item $\beta = \beta_c$ and $\lambda_0 \ne \lambda_*(\beta,q) + c_*(q) n^{-1/2} + o(n^{-1/2})$,
        \item $\beta\in (\betac,\betas)$ and $\lambda_0 < \lambda_*(\beta,q) + O(n^{-1/2})$,
    \end{enumerate}
    takes $\exp(\Omega(n))$ time to reach $o(1)$ TV-distance to stationarity. 
\end{theorem}
\begin{proof}
   We provide the details for item (1), the other cases following by similar reasoning. For any initialization parameter $\lambda_0 > \lambda_* - K n^{-1/2}$ for some $K = O(1)$, by Lemma~\ref{lem:monotonicity-of-Z-process} and~\ref{lem:Z-X-hitting-probability-comparison}, there is a positive probability $c_K>0$ that the process $L_t$ hits $\theta_* + \gamma n^{-1/2}$ before $\theta_* - \gamma n^{-1/2}$ (for sufficiently large $\gamma$) in some $t\le C_\gamma$ many steps. By Lemma~\ref{lem:Gt}, with probability $1- o(1)$, at exit, the configuration satisfies the necessary conditions to apply Lemma~\ref{lem:quasi-equilibration-away-from-saddle-SWCM} and quasi-equilibrate to the right to the ordered phase $\mu^\ord$. Putting these together, we find that for some $T_0 = O(\log n)$,
    \begin{align}
        \|\mathbb P(X_{T_0} \in \cdot ) - \mu^\ord \|_\TV  \le 1- c_K + o_{\gamma,n}(1)\,. 
    \end{align}
    By using the optimal coupling on these, and then the identity coupling of CM chains after time $T_0$,  
    for $t\ge 0$ we have 
    \begin{align}\label{eq:coupled-to-ord-after-T_0}
        \|\mathbb P(X_{t+T_0} \in \cdot ) - \mathbb P_{\mu^\ord}(X_{t}\in \cdot) \|_\TV \le 1-c_K +o_{\gamma,n}(1)\,,
    \end{align}
    where the first chain here is initialized from the product measure, and the second from $\mu^\ord$. 
    
    Next, we claim that by the results of~\cite{GLP}, a CM dynamics chain initialized from $\mu^\ord$ retains total-variation distance $1-o(1)$ for exponentially many steps to $\mu$ when $\beta \in (\betau,\betac)$. To see this, note from \cite[Lemma 4.7 and the proof of Theorem 2 in the subcritical/critical regime]{GLP} the existence of a bottleneck set $A$ (that the giant component is at least $(\theta_* +\epsilon)n$ and the number of vertices in non-giant components of size larger than $M$ is at most $\rho n$) such that it takes $\exp(\Omega(n))$ steps for a CM dynamics initialized from $\mu(\cdot \mid A)$ to leave $A$. The initialization from $\mu(\cdot \mid A)$ can be seen to be within $e^{ - \Omega(n)}$ total-variation distance of an initialization from $\mu^\ord$ by an application of Lemma~\ref{lem:equilibrium-estimates} and Observation~\ref{obs:tv-distance-conditioning}. From this, we deduce that for every $t \ge 1$, 
    \begin{align}\label{eq:A-likely-under-ord-start}
        \mathbb P_{\mu^\ord}(X_t \in A) \ge 1- t e^{ - \Omega(n)}\,.
    \end{align}
    By definition of total-variation distance,~\eqref{eq:coupled-to-ord-after-T_0} together with~\eqref{eq:A-likely-under-ord-start}, implies 
    \begin{align*}
        \mathbb P(X_{t+T_0}\in A) \ge c_K - o_{\gamma,n}(1) - te^{-\Omega(n)}\,.
    \end{align*}
    On the other hand, by Lemma~\ref{lem:equilibrium-estimates}, when $\beta \in (\betau, \betac)$, one has $\mu(A) =o(1)$, so the above bound implies that for $\gamma$ large, and some $t= e^{\Omega(n)}$, the total-variation to $\mu$ is at least $c_K/2$, say. 
\end{proof}

\subsection{Proof of Theorem~\ref{thm:SWCM-dynamics}}\label{subsec:proof-of-main-SWCM}

With the above ingredients at hand, we are in position to complete the proof of Theorem~\ref{thm:SWCM-dynamics}. 

\begin{proof}[\textbf{\emph{Proof of Theorem~\ref{thm:SWCM-dynamics}}}]
We begin by concluding the bounds of mixing when initialized on the appropriate side of the unstable fixed point, taking care of all off-critical portions of Theorem~\ref{thm:SWCM-dynamics}. 

For item 1, when $\beta\in (\betau,\betac)$, we let $\lambda_0 = \lambda_* - \omega(n^{-1/2})$, where $\lambda_*$ is the solution to $\alpha(\lambda_*) = \theta_*$; by differentiability of $\alpha$ with strictly positive derivative in the super-critical regime, if $\lambda_0 = \lambda_* - \omega(n^{-1/2})$, then $\alpha(\lambda_0) = \theta_* - \omega(n^{-1/2})$. In particular, by Lemma~\ref{lem:rgdeviation}, $X_0$ satisfies the conditions of Lemma~\ref{lem:quasi-equilibration-away-from-saddle-SWCM} with $\gamma = \omega(1)$. Furthermore, by Lemma~\ref{lem:equilibrium-estimates} and Observation~\ref{obs:tv-distance-conditioning}, when $\beta<\betac$, we have $\|\mu^\dis - \mu\|_{\TV} = o(1)$. The result follows by the triangle inequality. 
For item 3, the proof is symmetrical, with the observation that when $\beta >\betac$, we have $\|\mu^\ord - \mu\|_{\TV} = o(1)$.  

 We now turn to the mixing time at the critical point $\beta  = \betac$. Fix an $\epsilon$ total-variation distance we are trying to achieve, and in turn take $\gamma$ sufficiently large. By Lemma~\ref{lem:equilibrium-estimates}, the stationary distribution $\mu$ is a $\xi,1-\xi$ mixture of $\mu^\dis$ and $\mu^\ord$ for $\xi$ defined in~\eqref{eq:xi}. By Lemma~\ref{lem:monotonicity-of-Z-process}, there is a unique $c_{*}(\xi)$ such that the escape probabilities of Lemma~\ref{lem:monotonicity-of-Z-process} are within $o(1)$ of $\xi, 1-\xi$.  By Lemma~\ref{lem:Z-X-hitting-probability-comparison}, if $X_0 \sim G(n,\lambda_0)$ with $\lambda_0= \lambda_* + c_* n^{-1/2} + o(n^{-1/2})$, then $|\Pr(\tau_{\gamma}^X <\tau_{-\gamma}^X)  - \xi| \le \epsilon + o(1)$. 
Moreover, by Lemma~\ref{lemma:exittime}, the minimum of these two exit times is $O(1)$. Finally, by Lemma~\ref{lem:Gt}, with probability $1- o(1)$, at exit, the configuration satisfies the necessary conditions to apply Lemma~\ref{lem:quasi-equilibration-away-from-saddle-SWCM} and quasi-equilibrate to the phase-restricted measure on the side the dynamics exits with probability $1-O(\gamma^{-2})$. Combined, these imply there exists $T = O(\log n)$ such that if $X_0 \sim G(n,\lambda_0)$, then 
\begin{align*}
    \|\Pr (X_T \in \cdot) - ((1-\xi) \mu^\ord + \xi\mu^\dis)\|_{\TV} \le O(\gamma^{-2}) + \epsilon +o(1)\,,
\end{align*}
which will be less than $2\epsilon$ for $n$ large and $\gamma$ large. 

The slow mixing claims are exactly the statement of Theorem~\ref{thm:slow-mixing-CM}.
\end{proof}

We also include a proof of Theorem~\ref{thm:simulated-annealing-CM} on initializations from the random-cluster Gibbs measure at a different temperature, to justify why it follows from the above arguments.

\begin{proof}[\textbf{\emph{Proof of Theorem~\ref{thm:simulated-annealing-CM}}}]
    It is evident at this point that the only properties of the initialization used in our proof of Theorem~\ref{thm:SWCM-dynamics} were on the size of its giant component in relation to $\theta_* n$, its second largest component size, and the sum of squares of its non-giant components. The requisite properties on the latter two quantities are known to hold for samples from $\mu_{\beta_0}$ per Lemma~\ref{lem:equilibrium-estimates}. The requisite property of the giant component size drawn from $\mu_{\beta_0}$ is that it should be on the side of $\theta_*$ to which it aims to quasi-equilibrate. 
    
    In item 1 when $\beta \in (\betau,\betac)$, we let $b_*(\beta,q) = \betac$.
    To see this, first note that if $\beta_0< \betac$, then the size of the largest component in a sample from $\mu_{\beta_0}$ is $O(\log n)$ with high probability and from such initialization the CM dynamics will mix in $O(\log n)$ steps. 
    On the other hand, if $\beta_0 > \betac$, then a typical sample from $\mu_{\beta_0}$ has the largest component of size close to $\thetar(\beta_0)n$, where we recall from Lemma~\ref{lem:equilibrium-estimates} and Lemma~\ref{lem:factsaboutdrift} that $\thetar(b)$ is the typical size of a giant in a sample from $\mu_{b}$ as well as the second root of $f$. 
    By monotonicity, $\thetar(\beta_0) > \thetar(\beta) > \theta_*(\beta)$, 
    and hence the CM dynamics initialized from $\mu_{\beta_0}$ has exponential slow mixing.
    In item 2, 
    if $\beta_0 >\betac$ then $\thetar(\beta_0)>\thetar(\betac)>\theta_*(\beta_c)$ by monotonicity;
    also, if $\beta_0<\beta_c$ then the CM dynamics will start with a configuration from $\mu_{\beta_0}$, which has no large components.
    In either case, the size of largest component of the starting configuration is not close to $\theta_*(\beta)$, so
    there is no fast mixing.
    Finally, in item 3 when $\beta \in (\betac,\betas)$, we define $b_*(\beta,q) = \inf\{b: \thetar(b)>\theta_*(\beta)\}$ such that $O(\log n)$ mixing occurs if and only if $\beta_0 > b_*(\beta,q)$.
\end{proof}

\subsection{Deferred proofs: concentration and local limit theorem for the activation step}
We now include proofs of concentration of the activation steps and local limit theorems  that were deferred. 

\subsubsection{Concentration properties of activation steps}
We begin by describing some easy estimates on the activation step of the CM dynamics. 
For a graph $X$, let $X^r$ be the sub-graph of $X$ that does not get activated in $r$ activation steps.
\begin{observation}     
    \label{obs:Xr}
 Suppose $X\sim G(n,p)$. For any integer $r\ge 0$,
\begin{align*}
     \mathbb E[R_2(X^r)] = \big(1-\tfrac{1}{q}\big)^{r} \mathbb E[R_2(X)], \qquad
    \mathbb E[R_2^-(X^r)] = \big(1-\tfrac{1}{q}\big)^{r} \mathbb E[R_2^-(X)].
 \end{align*}
\end{observation}
\begin{proof}
            Let $\mathcal{B}_1, \dots, \mathcal{B}_n$ be independent Bernoulli random variables with parameter $(1-\frac{1}{q})^{r}$.
    Firstly, 
    \begin{equation}
        \label{eq:condexp}
        \E[R_2(X^r) \mid X]=
        \E\Big[ \sum_{j\ge 1} |\mathcal L_j(X)|^2 \cdot \mathcal{B}_j \mid X\Big]
        = \sum_{j\ge 1} \E\big[ |\mathcal L_j(X)|^2 \cdot \mathcal{B}_j \mid X \big] 
        =  \big(1-\frac{1}{q}\big)^{r} R_2(X).
    \end{equation}
    The first claim of the observation then follows by taking expected values and the second claim is analogous. 
\end{proof}

\begin{proof}[\textbf{\emph{Proof of Lemma~\ref{lem:concentration-of-G^r}}}]
        Let $\mathcal{B}_1, \dots, \mathcal{B}_n$ be i.i.d.\ Bernoulli random variables with parameter $(1-\frac{1}{q})^{r}$.
    First we compute a conditional variance. 
    \[
        \var\big( R_2^-(X^r) \mid X \big) 
        = \sum_{j\ge 2} \var (|\mathcal L_j(X)|^2 \cdot \mathcal{B}_j \mid X)
        \le \sum_{j\ge 2} |\mathcal L_j(X)|^2 
        = R_2^-(X).
    \]
    Then, by the law of total variance and Observation~\ref{obs:Xr}, we obtain
    \begin{align}
        \var\left( R_2^-(X^r) \right) 
        &= \E\left[ \var\left( R_2^-(X^r) \mid X \right) \right]
        + \var\left(        \E[R_2^-(X^r) \mid X]  \right) \nonumber \\
        & \le \E[R_2^-(X) ] + \var( R_2^-(X)). \label{eq:lotvGit}
    \end{align}
    If $n\cdot p > 1$ uniformly in $n$, then by Corollary~\ref{cor:R2moment} and \eqref{eq:lotvGit}, the right-hand side is $O(n)$ and the result follows by Chebyshev's inequality. 
    The case $np<1$ uniformly in $n$ follows analogous reasoning. 
\end{proof}

We also prove Lemma~\ref{lem:close2gstar} showing the approximability of $\mathbb E[R_2^-(X^r)]$ by $\mathbb E[R_2(G_*^r)]$. 
\begin{proof}[\textbf{\emph{Proof of Lemma~\ref{lem:close2gstar}}}]
    First, observe that $\tilde \beta:= (n\ka(\theta_*) +m) \cdot \frac{\betac}{n} > 1$ uniformly in $n$.  
    By  Lemma~\ref{lem:rgdeviation}, 
    \[
        \Pr\big(||\mathcal L_1(G)| - \alpha(\tilde \beta) (n\ka(\theta_*) +m) | > \sqrt{n}(\log n)^{2/3}\big) \le \frac{1}{2n^{10}}\,.
    \]
    Moreover, for each $U$ such that $||U| - \alpha(\tilde \beta) (n\ka(\theta_*) +m) | \le \sqrt{n}(\log n)^{2/3}$, Lemma~\ref{lemma:duality} implies if $U = \mathcal{L}_1(G)$ then 
    $G\setminus \mathcal{L}_1(G)$ can be coupled with $G_- \sim G(n\ka(\theta_*) + m - |\mathcal L_1(G)|,  \frac{\betac}{n})$ with probability $1 - e^{-\Omega(n)}$.
    Hence, with probability at least $1-n^{-10}$,
    \[
        m':=||\mathcal L_1(G)| - \alpha(\tilde \beta) (n\ka(\theta_*) +m) | \le \sqrt{n}(\log n)^{2/3}\,,
    \]
    and $G\setminus \mathcal{L}_1(G) = G_-$. 
    It follows from Observation~\ref{obs:Xr} that
    \begin{equation}
        \label{eq:removehead}
         \left| \E[R_2^-(G^r)] - \E[R_2(G_-^r)] \right| \le
        \left| \E[R_2^-(G)] - \E[R_2(G_-)] \right| = O(n^{-8})\,.
    \end{equation}
    Next we show when $m' \le\sqrt{n}(\log n)^{2/3}$,
    \begin{equation}
        \label{eq:nearstar}
         \left| \E[R_2(G^r_*)] - \E[R_2(G_-^r)] \right| = O(\sqrt{n} \log n)\,,
    \end{equation}
    and the lemma follows from \eqref{eq:removehead} and \eqref{eq:nearstar}.
    To see \eqref{eq:nearstar}, we give an estimate of the number of vertices $M$ in $G_-$.
    From the arguments above, we know
    \[
        M = n\ka(\theta_*) + m -\alpha(\tilde \beta) (n\ka(\theta_*) +m) + m'\,.
    \]
    By algebraic manipulation and Taylor expansion of $\alpha$, we obtain
    \begin{align*}
        M 
        &= n\kia(\theta_*) + (\theta_*n - \alpha(\tilde \beta)\cdot \ka(\theta_*)n) + O(\sqrt{n}(\log n)^{2/3}) \\
        & =n\kia(\theta_*) + [\theta_*n - \phi(\theta_*)n + m \betac \alpha'(\ka(\theta_*) \cdot \betac)\cdot \ka(\theta_*) + o(\sqrt{n})] + O(\sqrt{n} (\log n)^{2/3})\,,
    \end{align*}
    which gives~\eqref{eq:nearstar} since $\phi(\theta_*) = \theta_*$. 
\end{proof}

\noindent Finally, we use the variance of the activation steps to get a bound on the variance of the giant component after one step of the CM dynamics.  

\begin{fact}
    \label{lem:uniformvarbound}
    For $\beta \in (\betau, \betas)$, there exist constants $M_0 , M_1 > 0$ and $s>0$ such that if 
    a configuration $X_t$ satisfies that
    $L_t \in (\theta_*n -sn, \theta_*n+sn)$ and
    $R_2^-({X}_t) \le M_1 n$, then
    \[
     \var(L_{t+1} \mid {X}_t, \Lambda_{t+1})\le M_0^2 n
    \]
\end{fact}
\begin{proof}
    Let $G_{t+1}$ be the random graph in the percolation step of step $t+1$.
    In other words, $G_{t+1} \sim G(A_{t+1}, \beta/n)$.
    By the law of total variance, we have
    \begin{align}
        \var(L_{t+1} \mid & {X}_t,  \Lambda_{t+1}) \nonumber \\
        & = \E[\var(|\mathcal L_1(G_{t+1})| \mid A_{t+1}) \mid {X}_t, \Lambda_{t+1}]
        + \var(\E[|\mathcal L_1(G_{t+1})| \mid A_{t+1}] \mid {X}_t, \Lambda_{t+1})\,. \label{eq:lotvarXt} 
    \end{align}
    If $A_{t+1} \cdot \frac{\beta}{n} > 1$ uniformly in $n$, then by Lemma~\ref{lem:rgdeviation}, we have $\var(|\mathcal L_1(G_{t+1})| \mid A_{t+1}) \le M_2 n$,
    where $M_2$ depends only on $A_{t+1}$.
    By our assumption and the computation in \eqref{eq:varA1}, $\var(A_{t+1} \mid X_t, \Lambda_{t+1}) \le R_2^-({X}_t) \le  M_1 n$.
    Thus, by Chebyshev's inequality, 
    $A_{t+1}$ concentrates around its mean $\ka(L_t/n)$ with sufficiently small $\delta = \delta(M_1,s)>0$ deviation, 
    with probability $1-M_3 n^{-1}$, where $M_3$ depends only on $\delta$ and $M_1$. 
         Also, we know that for $\beta\in (\betau,\betas)$, $\ka(\theta_*)\beta > \ka(\thetas) \beta \ge 1$.
    Thus, by continuity, for small enough $s$, $\ka(L_t/n)\beta \ge k_a(\theta_*-s)\beta >1$,
    and $A_{t+1}$ is such that the percolation step is strictly supercritical.  
    Hence, we obtain
    \begin{equation}
        \label{eq:varXtterma}
         \E[\var(|\mathcal L_1(G_{t+1})| \mid A_{t+1}) \mid {X}_t, \Lambda_{t+1}]
         \le (1 - M_3 n^{-1}) \cdot M_2 n + M_3 n,
    \end{equation}
    where the $M_3 n$ contribution comes from the $M_3 n^{-1}$ probability event that 
    $A_{t+1} \cdot \frac{\beta}{n} > 1$ does not hold uniformly in $n$.
    
    Next we upper bound the second term in right-hand-side of \eqref{eq:lotvarXt}.
    By \eqref{eq:expectedgiantgnp}, so long as $A_{t+1}\beta/n$ is bounded away from $1$, we have 
    \[
        \E[|\mathcal L_1(G_{t+1})| \mid  A_{t+1} ] = \alpha\big(A_{t+1} \cdot \tfrac{\beta}{n}\big) \cdot A_{t+1} + \tilde{O}(1).
    \]
    Using that $0< \alpha(\cdot) < 1$, and taking the variance of the above, we get 
    \begin{align*}
        \var(\E[|\mathcal L_1(G_{t+1})| \mid A_{t+1}] \mid {X}_t, \Lambda_{t+1}) \le \text{Var}(A_{t+1} \mid X_{t},\Lambda_{t+1}) + M_3n + \tilde{O}(1)
    \end{align*}  As already claimed, the conditional variance is at most $R_2(X_t) \le M_1 n$. 
Putting the above bounds together, we conclude. 
\end{proof}

\subsubsection{Local limit theorem for the number of activated vertices}
We start with the necessary local limit theorem for the activation step of the CM dynamics. 
\begin{lemma}
\label{lemma:llt1}
    Let $q\ge 2$.
    Suppose $X$ is a graph satisfying that
    \begin{enumerate}
      \item $|\mathcal{L}_1(X)|  = \Omega(n)$;
        \item $R_2^-(X) = \Theta(n)$;
        \item $I_1(X) = \Omega(n)$;
        \item $R_3^-(X)= O(n)$.
    \end{enumerate}
     Let $A^-$ be the number of non-giant activated vertices of $X$, i.e., 
     $$A^- = \sum_{i\ge 2} |\mathcal L_i(X)| \cdot B_i, \qquad B_i \sim \text{Ber}(1/q) \quad \text{independently}.$$ 
    Then there exists a coupling $\Pp$ of $(A^-, Z)$ such that
    $\Pp(|A^- - Z| > 1) = O(n^{-1/8})$, where $Z \sim \mathcal{N}(\E[A^-], \var(A^-))$.
\end{lemma}
We prove Lemma~\ref{lemma:llt1} by showing that it fits the criteria of the following classical local limit theorem.  

\label{sec:llt}
\begin{lemma}[\cite{Petrov75}]
    \label{lem:petrov}
    Let $X_1, \dots, X_n$ be independent integer-valued random variables with mean $\mu_1, \dots, \mu_n$ and let $S_n = \sum_{j=1}^n X_j$.
    Let $\mu$ and $\sigma^2$ be the mean and variance of $S_n$. 
    Suppose the following conditions hold:
    \begin{enumerate}
        \item $\sigma^2 \rightarrow \infty$ as $n \rightarrow \infty$.
        \item
        $    \sum_{j=1}^n \E[ | X_j - \mu_j |^3 ] = O(\sigma^2)$. 
        \item For all $j$ and $r \neq 0$, $\Pr(X_j = 0) \ge \Pr(X_j = r)$.
        \item $\gcd\{M\in \mathbb Z: \frac{1}{\log n} \sum_{j=1}^n \Pr(X_j = 0) \Pr(X_j = M) \rightarrow \infty \text{ as } n \rightarrow \infty\} = 1$. 
    \end{enumerate}
    Then there exists a universal constant $C_1$ such that for $k \in \mathbbm{Z}$  we have
    \[
    \big| \Pr(S_n = k) - \frac{1}{\sigma \sqrt{2\pi}} e^{-\frac{(k-\mu)^2}{2\sigma^2}} \big| \le \frac{C_1}{\sigma^2}. 
    \]
\end{lemma}
\begin{proof}[\textbf{\emph{Proof of Lemma~\ref{lemma:llt1}}}]
    Suppose $X$ has $m+1$ components, where $m=\Theta(n)$ by assumption 3.
    For $j= 1, \dots, m$, let $V_j = |\mathcal{L}_{j+1}(X)|$ if $\mathcal{L}_{j+1}(X)$ is activated (when $B_{j+1} = 1$), and let $V_j = 0$ otherwise.
    Note that $A^- = \sum_{j=1}^m V_j$.
    Let $\mu$ and $\sigma^2$ be the mean and variance of $A^-$.
    
    To apply Lemma~\ref{lem:petrov}, we verify its conditions hold in our setting.
    First, the fact that the variance of $A^-$ goes to infinity follows from our second supposition that $R_2^-(X) = \Theta(n)\to\infty$. 
    Item 2 of Lemma~\ref{lem:petrov} follows from our fourth supposition that $R_3^-(X) = O(n)$. 
    Moreover, by our definition of $V_j$,  for each $j = 1, \dots, m$, we have for any $r\neq 0$,
    \[
        \Pr(V_j = 0) = 1 - \frac{1}{q} \ge \frac{1}{2} \ge  \frac{1}{q} \ge \Pr(V_j = r)\,.
    \] Finally, we analyze the fourth condition in Lemma~\ref{lem:petrov}.
    Note that in our case, if $M\in \mathbbm{Z}$,
    \begin{align*}
    \frac{1}{\log m} \sum_{j=1}^m \Pr(V_j = 0) \Pr(V_j = M)
    & = \frac{1}{\log m} \sum_{j=1}^m \left[ \big( 1 - \frac{1}{q} \big) \frac{1}{q} \cdot \mathbf{1}[L_{j+1}(X) = M] \right] \\ 
    & = O\left(\frac{1}{\log m}\right) \sum_{j=1}^m  \mathbf{1}[L_{j+1}(X) = M].
    \end{align*}
    The third assumption of the current lemma states that for $M=1$, we have $ \sum_{j=1}^m\mathbf{1}[L_{j+1}(X) = M] = \Omega(n)$.
    Hence, as $m\rightarrow \infty$, for $M=1$, 
    \[
     \frac{1}{\log m} \sum_{j=1}^m \Pr(V_j = 0) \Pr(V_j = M) = \Omega\left(\frac{m}{\log m}\right) \rightarrow \infty\,.
    \]
    Since the $\gcd$ of any number with $1$ is $1$, we get the fourth condition.   
    Therefore, Lemma~\ref{lem:petrov} implies that for each $k \in \mathbbm{Z}$,
    \[
    \epsilon_n(k):=\big| \Pr(A^- = k) - \frac{1}{\sigma \sqrt{2\pi}} e^{-\frac{(k-\mu)^2}{2\sigma^2}} \big| \le \frac{C_1}{\sigma^2} = O\big(\frac{1}{n}\big). 
    \]
    In particular, for each integer $k \in [\mu - \sigma \cdot n^{1/4}, \mu + \sigma \cdot n^{1/4}]$, we have $\epsilon_n(k) = O(1/n)$.
    Also,  for each $k \in \mathbbm{Z}$, by integrating the normal density, we obtain
    \begin{align}\label{eq:normal-rv-integer-approx}
        \Pr\big(k-\tfrac{1}{2} \le Z \le k+\tfrac{1}{2} \big) 
        =  \frac{1}{\sigma \sqrt{2\pi}} e^{-\frac{(k-\mu)^2}{2\sigma^2}} + O\big(\frac{1}{\sigma^2} \big)\,.
    \end{align}
    Also, the probability that it doesn't lie in $[\mu- \sigma n^{1/4}, \mu+\sigma n^{1/4}]$ is at most $n^{-1/8}$, say, by Chebyshev's inequality. 
    Using the fact that $\sigma^2 = \Theta(n)$, this implies that there exists a coupling $\Pp$ of $(A^-, Z)$ such that $\Pp(|A^- - Z|>1) = O(n^{-1/8})$
    as desired.
\end{proof}

The next lemma provides us with the necessary local limit theorem in the percolation step. 

\begin{lemma}
\label{lemma:llt2}
    Suppose that $G\sim G(m, \frac{\betac}{n})$ where $m=\Omega(n)$ satisfying $m \cdot \frac{\betac}{n} > 1$ is bounded away from $1$.
    Let $m' \in \mathbb R$ such that $|m - m'| = O(1)$.
    Suppose $Y\sim \mathcal{N}(\mu_Y, \sigma_Y^2)$, where $\mu_Y:=\alpha(\frac{\betac m'}{n}) \cdot m'$ and $\sigma_Y^2:=\sigma^2(\frac{\betac m'}{n}) \cdot m'$.
    Then there exists a coupling of $(Y, L(G))$ such that  $\Pr(|Y - |\mathcal L_1(G)|| \le 1) =1-o(1)$.
\end{lemma}
\begin{proof}
    Let $\delta > 0$ be an arbitrary number. 
    By definition,  $G\sim G(m, \frac{\betac}{n})$.
    Let $\mu_X = \alpha\left(\betac \cdot \frac{m}{n}\right)\cdot m$, and 
    $\sigma_X^2 = \sigma^2\left(\betac \cdot  \frac{m}{n} \right) \cdot m$.  
    Suppose $\mathcal{W} \sim \mathcal{N}(\mu_X, \sigma_X^2)$.
    Since $m\cdot \betac > n$ uniformly in $n$, Lemma~\ref{thm:giantLLT} shows that for each $k \in [\mu- L\sigma_X, \mu+L\sigma_X]$ and any $L>0$,
    \[
    \epsilon_G(k):=
    \Big| \Pr(|\mathcal L_1(G)| = k) -  \frac{1}{\sigma_X \sqrt{2\pi}} e^{-\frac{(k-\mu_X)^2}{2\sigma_X^2}} \Big| 
    \le \frac{\delta}{4\sqrt{2\pi}\sigma_X L} . 
    \]
    Then by~\eqref{eq:normal-rv-integer-approx} and Chebyshev's inequality, for $L$ and $m$ large, there is a coupling of  $(|\mathcal L_1(G)|, \mathcal{W})$ such that
    \[
    \Pr\big(| \mathcal{W} - |\mathcal L_1(G)|| >1  \big) \le \frac{2\sigma_X L \delta}{4\sqrt{2\pi}\sigma_X L} + O\big(\frac{2L\sigma_X}{\sigma_X^2} \big) + \frac{\delta}{4}  \le \frac{\delta}{2}.
    \]
    On the other hand, via our assumption on $|m-m'|$ and~\ref{eq:tvdifference}, we obtain that
    \[
        \| Y - \mathcal{W} \|_{\TV} \le \frac{3|\sigma_X^2 - \sigma_Y^2|}{2\sigma^2_Y} 
        + \frac{|\mu_X - \mu_Y|}{2\sigma_Y} = O(n^{-1/2})\,.
    \]
    Therefore, for every $\delta>0$, there exists a coupling of $(\mathcal L_1(G), \mathcal{W}, Y)$ such that
    $|Y - |\mathcal L_1(G)|| \le 1$ with probability at least $1-\delta$.
\end{proof}

\section{The Potts Glauber dynamics}\label{sec:Potts-dynamics}

Throughout this section, let $\mathcal S$ be the simplex $\mathcal{S} := \{x\in [0,1]^q: x_1 + \dots +x_q =1 \}$.
For $s\in \mathcal{S}$, we denote by $s_i$ the $i$-th coordinate of $s$.
Let $\sigma_t \in \Omega$ be the Potts configuration at the $t$-th step of the Glauber dynamics.
Let $S(\sigma_t)= (S_{t,1}, S_{t,2},\dots, S_{t,q}) \in \mathcal{S}$ be the proportion vector of $\sigma_t$
such that there are $n S_{t,i}$ spins of color $i\in [q]$ in $\sigma_t$, i.e.,
\[
    S_{t,k} = \frac{1}{n} \sum_{v\in [n]}\mathbf{1}\{\sigma_t(v)=k\}\,.
\]
We denote by $\{S_t\}_t := \{S(\sigma_t)\}_t$ this Markov chain on the state space $\mathcal{S} \cap \frac{1}{n}\mathbb Z^q$. 

\subsection{Preliminaries for the Potts Glauber dynamics}
For any $\beta \ge 0$, define $g_{\beta}: \mathcal S \to \mathcal S$ as 
\begin{align*}
    g_{\beta}(s) = (g_{\beta,1}(s),...,g_{\beta,q}(s)) \qquad \text{where} \qquad g_{\beta,k}(s):= \frac{e^{\beta\cdot s_k}}{\sum_{j=1}^q e^{\beta \cdot s_j}}\,.
\end{align*}
This vector approximates the expected proportion vector after $1$ step of Glauber dynamics initialized at $s$. Namely, the drift satisfies 
\begin{align}\label{eq:proportions-drift}
    \mathbb E[S_{t+1,k} - S_{t,k} \mid \mathcal F_t] = \frac{1}{n} g_{\beta,k}(S_t) -\frac{1}{n} S_{t,k} + O(n^{-2})\,;
\end{align}
see Eq.~(3.1) of~\cite{CDLLPS-mean-field-Potts-Glauber}.
Without loss of generality, we will be taking the first coordinate as a distinguished one tracking the dominant color class (when there is one). Given this, it is natural to define a drift function for the first coordinate, 
\begin{equation}
    \label{eq:driftdbeta}
    d_\beta(s):=g_{\beta,1}(s)-s_1\,.
\end{equation}
We also define $D_\beta:[0,1] \rightarrow \mathbb{R}$ as
\begin{equation}
    \label{eq:driftD}
    D_{\beta}(x):= \max_{s:s_1=x} d_\beta(s) = d_\beta\big(x, \tfrac{1-x}{q-1}, \dots, \tfrac{1-x}{q-1}\big)\,. 
\end{equation}
Equivalently, we can express \eqref{eq:driftD} as
\begin{equation}
    \label{eq:driftD2}
            D_\beta(x) = \Big(1+(q-1) \exp{ \big( \beta\cdot \tfrac{1-qx}{q-1}  \big)}\Big)^{-1} -x\,.
\end{equation}
It is easy to see that $D_\beta$ is a continuously differentiable function with derivative, 
    \begin{equation}
    \label{eq:driftderivative}
               D_\beta'(x):= \frac{d}{dx} D_\beta(x) = \frac{q\beta\exp{ \left( \beta\cdot \frac{1-qx}{q-1}  \right)} }{\Big(1+(q-1) \cdot \exp{ \big( \beta\cdot \frac{1-qx}{q-1}  \big)}\Big)^2} -1\,.
    \end{equation}
The following lemma characterizes the roots of $D_\beta$ when $\beta\in (\betau, \betas)$.
\begin{lemma}
    \label{lem:driftanalysis}
    Suppose $q>2$ and $\beta\in (\betau, \betas)$. 
    Then $D_\beta(\frac{1}{q}) = 0$; there are exactly two solutions for $D_\beta(x) = 0$ in $(\frac{1}{q}, 1)$, denoted by $m_*$ and $\mr$, where $m_* < \mr$. 
    Moreover, $D_\beta'(m_*)> 0$ and $D_\beta'(\mr) < 0$.
\end{lemma}
\begin{proof}
    The definition of $\betau$ is equivalent to $
        \betau=\sup\{
            \beta\ge 0: D_\beta(x) < 0, \forall x \in (1/q,1]
        \}$. From \eqref{eq:driftD2}, we know that if $1/q<x\le 1$, $D_\beta(x)$ for a fixed $x$ is a strictly increasing in $\beta$.
Hence,  by continuity of $D_\beta$, for $\beta>\betau$, 
    there exists $x^+\in (1/q,1]$ such that $D_\beta(x^+) > 0$.
    Next, by direct computation, for any $\beta\ge 0$ we have
    \[
        D_\beta(\tfrac{1}{q}) = 
        \frac{1}{1+(q-1) \cdot \exp{ \left[ \beta\cdot 0  \right]}} - \frac{1}{q} = 0\,,
 \qquad 
        D_\beta(1) = 
         \frac{1}{1+(q-1) \cdot \exp{ ( -\beta )}} - \frac{1}{q} < 0\,,
    \]
    and if $\beta<q = \betas$ then 
    \[
        \frac{d}{dx} D_\beta(x)\big\vert_{x=1/q} = \frac{q\beta }{\left[1+(q-1) \right]^2} -1 
        = \frac{\beta}{q}-1<0\,.
    \]
    This means for a sufficiently small $\epsilon>0$, $D_\beta\big(\frac{1}{q}+\epsilon\big) < 0$ and $\frac{1}{q} + \epsilon < x^+$.
    Since $D_\beta\big(\frac{1}{q}+\epsilon\big) < 0$, $D_\beta(x^+) > 0$ and $D_\beta(1)<0$, 
    by continuity of $D_\beta$, there are at least two roots for  $D_\beta(x) = 0$ in $(\frac{1}{q}, 1)$, among which there are two roots $m_* < \mr$ such that $D_\beta'(m_*) >0$, $D_\beta'(\mr) < 0$.

    Lastly we show that $m_*$ and $\mr$ are exactly the two roots for  $D_\beta(x) = 0$ in $(0,\infty)$.
    To see this, note that roots of  $D_\beta(x) = 0$ are roots of 
    \[
        \frac{1}{1-\psi(x)} = x\,, \qquad \text{where} \qquad  \psi(x)=(1-q)\cdot \exp  \left(\beta\cdot \frac{1-qx}{q-1} \right)\,.
    \]
    Since $\psi(x)$ is strictly monotone in $x$, there are at most two such zeros.
\end{proof}

The last preliminary estimates we require in order to analyze the Potts Glauber dynamics are the following two helpful sub/super-martingale concentration estimates we borrow from~\cite{CDLLPS-mean-field-Potts-Glauber}. 
 
\begin{lemma}[Lemma 2.1 in \cite{CDLLPS-mean-field-Potts-Glauber}]
\label{lem:mtghittingtime}
    For $x_0\in \mathbb R$, let $(X_t)_{t\ge 0}$ be a discrete time process initialized from $x_0$, with law $\Pp_{x_0}$, adapted to $(\mathcal{F}_t)_{t\ge 0}$, and satisfying
    \begin{enumerate}
        \item $\exists \delta\ge 0$: $\E_{x_0}[X_{t+1} - X_t \mid \mathcal{F}_t] \le -\delta $ on $\{X_t \ge 0\}$ for all $t\ge 0$.
        \item $\exists R >0: |X_{t+1} - X_t|\le R$ for all $t\ge 0$.
    \end{enumerate}
    Let $\tau_x^-=\inf\{ t:X_t \le x\}$ and $\tau_x^+ = \inf\{ t: X_t \ge x\}$.
    The following holds.
    \begin{enumerate}
        \item If $\delta>0$ then for any $t_1 \ge 0$:
        \[
            \Pp_{x_0}(\tau_0^- > t_1) \le \exp\Big( 
                    -\frac{(\delta t_1 - x_0)^2}{8t_1R^2}
            \Big)\,.
        \]  
        \item  If $x_0 \le 0$ then for any $x_2>0$ and $t_2 \ge 0$,
        \[
            \Pp_{x_0}(\tau_{x_2}^{+} \le t_2) \le 2\exp\Big(- \frac{(x_2-R)^2}{8t_2 R^2} \Big)\,.
        \]
        \item If $x_0 \le 0, \delta>0$ then for any $x_1>0$ and $t_1 \ge 0$,
        \[
            \Pp_{x_0}(\tau_{x_1}^{+} \le t_3) \le t_3^2\exp\Big(- \frac{(x_1-R)\delta^2}{8R^3} \Big)\,.
        \]
    \end{enumerate}
\end{lemma}

\begin{lemma}[Lemma 2.2 in \cite{CDLLPS-mean-field-Potts-Glauber}]
    \label{lem:supermtganticoncen}
    Let $\{X_t\}_{t\ge 0}$ be a process adapted to $\{\mathcal{F}_t\}_{t\ge 0}$ and satisfying the following conditions for some $0\le 2\delta <a$:
    \begin{enumerate}
        \item $X_{t+1} - X_t \in \{-1,0,1\}$.
        \item $\E[ X_{t+1} - X_t \mid \mathcal{F}_t] \ge -\delta$.
        \item $\var(X_{t+1} \mid \mathcal{F}_t)\ge a$.
        \item $X_0 \ge 0$.
    \end{enumerate}
    Let $\tau_r^+ = \inf\{ t: X_t \ge r\}$. Then
    \[
        \Pr(\tau_r^+ \le t) \ge C_1 \exp \Big( -C_2 \Big(\frac{r}{\sqrt{t}} +\delta\sqrt{t} \Big)^2 \Big) + O(t^{-1/2}),
    \] where $C_1, C_2$ are positive constants which depend only on $a$.
\end{lemma}

\subsection{Mixing away from the saddle when \texorpdfstring{$\beta \in (\betau,\betas)$}{beta in (betau,betas)}}
We first establish that as soon as the proportion vector has reached an $\omega(n^{-1/2})$ distance from the unstable fixed point of $(m_*, \frac{1-m_*}{q-1},...,\frac{1-m_*}{q-1})$ in the first coordinate, with probability going to $1$ as $n\to\infty$, 
the Glauber dynamics rapidly quasi-equilibrates to the corresponding phase. We need to handle the regimes $\beta\in (\betau,\betas)$ and $\beta >\betas$ separately, with the latter introducing additional complications; this subsection is focused on the former. 
Let $\zeta_n=o(1)$ be a sequence going to zero sufficiently slowly (chosen after everything else), and define the $q+1$ stable Potts phases when $\beta\in (\betau,\betas)$ as 
\begin{align*}
    \Omega^\dis & = \{\sigma: \|S(\sigma) - (\tfrac{1}{q},...,\tfrac{1}{q})\|_1\le \zeta_n\}\,, \\
    \Omega^{\ord,i} & = \{\sigma: \|S(\sigma) - (\tfrac{1-\mr}{q-1},...,\tfrac{1-\mr}{q-1},\mr,\tfrac{1-\mr}{q-1},...,\tfrac{1-\mr}{q-1})\|_1 \le \zeta_n\}\,, \qquad \text{and} \\
    \Omega^{\ord} & = \Omega^{\ord,1} \cup \dots \cup \Omega^{\ord,q},   
\end{align*}
where in $\Omega^{\ord,i}$, $\mr$ is the $i$'th coordinate. Let $\pi^\dis = \pi(\cdot \mid \Omega^\dis)$, $\pi^{\ord,i} = \pi(\cdot \mid \Omega^{\ord,i})$ and  $\pi^{\ord} = \pi(\cdot \mid \Omega^{\ord})$.  

The initialization $\hat \nu^{\otimes}(m_0)$ 
in Theorem~\ref{thm:Potts-Glauber}
    has one distinguished coordinate (randomly chosen on $\{1,...,q\}$) which at least at time zero is the dominant color.  By permutation symmetry, it is sufficient for us to assume that it is the first coordinate and we do so in what follows. 

\begin{theorem}
    \label{thm:PottsHitLeftTarget}
    Let $q > 2$ and $\beta \in (\betau, \betas)$.
    Let $\gamma>0$ be a large constant.
    Suppose $\sigma_0$ is a configuration such that
    $S_{0,1} = m$ for $m > 1/q$, $m\le m_* - \gamma n^{-1/2}$
    and $\|S_0 - (m, \frac{1-m}{q-1},...,\frac{1-m}{q-1})\|_1 = O(n^{-1/2}\log n)$.
    Suppose $\sigma_0' \sim \pi^{dis}$.
    Then there exists $T=O(n \log n)$ such that
    \begin{align*}
        \|\Pr_{\sigma_0} (S(\sigma_T) \in \cdot) - \Pr_{\sigma_0'} (S(\sigma_T') \in \cdot) \|_{\TV} = e^{-\Omega(\gamma^2)}\,.
    \end{align*}
    An analogous statement holds with $\sigma_0$ and $\sigma_0'$ such that 
       $S_{0,1} = m$ for $m\ge m_* + \gamma n^{-1/2}$, 
    $\|S_0 - (m, \frac{1-m}{q-1},...,\frac{1-m}{q-1})\|_1 = O(n^{-1/2}\log n)$
    and $\sigma_0 ' \sim \pi^{\ord,1}$. 
\end{theorem}

The following allows us to treat the $q-1$ non-dominant coordinates as constant $\frac{1-S_{t,1}}{q-1}$ and approximate the analysis by a 1-dimensional process for the first coordinate.

\begin{lemma} 
    \label{lem:nondominantcoordinates}
    Let $q > 2$ and $\beta <\betas$.
    For any $\epsilon>0$, let $\tau_\epsilon:=\inf\{t: S_{t,1} \le \frac{1}{q}+\frac{\epsilon}{2}\}$.
    Suppose $\|S_0 - (m, \frac{1-m}{q-1},...,\frac{1-m}{q-1})\|_1 = O(n^{-1/2}\cdot \log n)$ for some $m>\frac{1}{q}+ \epsilon$.
    Then
    for any $T = A_1 n\log n$, there exists $A_2 >0$ such that 
    $\max_{i,j\ne 1} |S_{T\wedge\tau_{\epsilon},i} - S_{T\wedge \tau_{\epsilon},j}| < \frac{2A_2 \log n}{\sqrt{n}}$ 
    with probability $1-O(n^{-2})$.
\end{lemma}

\begin{proof}
    Let $K>0$ be a large constant such that $\max_{i,j\ne 1}|S_{0,i} - S_{0,j}| \le Kn^{-1/2}\log n$.
    Let $T= A_1 n\log n$, and define the stopping times 
    $$\tau_{i,j}^+ = \min\Big\{ t: S_{t,i} - S_{t,j} -\tfrac{K \log n}{\sqrt{n}} \ge \tfrac{A_2 \log n}{\sqrt{n}}\Big\}\,, \qquad \text{and}\qquad \tau^+= \min_{i,j\ne 1} \tau_{i,j}^+\,.$$
    Our aim is to show existence of a large constant $A_2>K$ such that the probability that $\{\tau^+ < T \wedge \tau_\epsilon\}$ is at most $O(n^{-2})$. 

    Fix any pair of distinct $(i,j)$ where $i, j \neq 1$.
    Let $W_t := S_{t\wedge\tau^+\wedge\tau_\epsilon ,i} - S_{t\wedge \tau^+ \wedge \tau_\epsilon,j} -  Kn^{-1/2} \log n$ for all $t\ge 0$.
    Clearly $W_0\le 0$ and $|W_{t+1} - W_t| \le \frac{2}{n}$ for all $t\ge 0$.
    We will show that on $\{W_t \ge 0\}$, if $t< \tau^+\wedge \tau_\epsilon$,
    then there exists a constant  $\eta=\eta(\epsilon, q, \beta)>0$  such that
    \begin{equation}
        \label{eq:Wdrift}
        \E[W_{t+1} - W_t \mid \mathcal{F}_t] 
        \le - \frac{\eta  Kn^{-1/2}\log n}{2n},
    \end{equation}
    and $ \E[W_{t+1} - W_t \mid \mathcal{F}_t] = 0$ if $t\ge \tau^+\wedge\tau_\epsilon$.
    Given those, item (2) of Lemma~\ref{lem:mtghittingtime} would imply that 
    \begin{align*}
        \Pr(\tau_{i,j}^{+} \le T\wedge \tau_\epsilon) \le 2 \exp\Big(- \frac{(\frac{A_2\log n}{\sqrt{n}} - 2/n)^2}{8T\cdot (2/n)^2} \Big)
        & \le  2 \exp\Bigg(
            - \frac{ \big(\frac{ A_2 \log n}{2\sqrt{n}}\big)^2  }
            {  \frac{2A_1n \log n}{n^2}  }    
        \Bigg) \\
        & = 2 \exp\Bigg(
          - \frac{A_2^2 \log n}{8A_1} 
        \Bigg)\,.
    \end{align*}
    For sufficiently large $A_2$, this is at most $n^{-2}$. 
    By a union bound over all pairs of $(i,j)$ where $i,j\neq 1$, 
    $\Pr( \tau^+ \le T \wedge \tau_\epsilon) \le (q-1)^2/n^{2}$, concluding the proof.  
    To show~\eqref{eq:Wdrift}, recalling~\eqref{eq:proportions-drift}, note that for any $t< \tau^+\wedge\tau_\epsilon$, 
    \begin{equation}
        \label{eq:gcomparisons}
         \E[W_{t+1} - W_t \mid \mathcal{F}_t] 
         = \frac{1}{n} [(-S_{t,i} + g_{\beta,i}(S_t) ) - (-S_{t,j} + g_{\beta,j}(S_t))] + O(n^{-2}).
    \end{equation}
    In what follows, for a proportions vector $S_t$, define $$\hat{S}_t:=(S_{t,1}, \frac{1-S_{t,1}}{q-1}, \dots,  \frac{1-S_{t,1}}{q-1})\,.$$
    Since $|S_{t,i} - S_{t,j}| \le \frac{2A_2 \log n}{\sqrt{n}}$ for all $i,j \neq 1$  so long as $t < \tau^+$, $\|S_t - \hat{S}_t \|_2^2 = O(\frac{\log^2 n}{n})$, and so
     by Taylor expansion of $g_{\beta,i}(S_t)$ and $g_{\beta,j}(S_t)$ about $\hat S_t$, 
    \begin{align}
    g_{\beta,i}(S_t) -g_{\beta,j}(S_t)  &= g_{\beta,i}(\hat{S}_t) -  g_{\beta,j}(\hat{S}_t) + \langle S_t - \hat{S}_t, \nabla g_{\beta,i}(\hat{S}_t) - \nabla g_{\beta,j}(\hat{S}_t) \rangle 
        + O\left(\|S_t - \hat{S}_t \|_2^2\right) \nonumber \\
        &=  \langle S_t - \hat{S}_t, \nabla g_{\beta,i}(\hat{S}_t) - \nabla g_{\beta,j}(\hat{S}_t) \rangle  + O\left(\frac{\log^2 n}{n}\right) \nonumber \\
        &= (S_{t,i} -S_{t,j})\cdot 
            \Big( \frac{d}{dx_i}g_{\beta,i}\Big\vert_{x=\hat{S}_t} -  \frac{d}{dx_i}g_{\beta,j}\Big\vert_{x=\hat{S}_t}\Big)
            + O\left(\frac{\log^2 n}{n}\right).  \label{eq:g2derivative}
    \end{align}
    We now show there exists $\eta>0$ for which $\frac{d}{dx_i}g_{\beta,i}\vert_{x=\hat{S}_t} -  \frac{d}{dx_i}g_{\beta,j}\vert_{x=\hat{S}_t} < 1-\eta$ when $\beta < \betas$ and $t< \tau_{\epsilon}$.
    Indeed, when $\beta<\betas =q$, $i \neq 1$
    \begin{align*}
        \left(\frac{d}{dx_i}g_{\beta,i}\Big\vert_{x=\hat{S}_t} -  \frac{d}{dx_i}g_{\beta,j}\Big\vert_{x=\hat{S}_t} \right) 
        &
         < \frac{q e^{\beta\cdot \hat{S}_{t,i}}}{\sum_{k=1}^q e^{\beta \cdot \hat{S}_{t,k}}} \\
        &= 1 - \frac{e^{\beta\cdot \hat{S}_{t,1}} - e^{\beta\cdot \hat{S}_{t,i}}}{\sum_{k=1}^q e^{\beta \cdot \hat{S}_{t,k}}} \\
        &\le 1 - \frac{e^{\beta/q+\beta\epsilon/2} - e^{\beta/q-\epsilon\beta/(2q-2)}}{e^{\beta/q+\beta\epsilon/2}  + (q-1)e^{\beta/q-\epsilon\beta/(2q-2)}} =: 1-\eta.
    \end{align*}
    Hence, 
    \[
        \E[W_{t+1} - W_t \mid \mathcal{F}_t]= -\frac{\eta}{n}\cdot (S_{t,i} - S_{t,j}) +  O\Big(\frac{\log^2 n}{n^2}\Big).
    \]
    If $W_t \ge 0$, then the first term is at most $- \frac{\eta Kn^{-1/2}\log n}{n}$ and the second term is lower order, yielding~\eqref{eq:Wdrift}.  
\end{proof}

We now turn to the convergences to the disordered and $q$ ordered phases when the initialization is $\omega(n^{-1/2})$ away from the saddle point $(m_*, \frac{1-m_*}{q-1},...,\frac{1-m_*}{q-1})$. Their proofs are analogous; we'll present the full proofs for the former, then mention any modifications that need to be made in the other direction.  

\begin{lemma}
\label{lem:sdecayPotts}
Let $q > 2$ and $\beta \in (\betau, \betas)$.
   For a large constant $\gamma>0$, suppose there exists $m\le m_* - \gamma n^{-1/2}$ such that $S_{0,1} = m$  and 
    $\|S_0 - (m, \frac{1-m}{q-1},...,\frac{1-m}{q-1})\|_1 =O(n^{-1/2} \cdot \log n )$. 
    Then there exist $T= O(n \log n)$ and $s=\Omega(1)$ such that with probability $ 1 - e^{-\Omega(\gamma^2)}$,
        the hitting time of $\{S_{t,1} \le m_* - s\}$ is less than $T$.
\end{lemma}

\begin{proof}
    Lemma~\ref{lem:driftanalysis} implies that $c_*=D_\beta'(m_*)>0$. We set $s>0$ to be a sufficiently small constant to be chosen later.  
    
    We define several stopping times that will be useful: let $\tau_0 = 0$, 
    for $i\ge 0$, 
    $$\tau_{i+1} := \inf\left\{t\ge \tau_i : S_{t,1} < m_* - \Big(1 +\frac{c_*}{16} \Big)^{i+1}  \frac{\gamma}{\sqrt{n}} \right\},$$
    and 
    \[
        \tilde{\tau}_{i+1}:= \inf\left\{ t>\tau_i: S_{t,1} > m_* -  \Big(1 +\frac{c_*}{16} \Big)^{i}  \frac{\gamma}{2\sqrt{n}} \right\}. 
    \]
    Let $k=C_1 \log n$ be the least positive integer such that 
    $
        \gamma (1+\frac{c_*}{16})^{k}  > s\sqrt{n}.
    $
    Note that $\tau_{k+1} \ge \tau_k \ge \dots >\tau_1 \ge \tau_0$.
    If $\tau_k= \tau_0$, the lemma holds trivially. Thus we assume $\tau_k>\tau_0$ and
    will show that for all $i= 0,1,\dots, k-1$, 
    \begin{equation}
        \label{eq:UBtaui}
        \Pr\left( \tau_{i+1}<\tau_i + n\,,\, \tilde \tau_{i+1} > \min\{\tau_{k+1},\tau_i + n\} \mid S_{\tau_i}
        \right) 
        \ge 1 - 3r_i, 
    \end{equation}
    where $r_i := \exp(-C_2\gamma^2(1+c_*/16)^{2i})$ for some constant $C_2>0$.

    By averaging over $S_{\tau_i}$ and taking a union bound over $i=0,1,\dots, k$, 
    \[
        \Pr\Big( \bigcap_{i=0}^k \{\tau_i \le i\cdot n\} \Big)
        \ge 1 - \sum_{i=0}^{k-1}3r_i 
        \ge 1 - e^{-\Omega(\gamma^2)}.
    \]
    Since $S_{\tau_k,1} < m_* -s$, the lemma follows.
    
    Now we proceed to prove \eqref{eq:UBtaui} by showing that the two inequalities hold with probability $1-2r_i$ and $1-r_i$ respectively.
    First we show that 
    \begin{equation}
    \label{eq:LBtildetaui}
                \Pr(\tilde{\tau}_{i+1} \le \tau_{k+1} \wedge (\tau_i + n) \mid S_{\tau_i}) \le 2r_i.
    \end{equation}
    Consider the process $\{ Z_t\}_{t\ge 0}$ given by 
    \[
        Z_t:=S_{(t+\tau_i)\wedge\tilde{\tau}_{i+1} \wedge \tau_{k+1}, 1} - S_{\tau_i,1}\,.
    \]
    It can be verified that $Z_0 = 0$ and
    $|Z_{t+1} - Z_t| \le n^{-1}$ for all $t\ge 0$.
    We will also show that  
    \begin{equation}
        \label{eq:Ztdrift}
                \E_{0}[Z_{t+1} - Z_t \mid \mathcal{F}_t] \le 0\,,
    \end{equation}
    so  $\{ Z_t\}_{t\ge 0}$ satisfies all conditions of Lemma~\ref{lem:mtghittingtime}.
    We defer the proof of \eqref{eq:Ztdrift} momentarily and conclude the proof of \eqref{eq:LBtildetaui} using Lemma~\ref{lem:mtghittingtime}.
    By the second part of Lemma~\ref{lem:mtghittingtime}, 
    \begin{align*}
        \Pr(\tilde{\tau}_{i+1} \le (\tau_i + n)\wedge\tau_{k+1})
        \le \Pr(\tilde{\tau}_{i+1} \le \tau_i + n) & \le 2\exp\Bigg(
            - \frac{[(1+\frac{c_*}{16})^{i} \cdot \frac{\gamma}{2\sqrt{n}}  - \frac{1}{n}]^2}{8n\cdot n^{-2}}
        \Bigg)\\
        &\le 2\exp\Big( 
            -\frac{\gamma^2}{33} \cdot (1+\frac{c_*}{16})^{2i}
        \Big) \\ &\le \exp(-C_2\gamma^2(1+c_*/16)^{2i}) \\
        &= 2r_i\,,
    \end{align*}
    where the last inequality holds for a small $C_2>0$.
    
    To see \eqref{eq:Ztdrift}, we recall the drift function $d_\beta$.
    If $t+\tau_i \ge \tilde{\tau}_{i+1} \wedge \tau_{k+1}$, then $Z_{t+1} = Z_t$ and \eqref{eq:Ztdrift} clearly holds.
    It suffices to consider $t\ge 0$ such that $t+\tau_i < \tilde{\tau}_{i+1} \wedge \tau_{k+1}$.
    In this case, $Z_{t+1} - Z_t = S_{t+1+\tau_i,1} - S_{t+\tau_i,1}$, and so by~\eqref{eq:proportions-drift}--\eqref{eq:driftD} and Taylor expansion, 
    \begin{align}\label{eq:drift-step1}
                \E_0[Z_{t+1} - Z_t \mid S_{t+\tau_i}] 
                &
                \le  \frac{1}{n}D_\beta(S_{t+\tau_i,1}) + O(n^{-2})  \\
                &= \frac{1}{n}\cdot [D_\beta(m_*) + (S_{t+\tau_i,1} - m_*)\cdot c_* + O(|S_{t+\tau_i,1}- m_*|^2)] + O(n^{-2}) \nonumber \\
                & \le \frac{S_{t+\tau_i,1} - m_*}{2n} \cdot c^* <0  \label{eq:negativedrift7},
    \end{align}
    where the steps in the last line of \eqref{eq:negativedrift7} hold since $-s (1+\frac{c_*}{16})^2 < S_{t+\tau_i,1} - m_* <0$, and $s$ is sufficiently small in terms of $c_*$ that the second order term is at most $c_* (S_{t+\tau_i,1}- m_*)/2$, say. 

    Next we show the other inequality of \eqref{eq:UBtaui} using an auxiliary process.
    Consider the process ${S}_t'$ defined with  $S_{\tau_i} = {S}_{\tau_i}'$ but such that at any step $t\ge \tau_i$, ${S}_t'$ rejects the update at time $t+1$ if the resulting state would be such that
    ${S}_{t+1,1}' > m_* - (1+\frac{c_*}{16})^i\cdot \frac{\gamma}{2\sqrt{n}}$.
    We have ${S}_t' = S_t$ for all $t \in [\tau_i, \tilde{\tau}_{i+1})$.
    For ${S}_t'$ we use $\tau_{i}'$ for its corresponding analog of $\tau_{i}$.
    Let $\{W_t\}_{t\ge0}$ be the process given by 
    \[
        W_t:={S}_{t+\tau_i,1}' - \Big[m_* - \Big(1 +\frac{c_*}{16} \Big)^{i+1} \cdot \frac{\gamma}{\sqrt{n}} \Big].
    \]
    In addition, $W_0 = w_0>0$ and $|W_{t+1} - W_t|\le n^{-1}$ for all $t\ge 0$.
    Moreover, for all $t\ge0$ on $\{W_t\ge 0\}$,  
    by a bound analogous to \eqref{eq:negativedrift7}, we have
    \begin{align}\label{eq:drift-step2}
        \E_{w_0}[W_{t+1} - W_t\mid \mathcal{F}_t] 
        \le \frac{{S}_{t+\tau_i,1}' - m_*}{2n} \cdot c_* 
        \le  -(1+\frac{c_*}{16})^i\cdot \frac{\gamma}{2\sqrt{n}} \cdot \frac{c_*}{2n} =: -\delta.
    \end{align}
    Item (1) of Lemma~\ref{lem:mtghittingtime} implies that
    \begin{equation}
        \label{eq:UBtauiprime}
        \Pr(\tau_{i+1}' > \tau_i' + n \mid S_{\tau_i}')\le
         \exp\Big(-\frac{(\delta n - w_0)^2}{8n\cdot n^{-2}} \Big)
        \le \exp\Bigg(-\frac{C_3[(1+\frac{c_*}{16})^i\cdot \frac{\gamma}{\sqrt{n}}]^2}{8 n^{-1}} \Bigg)
        \le r_i\,,
    \end{equation}
    holds for suitable $C_2$ and $C_3$.
    In addition, for $i<k$, we know $\Pr( \tau_{i+1}' >  (\tau_i' + n)\wedge \tau_{k+1}' \mid S_{\tau_i}) = \Pr( \tau_{i+1}' >  \tau_i' + n \mid S_{\tau_i})$.
    
    Finally, \eqref{eq:UBtaui} follows from a union bound of 
    \eqref{eq:UBtauiprime} and \eqref{eq:LBtildetaui}. 
\end{proof}

By the same reasoning as in the proof of Lemma~\ref{lem:sdecayPotts}, replacing the drift function by a uniform bound that holds between $m_* - s$ and $\frac{1}{q}+\rho_0$ (by continuity and the fact that $\frac{1}{q},m_*$ are the only two zeros of $D_\beta$), we arrive at the following. (Since the proof is otherwise completely analogous, we omit it.)

\begin{lemma}
    \label{lem:sdecayPottsstrong}
    Let $q >2$ and $\beta \in (\betau, \betas)$.
      Suppose $\|S_0 - (m, \frac{1-m}{q-1},...,\frac{1-m}{q-1})\|_1 = O(n^{-1/2} \log n)$ for some $m\le m_* - s$ where $s=\Omega(1)$.
      Then with probability $1- e^{-\Omega(n)}$, the followings hold:
      \begin{enumerate}
          \item     For all $t = O(n \log n)$, $S_{t,1} \le m_* - \frac{s}{2}$.
          \item     For any constant $\rho_0 >0$, there exists $T=O(n\log n)$ such that
          $S_{T,1} \le \frac{1}{q}+\rho_0$.
      \end{enumerate}
\end{lemma}

Finally, it is known from the essential mixing results of~\cite{CDLLPS-mean-field-Potts-Glauber} that contraction for the distance to equiproportionality holds when a configuration starts close enough to being equiproportional and that mixing to the disordered phase follows from the equiproportionality. 

\begin{lemma}[Lemma 4.1 in \cite{CDLLPS-mean-field-Potts-Glauber}]
    \label{lem:contractionPotts}
    Suppose $q>2$ and $\beta < \betas$.
    There exists $\rho_0 = \rho_0(\beta, q)>0$ such that for all $r>0$:
    if $\|S_{0}\|_\infty \le \frac{1}{q} + \rho_0$,
    there exists a constant $\alpha>0$ such that 
    \[
        \Pr\Big(
            \big\|S_{{\alpha n \log n}} - (\tfrac{1}{q}, \dots, \tfrac{1}{q}) \big\|_1 > \frac{r}{\sqrt{n}}
        \Big) = O(r^{-1}),
    \]
\end{lemma}

\begin{lemma}[Corollary 4.4 and  Lemma 4.5 in \cite{CDLLPS-mean-field-Potts-Glauber}]
    \label{lem:subcriticalcouple}
    Suppose $q>2$, $r>0$ and $\beta < \betas$.
    Let $\{\sigma_t\}_{t\ge 0}$ and  $\{\sigma_t'\}_{t\ge 0}$ be two instances of Potts Glauber dynamics satisfying 
    \[
        \big\|S(\sigma_0) -(\tfrac{1}{q}, \dots, \tfrac{1}{q}) \big\|_1 \le  \frac{r}{\sqrt{n}}, \qquad \text{and} \qquad
        \big\|S(\sigma_0') - (\tfrac{1}{q}, \dots, \tfrac{1}{q}) \big\|_1 \le  \frac{r}{\sqrt{n}}.
    \]
    Then there exist a coupling $\Pp$ of $\{(\sigma_t, \sigma_t')\}_{t\ge 0}$ and $T=O(n)$ such that
    $\Pp(S(\sigma_T)  = S(\sigma_T')) = 1-o(1)$.
\end{lemma}

Combining the above results, we are ready to prove the first part of Theorem~\ref{thm:PottsHitLeftTarget} regarding the disordered phase.

\begin{proof}[\textbf{\emph{Proof of Theorem~\ref{thm:PottsHitLeftTarget}}: the $m\le m_* - \gamma n^{-1/2}$ case}]
  Suppose $\sigma_0$ is a configuration such that
    $S_{0,1} = m$ for $m\le m_* - \gamma n^{-1/2}$
    and $\|S_0 - (m, \frac{1-m}{q-1},...,\frac{1-m}{q-1})\|_1 = O(n^{-1/2}\log n)$.
    Let $\rho_0$ be as in Lemma~\ref{lem:contractionPotts}.
    It follows from combining Lemma~\ref{lem:nondominantcoordinates},  Lemma~\ref{lem:sdecayPotts}, Lemma~\ref{lem:sdecayPottsstrong} that with probability $1- e^{-\Omega(\gamma^2)}$
    there exists $T_1=O(n\log n)$ such that $\|S_{T_1}\|_\infty \le \frac{1}{q} + \rho_0$.
    Then Lemma~\ref{lem:contractionPotts} implies that 
    there exists $T_2=T_1+ \alpha n \log n$ such that with probability $1- e^{-\Omega(\gamma^2)}$, $\big\|S(\sigma_{T_2}) -(\tfrac{1}{q}, \dots, \tfrac{1}{q}) \big\|_1= O(n^{-1/2})$.
    The same holds for $\sigma_{T_2}'$, and therefore 
we can now conclude the first part of the theorem from that point by Lemma~\ref{lem:subcriticalcouple}. 
\end{proof}

In order to handle the case $m\le m_* - \gamma n^{-1/2}$, we require the following four lemmas. 

\begin{lemma}
\label{lem:superdecayPotts}
Let $q>2$ and $\beta \in (\betau, \betas)$.
  For a large constant $\gamma>0$, suppose there exists $m\ge m_* + \gamma n^{-1/2}$ such that $S_{0,1} = m$  and 
    $\|S_0 - (m, \frac{1-m}{q-1},...,\frac{1-m}{q-1})\|_1 =O(n^{-1/2} \cdot \log n )$. 
    Then there exists $T= O(n \log n)$ and $s=\Omega(1)$ such that         
    $S_{T,1} \ge m_* + s$ and 
    $\|S_T - (S_{T,1}, \frac{1-S_{T,1}}{q-1},...,\frac{1-S_{T,1}}{q-1})\|_1 =O(n^{-1/2} \cdot \log n )$
    with probability $ 1 - e^{-\Omega(\gamma^2)}$.
\end{lemma}

\begin{lemma}
    \label{lem:moveforwardright}
    Let $q>2$ and $\beta \in (\betau, \betas)$.
      Suppose $\|S_0 - (m, \frac{1-m}{q-1},...,\frac{1-m}{q-1})\|_1 = O(n^{-1/2}\log n)$ for some $m\ge m_* + s$ where $s=\Omega(1)$.
      Then with probability $1- n^{-\Omega(1)}$, the followings hold:
      \begin{enumerate}
          \item     For all $t = O(n \log n)$, $S_{t,1} \ge m_* + \frac{s}{2}$.
          \item     For any constant $\rho_1 >0$, there exists $T=O(n\log n)$ such that
          $S_{T,1} \ge \mr  - \rho_1$.
      \end{enumerate}
\end{lemma}

\begin{lemma}
    \label{lem:converge2mr}
    Suppose $q>2$ and $\beta > \betau$.
    There exists $\rho_1 = \rho_1(\beta, q)>0$ such that 
    if 
    $$\|S_{0} - (\mr, \tfrac{1-\mr}{q-1},...,\tfrac{1-\mr}{q-1} ) \|_\infty \le \rho_1 ,$$ 
    then for every $r>0$ there exists $T=O(n \log n)$ such that
    $$
        \mathbb P\Big(\big\|S_{T} - \big(\mr, \tfrac{1-\mr}{q-1},...,\tfrac{1-\mr}{q-1} \Big) \big\|_1 > \frac{r}{\sqrt{n}}\Big)  = O(r^{-1})\,.
    $$
\end{lemma}

\begin{lemma}
    \label{lem:supercriticalcouple}
    Suppose $q>2$, $r>0$ and $\beta > \betau$.
    Let $\{\sigma_t\}_{t\ge 0}$ and  $\{\sigma_t'\}_{t\ge 0}$ be two instances of Potts Glauber dynamics satisfying 
    \[
        \big\|S(\sigma_0) - \big(\mr, \tfrac{1-\mr}{q-1},...,\tfrac{1-\mr}{q-1} \Big) \big\|_1 \le  \frac{r}{\sqrt{n}} \quad \text{and} \quad
        \big\|S(\sigma_0') - \big(\mr, \tfrac{1-\mr}{q-1},...,\tfrac{1-\mr}{q-1} \Big) \big\|_1 \le  \frac{r}{\sqrt{n}}\,.
    \]
    Then there exist a coupling $\Pp$ of $\{(\sigma_t, \sigma_t')\}_{t\ge 0}$ and $T=O(n \log n)$ such that
    $\Pp(S(\sigma_T) = S(\sigma_T')) = 1-o(1)$.
\end{lemma}

The proofs of these four lemmas closely follow those of their analogs, Lemmas~\ref{lem:sdecayPotts}--\ref{lem:subcriticalcouple}.
For Lemma~\ref{lem:supercriticalcouple}, we note that the proof for Lemma~\ref{lem:subcriticalcouple} in~\cite{CDLLPS-mean-field-Potts-Glauber} only relies on the estimates of variance and drift around a stable fixed point, and it is not specific the disordered phase.
The main change to note is that for Lemma~\ref{lem:superdecayPotts} instead of maximizing the drift for the first coordinate by $D_\beta$ as done in~\eqref{eq:drift-step1}--\eqref{eq:drift-step2}, the drift is simply approximated by $D_\beta$ using its Taylor expansion and the fact that the other coordinates are roughly proportional to each other by Lemma~\ref{lem:nondominantcoordinates}. 
Due to this change, we include below the details for the proof of Lemma~\ref{lem:superdecayPotts}.

\begin{proof}[\textbf{\emph{Proof of Lemma~\ref{lem:superdecayPotts}}}]
    The proof is analogous to that of Lemma~\ref{lem:sdecayPotts}. 
    We borrow the definition of $c_*, s, \tau_0, k$ and $r_i$ from Lemma~\ref{lem:sdecayPotts}.
    and explain the main differences here.
    First, we introduce the notations that are new or different from the previous proof.
    For $i\ge 0$, 
    $$\tau_{i+1} := \inf\left\{t\ge \tau_i : S_{t,1} > m_* + \Big(1 +\frac{c_*}{16} \Big)^{i+1}  \frac{\gamma}{\sqrt{n}} \right\},$$
    \[
        \tilde{\tau}_{i+1}:= \inf\left\{ t>\tau_i: S_{t,1} < m_* +  \Big(1 +\frac{c_*}{16} \Big)^{i}  \frac{\gamma}{2\sqrt{n}} \right\}, 
    \]
    and
    \[
        \tau_{i+1}^+:=
        \inf\left\{t\ge \tau_i : \max_{j,l\neq 1} |S_{t,j} - S_{t,l}| > \frac{A\log n}{\sqrt{n}} \right\},
    \]
    where $A>0$ is a large constant depending on $S_{\tau_i}$.
    Note that $\tau_{k+1} \ge \tau_k \ge \dots \ge \tau_0$, and $\tilde{\tau}_i<\tau_\epsilon$, 
    where $\tau_{\epsilon}$ is as in Lemma~\ref{lem:nondominantcoordinates} with sufficiently small $\epsilon$.

    Assume $\tau_k>\tau_0$ again.
    In lieu of \eqref{eq:UBtaui},
    we shall prove that for all $i=0,\dots, k-1$, if $ \max_{j,l\neq 1} |S_{\tau_i,j} - S_{\tau_i,l}| \le \frac{K\log n}{\sqrt{n}}$, then there exists a constant $A>K$ such that 
    \begin{align}
        \Pr\left( 
            \tau_{i+1} \le \tau_{i}+n, 
            \tilde{\tau}_{i+1} > \min\{\tau_{k+1}, \tau_i +n \},
            \tau^+_{i+1} >  \min\{ \tau_{k+1}, \tau_i +n \}
            \mid  S_{\tau_i}
        \right) & \nonumber \\
        & \!\!\!\!\!\!\!\!\!\!\!\!\!\!\!\!\!\!\! \ge 1 - 3r_i - O(n^{-2})\,.  \label{eq:maingoal123}
    \end{align}
    Then the lemma follows from $\eqref{eq:maingoal123}$.
    We show \eqref{eq:maingoal123} in three steps:
    \begin{equation}
        \label{eq:decaysubgoal1}
        \Pr(\tau_{i+1}^+ \le (\tau_i +n)\wedge \tilde{\tau}_{i+1} | S_{\tau_i}) = O(n^{-2}),
    \end{equation}
    \begin{equation}
        \label{eq:decaysubgoal2}
        \Pr(  \tilde{\tau}_{i+1} \le \tau_{k+1} \wedge (\tau_i + n) \wedge \tau_{i+1}^+  | S_{\tau_i}) \le 2 r_i,
    \end{equation}
    and 
    \begin{equation}
        \label{eq:decaysubgoal3}
        \Pr(\tau_{i+1}' > \tau_i' + n | S_{\tau_i}') \le r_i,
    \end{equation}
    where the stopping time $\tau_i'$ is the analog of $\tau_i$ for each $i$ with regard to the auxiliary process $S_t'$ that agrees with $S_t$ except that
    any any step $t\ge \tau_i$, $S_t'$ rejects the update at $t+1$ if $S_{t+1,1}' < m_*+( (1+\frac{c_*}{16})^i\cdot \frac{\gamma}{2\sqrt{n}})$ or $\max_{j,l\neq 1} |S_{t+1,j} - S_{t+1,l}| > \frac{A\log n}{\sqrt{n}}$.
    Observe that \eqref{eq:maingoal123} follows from a union bound over \eqref{eq:decaysubgoal1}, \eqref{eq:decaysubgoal2} and \eqref{eq:decaysubgoal3}.
    Moreover, \eqref{eq:decaysubgoal1} is a consequence of Lemma~\ref{lem:nondominantcoordinates}; 
    \eqref{eq:decaysubgoal2} and \eqref{eq:decaysubgoal3} are themselves analogs of \eqref{eq:LBtildetaui} and \eqref{eq:UBtauiprime} respectively.

    To illustrate the main difference in the current proof, it suffices for us to prove \eqref{eq:decaysubgoal2}.
    Define the process $\{Z_t\}_{t\ge 0}$ given by
    \[
        Z_t = S_{\tau_i, 1} - S_{(t+\tau_i)\wedge \tilde{\tau}_{i+1}\wedge \tau_{k+1} \wedge \tau^+_{i+1} , 1}. 
    \]
    Clearly we have $Z_0 = 0$, $| Z_{t+1} - Z_t|\le  n^{-1}$ for all $t\ge 0$. 
    Once we show in addition that
    \begin{equation}
        \label{eq:drifttowardsright}
                  \E_{0}[Z_{t+1} - Z_t \mid \mathcal{F}_t] \le 0,
    \end{equation}
    we can conclude \eqref{eq:decaysubgoal2} by Lemma~\ref{lem:mtghittingtime}(2).
    Since $Z_{t+1} = Z_t$ when $t+\tau_i \ge \tilde{\tau}_{i+1}\wedge \tau_{k+1} \wedge \tau^+_{i+1}$,
    it remains to show \eqref{eq:drifttowardsright} for this process 
    when $0\le t$ such that $t+\tau_i < \tilde{\tau}_{i+1}\wedge \tau_{k+1} \wedge \tau^+_{i+1}$.
    In this case, $Z_{t+1} - Z_t = -S_{t+1+\tau_i, 1} + S_{t+\tau_i, 1}$.
    Recall that for any $S_t$ we set  $\hat{S}_t:=(S_{t,1}, \frac{1-S_{t,1}}{q-1}, \dots,  \frac{1-S_{t,1}}{q-1})$.
    By \eqref{eq:proportions-drift}--\eqref{eq:driftdbeta} and Taylor expansion we have
    \begin{align}\label{eq:recall-this-later}
            &\E_0[Z_{t+1} - Z_t \mid S_{t+\tau_i}] = -\frac{1}{n}d_\beta(S_{t+\tau_i}) + O(n^{-2}) \nonumber \\ 
            &=-\frac{1}{n}\left[  
                d_\beta(\hat{S}_{t+\tau_i}) + 
                    \langle S_{t+\tau_i} - \hat{S}_{t+\tau_i}, \nabla d_\beta(\hat{S}_{t+\tau_i}) \rangle 
                    + O(\| S_{t+\tau_i} - \hat{S}_{t+\tau_i}\|_2^2)
            \right] + O(n^{-2})
    \end{align}
    Note that $ \langle S_{t+\tau_i} - \hat{S}_{t+\tau_i}, \nabla d_\beta(\hat{S}_{t+\tau_i}) \rangle =0$ 
    and $\| S_{t+\tau_i} - \hat{S}_{t+\tau_i}\|_1 = O\big(\tfrac{\log n}{\sqrt{n}} \big)$.
    Hence, by \eqref{eq:driftD} and \eqref{eq:drift-step1}--\eqref{eq:negativedrift7}
    \begin{align}\label{eq:recall-this-later-2}
            \E_0[Z_{t+1} - Z_t \mid S_{t+\tau_i}] &= 
             -\frac{1}{n}D_\beta(S_{t+\tau_i, 1}) + O\big( \frac{\log^2 n}{n^2}\big) \le - \frac{S_{t+\tau_i, 1} - m_*}{2n}\cdot c_* < 0,
    \end{align}
    concluding the proof.
\end{proof}

\begin{proof}[\textbf{\emph{Proof of Theorem~\ref{thm:PottsHitLeftTarget}}: the $m\ge m_* + \gamma n^{-1/2}$ case}]
On the other side of $m_*$, 
if 
       $S_{0,1} = m$ for $m\ge m_* + \gamma n^{-1/2}$ and 
    $\|S_0 - (m, \frac{1-m}{q-1},...,\frac{1-m}{q-1})\|_1 = O(n^{-1/2}\log n)$,
then the coalescence of the proportion vectors 
can be proved analogously to the $m\le m_* + \gamma n^{-1/2}$ case, with 
Lemmas~\ref{lem:superdecayPotts}--\ref{lem:supercriticalcouple} in place of Lemmas~\ref{lem:sdecayPotts}--Lemma~\ref{lem:subcriticalcouple}.
\end{proof}

\subsection{Getting away from the saddle point when \texorpdfstring{$\beta = \beta_c$}{beta = betac}}

In order to handle the critical point itself, we need to also show that in $\Omega(n)$ number of steps the proportions chain gets $\gamma n^{-1/2}$ away from $(m_*, \frac{1-m_*}{q-1},...,\frac{1-m_*}{q-1})$ with high probability, and furthermore it does so to the right and to the left with the correct relative probabilities.

To upper bound the exit time of the $O(n^{-1/2})$ window around the saddle point $(m_*, \frac{1-m_*}{q-1},...,\frac{1-m_*}{q-1})$, our proof goes by considering a batch of $\gamma^2 n$ updates, after which there is a constant chance that the process gained exited the $\gamma n^{-1/2}$ window using the variance alone (even taking a worst-case bound on the drift functions). Iterating this ensures that in $O(e^{\gamma^4}\cdot \gamma^3 n)$ time the process will likely have escaped. 

We use the following notations in this section.   
For a constant $\gamma$, let $\tau_\gamma^-=\inf\{ t>0: S_{t,1}<m_* -\gamma/ \sqrt{n} \}$ and 
    $\tau_\gamma^+=\inf\{ t>0: S_{t,1}> m_* +\gamma/ \sqrt{n} \}$.

\begin{lemma}
    \label{lem:Pottsexittime}
    If $\|S_0 - (m_*, \frac{1-m_*}{q-1},...,\frac{1-m_*}{q-1})\|_1 = O(n^{-1/2})$, 
    then for all large $\gamma$, after $T= \gamma^3 e^{O(\gamma^4)} n$ many steps,  $S_{T,1} \notin  [m_* - \gamma/\sqrt{n}, m_* + \gamma/\sqrt{n}]$ with probability $1-O(\gamma^{-1})$.
\end{lemma}
\begin{proof}
    Suppose $S_{0,1} = x_0 \in [m_* - \gamma/ \sqrt{n}, m_* + \gamma/\sqrt{n}]$.
    Denote by $\{W_t\}_{t\ge 0}$ the simple random walk on $\mathbb{Z}$;
    define $\{Z_t\}_{t\ge 0}$ to be a process given by 
    \[
        Z_t:=(m_*n+\gamma \sqrt{n}) - S_{t,1}n \cdot \mathbf{1}\{t<\tau_\gamma^+ \} 
        + W_{t- \tau_{\gamma}^+} \cdot \mathbf{1}\{t\ge \tau_{\gamma}^+\}.
    \]
    We now verify that $\{Z_t\}_{t\ge 0}$ satisfies the conditions in Lemma~\ref{lem:supermtganticoncen}.
    Indeed,  by definition,  
    $Z_0 = (m_*n + \gamma\sqrt{n}) -  x_0n \ge 0$, 
    $Z_{t+1} - Z_t\in \{-1,0,1\}$.
    Moreover, when $t<\tau_\gamma^+$, there exist constants $C_1>0$ and $C_2>0$ such that
    \begin{align*}
           \E[Z_{t+1} - Z_t \mid \mathcal{F}_t] &= n \E[S_{t, 1} - S_{t+1, 1}\mid \mathcal{F}_t] \\ 
        &= -d_\beta(S_t) +O(n^{-1})\ge -D_\beta(S_{t,1}) +O(n^{-1}) \\ &\ge - C_1 \gamma n^{-1/2},
    \end{align*}
    and $\var(Z_{t+1} \mid \mathcal{F}_t) \ge C_2$;
    when $t\ge \tau_\gamma^+$, $  \E[Z_{t+1} - Z_t \mid \mathcal{F}_t] =  \E[W_{t+1} - W_t \mid \mathcal{F}_t] =0$ and $  \var(Z_{t+1} \mid \mathcal{F}_t) = \var[W_{t+1} \mid \mathcal{F}_t] = 1$.
    Define $\tau:=\inf\{t:Z_t > 2\gamma \sqrt{n}\}$ and note that $\min\{\tau^+,\tau^-\} \le \tau$. 
    Lemma~\ref{lem:supermtganticoncen} implies that
    \begin{align*}
        \Pr(\tau \le \gamma^2 n) &\ge C_3 \cdot \exp\Big( 
            -C_4 \Big(
                \frac{2\gamma\sqrt{n}}{\sqrt{\gamma^2 n}} + C_1\gamma n^{-1/2} \cdot \sqrt{\gamma^2 n} \
            \Big)^2 \Big) + O(1/\sqrt{\gamma^2 n})  \\
        & \ge C_3  \exp\Big( 
            -C_4 \Big(
             2 + C_1\gamma^2 \
            \Big)^2 \Big)\,,
    \end{align*}
    where the constants $C_3>0$ and $C_4>0$ depend only on $C_2$. 
    
    On the event that $\tau > \gamma^2 n$, then we can restart the process from the value of $S_{\gamma^2 n,1} \in [m_* - \gamma n^{-1/2}, m_* + \gamma n^{-1/2}]$ whence there is a fresh attempt of probability at least $e^{ - \Omega(\gamma^4)}$ of exiting the window in the next $\gamma^2 n$ steps. Repeating this argument $K = \gamma^3 e^{ \Omega(\gamma^4)}$ many times, each consisting of $\gamma^2 n$ steps, boosts the probability of having exited up to $1-O(\gamma^{-1})$. 
\end{proof}

The next lemma ensures there exists an initialization perturbation of $m_*$ by order $n^{1/2}$ such that from there, we get the correct relative probabilities for exiting the saddle to the right vs.\ the left to ensure convergence. 
\begin{lemma}
    \label{lem:existcstar}
    There exists a unique constant $\hat c_*$ such that initialized from $\hat \nu^{\otimes}(m_0)$ with $m_0 = m_* + \hat c_* n^{-1/2} + o(n^{-1/2})$, we have  
    $$\Pr(\tau^-_{\gamma} < \tau^+_{\gamma}) = \xi - o_{\gamma,n}(1)\,.$$
\end{lemma}
\begin{proof}
    Consider the time and space rescaled process \[\bar S_t = \sqrt{n}(S_{tn} - (m_*, \frac{1-m_*}{q-1},..., \frac{1-m_*}{q-1})).\] By standard results regarding limits of discrete stochastic dynamics as stochastic differential equations (see, e.g.,~\cite{kloeden2011numerical}), $\bar S_t$ converges to the solution of a stochastic differential equation (SDE) $Z_t$ on $\mathcal S$ with  drift and volatility functions with bounded Lipschitz constants. On a compact space, like $\mathcal S$, such convergence results only require $O(1)$-Lipschitz bounds on the drift function (which hold in our setting for $g_{\beta} - s$) and moment estimates on the step-wise increments (for which ours are uniformly bounded by $\pm \frac{1}{\sqrt{n}}$). 

    In our context, the limiting SDE we end up with for the first coordinate of the rescaled process, $\bar S_{t,1}$, is a 1-dimensional SDE $Z_t^1$. This is because the effect of the other coordinates $(\bar S_{t,2},...,\bar S_{t,q})$ cancels to first order when they are in a neighborhood of equiproportionality, as seen in the cancellation of the first order dependence in between~\eqref{eq:recall-this-later}--\eqref{eq:recall-this-later-2}. The limiting volatility is constant because the effect of corrections to the vector $(m_*, \frac{1-m_*}{q-1},...,\frac{1-m_*}{q-1})$ on the variance are vanishing. In total, we get that $\bar S_{t,1}$ converges to an SDE $Z_{t}^1$ solving 
    \begin{align*}
        dZ_{t}^1 = D_{\beta_c}(Z_t^1) dt + A dB_t^1\,,
    \end{align*}
    where $B_t^1$ is a standard Brownian motion, $A = A(q)$ is a constant, and this is 
    initialized from $Z_0^1 \sim \mathcal N(\hat c_*, d^2)$ for some variance $d^2(\beta,q)$. 

    If for each initial $\hat c_*$ we define $f(\hat c_*)$ to be the probability that $\bar S_{t,1}$ hits $-\gamma$ before $\gamma$ and we define $f^Z(\hat{c}_*)$ to be the same probability for $Z_t^1$< then  by the convergence described above, $$|f(\hat c_*) - f^Z(\hat c_*)| = o_n(1)$$ for all $\hat c_*$. 
    Furthermore, $f^Z$ is easily checked to be $n$-independent, continuous, monotone, and going to $0$ as $\hat c_* \to \gamma$ and $1$ as $\hat c_* \to -\gamma$. Therefore, there is a unique $\hat c_*$ where $f^Z$ is the desired $\xi$. Finally, to see that the choice of $\hat c_*$ is $\gamma$-independent up to $o_\gamma(1)$, note that if $\tau_{\gamma}^-<\tau_{\gamma}^+$, then with probability $1-o_\gamma(1)$, also $\tau_{2\gamma}^- <\tau_{2\gamma}^+$ (by the argument in Lemma~\ref{lem:sdecayPotts}). 
\end{proof}

\subsection{Escaping the unstable fixed point at low temperatures \texorpdfstring{$\beta>\betas$}{beta>betas}}

In the regime of $\beta>\betas$, the saddle becomes at $1/q$ and the landscape's geometry changes somewhat so that this saddle is separating all $q$ ordered phases (the disordered phase no longer being metastable). This introduces some additional complications, particularly because the other coordinates besides the first one do not drift towards equiproportionality as they do in Lemma~\ref{lem:nondominantcoordinates}. This necessitates more understanding of the full $q$-dimensional landscape. We establish the following quasi-equilibration result. 

\begin{theorem}
    \label{thm:lowtempcoupling}
        Let $q > 2$ and $\beta > \betas$.
    Suppose $\sigma_0$ is a configuration such that 
$S_{0,1} > S_{0,i} + \gamma n^{-1/2}$ for every $i=2,\dots, q$ and
for a large constant $\gamma >0$.
    Then there exists $T=O(n \log n)$ such that 
    \begin{align*}
        \|\Pr_{\sigma_0} (S(\sigma_T) \in \cdot) - \Pr_{\pi^{\ord,1}} (S(\sigma_T') \in \cdot) \|_{\TV} = o_\gamma(1)\,.
    \end{align*}
\end{theorem}

The proof of Theorem~\ref{thm:lowtempcoupling} breaks up into several parts. We begin with some preliminary lemmas about the drift function's behavior, and classification of its fixed points and their attractive/repulsive directions. Using that, in Lemma~\ref{lem:ODE-gets-you-to-mr}, we show that as soon as one coordinate has a macroscopically larger fraction than the other coordinates, (even in the absense of equiproportionality of the other coordinates), the Glauber dynamics quickly quasiequilibrates to the corresponding ordered phase. Lemma~\ref{lem:movefaraway} shows that if the dominant coordinate has $\omega(n^{-1/2})$ larger proportion than all the others, this gets boosted to a macroscopic bias. 
Finally, we use anti-concentration of the proportions vector to argue that if the initialization is the fully uniform-at-random initialization, i.e., $m_0 = \frac{1}{q}$, after $O(n)$ steps,  with high probability, one coordinate is $\omega(n^{-1/2})$ larger than all the others. 

\subsubsection{Fixed point analysis of the \texorpdfstring{$q$}{q}-dimensional drift function}
We first provide a lemma that characterizes the 1-dimensional drift function $D_\beta$ in this regime.

\begin{lemma}
    \label{lem:lowtempdrift}
    Suppose $q>2$ and $\beta > \betas$. 
    Then $D_\beta(\frac{1}{q}) = 0$ and there is a unique root for $D_\beta(x) = 0$ in $(\frac{1}{q}, 1)$, denoted $\mr$.
    Moreover, $D_\beta'(1/q)>0$ and $D_\beta'(\mr) < 0$.
\end{lemma}
\begin{proof}
    We recall the following facts regarding $D_\beta$ from the proof of Lemma~\ref{lem:driftanalysis}, 
    \[
    D_\beta(1/q)=0, \qquad D_\beta(1)<0 \qquad \text{and} \qquad D_\beta'(1/q)=-1+\frac{\beta}{q}.
    \]
    Since $\beta>\betas = q$, we have $D_\beta'(1/q)>0$. 
    Thus, for small enough $\epsilon>0$, $D_\beta(1/q+\epsilon)>0$, and it follows from continuity 
    that there exists a zero of $D_\beta$ between $1/q+\epsilon$ and $1$.
    Let $\mr$ be the smallest zero that is greater than $1/q$. 
    Using \eqref{eq:driftD2} and \eqref{eq:driftderivative}, one can verify that if $D_\beta(x)=0$ then $D_\beta'(x)\neq 0$ for $x\in(1/q,1]$ so $D_\beta'(\mr)<0$.
    By continuity of $D_\beta'$, there exists a point $x^+\in(1/q, \mr)$ such that $D_\beta'(x^+)=0$ and $D_\beta(x^+)>0$.

    Now, by~\eqref{eq:driftderivative}, $D_\beta'(x)=0$ has at most two zeros.
    Aside from $x^+$, let $x^{++}$ denote the other zero (if it exists).
    If $x^{++}$ does not exist or $x^{++}\le \mr$, then $D_\beta'(x)<0$ for all $x>\mr$ so $\mr$ is the unique zero in $(1/q,1)$.
    Now assume $x^{++}>\mr$.
    By continuity of $D_\beta'$, $D_\beta'(x)<0$ for all $x\in (\mr,x^{++})$,
    so by integration $D_\beta(x)<0$ for $x\in(\mr,x^{++}]$.
    Moreover, note that $D_\beta$ is a monotone function in $(x^{++},1]$ since $D_\beta'$ does not longer change sign in this interval.
    Hence, $D_\beta$ has no zero on $[x^{++},1]$ and $\mr$ is the unique root in $(1/q,1]$.
\end{proof}

Unlike the $\beta \in (\betau,\betas)$ regime, though, we need a more refined understanding of the full $q$-dimensional landscape, establishing that the system driven by $d_{\beta}$ has its only stable fixed points at $(\mr, \frac{1-\mr}{q-1},...,\frac{1-\mr}{q-1})$. This will be used to show that as long as the proportions vector is $\Omega(1)$ away from a fixed point of the system, it rapidly reaches a small neighborhood of a fixed point of the form $(\mr, \frac{1-\mr}{q-1},...,\frac{1-\mr}{q-1})$. The following classifies the fixed points of the dynamical system $ds_t = d_\beta(s_t)dt$.

\begin{lemma}\label{lem:drift-fixpoint-analysis-low-temp}
    The set of solutions of $d_\beta(s) = 0$, or equivalently $g_\beta(s) = s$ are classified (up to permutations) as the following. If $g_\beta(s) = s$, then for some $k \in \{1,...,q\}$, the vector $s$ must be $$(\underbrace{a,...,a}_{k}, \underbrace{\tfrac{1-k a}{q-k},...,\tfrac{1-ka}{q-k}}_{q-k})$$ where $a(\beta,k)\ge 1/q$.  

    Moreover, when $\beta>\betas$, the only stable solution is the one where $k=1$, and $a = \mr$, and the other fixed points are specifically unstable in the direction of increasing the first coordinate and decreasing another. 
\end{lemma}

\begin{proof}
    We first reason that all solutions are of the form of permutations of \[(a,...,a, \frac{1-ka}{q-k},...,\frac{1-ka}{q-k}).\] In order to see this, we suppose by way of contradiction there exists a solution $\bar s$ having three distinct values $a,b,c$ appearing in its proportions vector (wlog as $\bar s_1=a, \bar s_2 = b$ and $\bar s_3 =c$; let $\bar Z = \sum_{j} e^{\beta \bar s_j}$). Then, $e^{\beta a}= a \bar Z$ and similarly for $b$ and $c$, and $a\ne b\ne c$. But this is impossible because for any $C,\beta>0$, the equation $e^{\beta x} = C x$ has at most two solutions on $[0,1]$. When $k=1$, the only possible such solution is the one which has $D_{\beta}=0$ (because the other coordinates are equal, which is where $D_\beta = d_\beta$), and we can apply Lemma~\ref{lem:lowtempdrift} to read off that in that case $a=\mr$. 

    Now fix any vector $\bar s$ of the form $(a,...,a, \frac{1-ka}{q-k},...,\frac{1-ka}{q-k})$ for $k\ge 2$ with $g_\beta(\bar s) = \bar s$. We wish to show that 
    \begin{align*}
        \langle \nabla d_{\beta,1}(\bar s), (1,-1,0,...,0)\rangle > 0
    \end{align*}
    as that would say that such a fixed point is unstable with a drift towards the first coordinate increasing if we perturb in the $(1,-1,0,...,0)$ direction. To see this, differentiating $g_{\beta,1}$ and plugging in $g_{\beta,1}(\bar s) = \bar s$ and $\bar s_1 = \bar s_2 = \bar a$, 
    \begin{align*}
        \frac{d}{d x_1} d_{\beta,1} (\bar s) - \frac{d}{d x_2} d_{\beta,1}(\bar s) & = \beta (\bar s_1 - \bar s_1^2) - 1  + \beta \bar s_1 \bar s_2   = \beta \bar a - 1\,.
    \end{align*}
    Since $\bar a>1/q$ and $\beta>\betas = q$, this is strictly positive as claimed. 
\end{proof}

\subsubsection{Getting to the stable fixed point once one coordinate dominates}

Using the above fixed point analysis, we can show that initialized macroscopically away from the fixed points of the drift function, the dynamics rapidly equilibrates to the right phase. 

\begin{lemma}\label{lem:ODE-gets-you-to-mr}
    Suppose $S_0$ is such that $S_{0,1}\ge S_{0,i} + \delta$ for all $i\ge 2$ for some $\delta>0$. Then, for every $\rho>0$, there exists a time $T = O(n)$ such that $\|S_T - (\mr,\frac{1-\mr}{q-1},...,\frac{1-\mr}{q-1})\|_1 \le \rho$ with probability $1-o(1)$.
\end{lemma}

\begin{proof}
We argue that away from the fixed points of the drift function $d_{\beta}(s) = g_\beta(s) -s$, the evolution of the proportions vector is well-approximated by the deterministic process defined by
\begin{align}\label{eq:deterministic-ODE-low-temp}
    \mathbf{S}_t = \mathbf{S}_0  + \frac{1}{n} \sum_{s\le t} d_\beta(\mathbf{S}_s)\,.
\end{align}
By convergence of Euler discretization for ODEs (see e.g.,~\cite{griffiths2010numerical}), since $d_\beta$ has bounded Lipschitz constant, as $n \to\infty$, $\mathbf{S}_{tn}$ converges in $C^\infty[0,T]$ to the solution to the ODE system $d \mathbf{S}_t = d_\beta(\mathbf{S}_t)dt$ initialized from $\mathbf{S}_0$.  This latter process is easily checked to be a gradient dynamical system, i.e., the gradient flow for $$F_\beta(\mathbf{S}) =  - \frac{1}{\beta} \log \sum_{i\in \{1,...,q\}} e^{ \beta \mathbf{S}_i}  + \frac{1}{2}\|\mathbf{S}\|_2^2\,;$$
therefore it has no closed orbits, and by the bounded Lipschitz constant of $d_\beta$, converges exponentially fast to the fixed point in whose basin it is initialized. Moreover, by Lemma~\ref{lem:drift-fixpoint-analysis-low-temp}, if initialized with $\mathbf{S}_{0,1} \ge \mathbf{S}_{0,i}+ \delta$ for all $i = 2,...,q$ for some $\delta>0$, then it is in the basin of attraction of the stable fixed point $(\mr, \frac{1-\mr}{q-1},...,\frac{1-\mr}{q-1})$. In particular, for every $\delta>0$, and $\rho>0$, there is a $T = O(1)$ such that under the assumptions on $\mathbf{S}_0$ of the lemma, the system of~\eqref{eq:deterministic-ODE-low-temp} attains 
\begin{align}\label{eq:ODE-gets-to-stable-fixpoint}
    \|\mathbf{S}_{Tn} - (\mr,\tfrac{1-\mr}{q-1},...,\tfrac{1-\mr}{q-1})\|_1\le \rho\,.
\end{align}
It remains to show that for linear times $t=O(n)$, we have  $\|S_t - \mathbf{S}_t\|_1 = o(1)$ with high probability. For this, we can write 
\begin{align}\label{eq:stochastic-approximation}
    \|S_{tn} - \mathbf{S}_{tn}\|_1 \le \frac{1}{n} \sum_{l < tn} \|d_{\beta}(S_l) - d_{\beta}(\mathbf{S}_l)\|_1 + \Big\|\sum_{l< tn} \Big((S_{l+1}-S_l) - \frac{1}{n}d_\beta (S_l)\Big)\Big\|_1\,.
\end{align}
Since $d_\beta$ has bounded Lipschitz coefficient (by some $C_{\beta,q}$), the first sum is at most $\frac{C_{\beta,q}}{n} \sum_{l < tn} \|S_l - \mathbf{S}_l\|_1$. By~\eqref{eq:proportions-drift}, the second sum is a sum of martingale increments (up to an error of $n\cdot O(n^{-2}) = O(n^{-1})$), each of which take values in $\{-\frac{1}{n},0,\frac{1}{n}\}$, so by standard martingale concentration estimates (Doob's maximal inequality and Azuma--Hoeffding bound), with probability $1-o(1)$, the maximum over all $t\le Tn$ of the second term above is $O(n^{-1/2}\log n)$. Applying the discrete Gronwall inequality, we get 
\begin{align}\label{eq:discrete-Gronwall}
    \sup_{t\le T} \|S_{tn} - \mathbf{S}_{tn}\|_1 \le O(n^{-1/2}\log n) \cdot e^{ C_{\beta,q}T}\,.
\end{align}
Combined with~\eqref{eq:ODE-gets-to-stable-fixpoint}, and reparametrizing $Tn$ to  $T=O(n)$, we deduce the lemma. 
\end{proof}

 \subsubsection{Getting one coordinate to dominate}
Once one of the color classes has a bias of at least $\gamma/\sqrt{n}$, we can call it the dominant color class and without loss of generality, label it the first coordinate. From there, the Potts Glauber dynamics gradually shifts 
 $\Omega(1)$ away from a fixed point.

\begin{lemma}
    \label{lem:movefaraway}
    Let $q >2$ and $\beta> \betas$.
    Suppose $\sigma_0$ is a configuration such that 
$S_{0,1} > S_{0,i} + \gamma n^{-1/2}$ for every $k=2,\dots, q$ and a large $\gamma>0$.
    Then there exist $T = O(n \log n)$ and $\delta =\Omega(1)$ such that with probability $1 - e^{-\Omega(\gamma^2)}$, 
    \[
        S_{T,1} \ge S_{T,k} + \delta, \qquad \text{ for all } k=2,\dots, q. 
    \]
\end{lemma}

\begin{proof}
    We first make several useful definitions.
    Set $\delta$ and $\epsilon$ be sufficiently small positive constants such that $q(1+\epsilon)^2 \le \beta$ and $4\delta \beta < \ln (1+\epsilon)$, 
    and set $L=O(\log n)$ be the least integer such that $\gamma \big(1+\frac{\epsilon}{16} \big)^L > 2\delta \sqrt{n}$.
    For every $k\ge 2$, we define $\tau_{0,k}= 0$,
    \[
        \tau_{i+1,k} := \inf\left\{t\ge \tau_{i,k} : S_{t,1} > S_{t,k} + \Big(1 +\frac{\epsilon}{16} \Big)^{i+1}  \frac{\gamma}{\sqrt{n}} \right\},
    \]
    and 
    \[
          \tilde{\tau}_{i+1,k}:= \inf\left\{ t>\tau_{i,k}: S_{t,1} < S_{t,k} +  \Big(1 +\frac{\epsilon}{16} \Big)^{i}  \frac{\gamma}{2\sqrt{n}} \right\}, 
    \]
    for every $i\ge 0$. 
    Moreover, define $
        \tau_{k}^-:=
        \inf\left\{t\ge 0 : \frac{\beta e^{\beta S_{t,k}}}{\sum_{j=1}^q e^{\beta S_{t,j}}} \le 1+\epsilon \right\},
    $
    and $\tau_1^+ :=\inf\{t\ge 0: S_{t,1}\le \max_{k\neq 1} S_{t,k} \}$.
    Note that by the choice of $\delta$ and $\epsilon$, $\tau_k^- \ge \tau_{L+1,k}\wedge \tau_1^+$ for each $k$, and that
    $\tau_1^+ >\bigwedge_{i,k} \tilde{\tau}_{i,k}$.
    Define $S_t'$ to be the auxiliary process such that $S_t = S_t'$ for $t < \bigwedge_{i,k} \tilde{\tau}_{i,k}$.
    Define $\{\tau_{i,k}'\}$ as the stopping times for $S_t'$ in place of $\{\tau_{i,k}\}$. 
    For $i\ge0$, set $r_i = \exp\Big[- \Omega(\gamma^2 (1+\tfrac{\epsilon}{16})^{2i} ) \Big]$.
    We will show that 
    \begin{align}
         \Pr&\left(  
            \bigcap_{k=2}^{q} 
            \left\{
                \bigcap_{i=0}^{L-1} \{ \tau_{i+1,k}'\le \tau_{i,k}' + n, 
                \tilde{\tau}_{i+1,k} > \min\{\tau_{i,k} + n, \tau^+_1, \tau_{L+1,k} \}  \} 
                \cap
                \{
                    \tilde{\tau}_{L+1,k} > \tau_L +Ln
                \}
            \right\}
        \right)  \nonumber \\
         & \qquad \qquad\qquad  \ge 1 - \sum_{i=0}^{L-1} 3q r_i - o(1)\,.  \label{eq:dreamgoal}
    \end{align}
    The events in \eqref{eq:dreamgoal} together imply that at time $T=Ln$, 
    $S_{T,1} \ge S_{T,k} + \delta$ for all $k = 2, \dots, q$.

    First we show that for each $k$ and each $i\in [0,L-1]$, 
    $ \tau_{i+1,k}'\le \tau_{i,k}' + n$ with probability $1-r_i$.
    Define the process $\{W_t\}_{t\ge 0}$ given by
    \[
        W_t=\big(1+\frac{\epsilon}{16} \big)^{i+1} \frac{\gamma}{\sqrt{n}} - S_{t,1}' + S_{t,k}'.
    \]
    Clearly, $|W_{t+1} - W_{t}| \le 2/n$ for $t\ge 0$.
    We will show that on $\{W_t \ge 0\}$, 
    \begin{equation}
        \label{eq:dreamstep1}
            \E[W_{t+1} - W_t \mid \mathcal{F}_t] \le -\frac{\epsilon}{2n} \cdot (1+\frac{\epsilon}{16})^i\cdot \frac{\gamma}{2\sqrt{n}},
    \end{equation}
    and obtain $\Pr( \tau_{i+1,k}'> \tau_{i,k}' + n \mid \tau_{i,k}') \le r_i$ following Lemma~\ref{lem:mtghittingtime}(1). 
    To show \eqref{eq:dreamstep1},
    recalling \eqref{eq:gcomparisons}, we have
    \begin{align*}
           \E[W_{t+1} - W_t \mid \mathcal{F}_t] &= 
           -\frac{1}{n}[ g_{\beta,1}(S_t') - g_{\beta,k}(S_t') - (S_{t,1}' - S_{t,k}') ] + O(n^{-2}) \\
           &=-\frac{1}{n} \left[  
            \frac{e^{\beta S_{t,1}'} - e^{\beta S_{t,k}'} }{\sum e^{\beta S_{t,j}'}}
            - (S_{t,1}' - S_{t,k}') 
           \right]+ O(n^{-2}) \\
           & \le -\frac{1}{n} 
           \big( \frac{\beta e^{\beta S_{t,k}'} }{\sum e^{\beta S_{t,j}'}} - 1\big)
             (S_{t,1}' - S_{t,k}') 
           + O(n^{-2})\,.
    \end{align*}
    Observe that $\frac{\beta e^{\beta S_{t,k}'}}{\sum_{j=1}^q e^{\beta S_{t,j}'}} > 1+\epsilon$ for $t\le \tau_{L+1,k}'$, since otherwise 
    \[
        e^{\beta S_{t,k}'}\le  \frac{1 + \epsilon}{\beta}\sum_{j=1}^q e^{\beta S_{t,j}'}
        \le \frac{1}{(1+\epsilon) q} \sum_{j=1}^q e^{\beta S_{t,j}'}\le \frac{1}{1+\epsilon} e^{\beta S_{t,1}'},
    \] 
    and $S_{t,k}'\le S_{t,1}' - \beta^{-1}\ln (1+\epsilon) < S_{t,1}' - 4\delta$.
    Hence,  we establish \eqref{eq:dreamstep1} by further noting that $S_{t,1}' - S_{t,k}' \ge  \Big(1 +\frac{\epsilon}{16} \Big)^{i}  \frac{\gamma}{2\sqrt{n}}$.

    Moreover, 
    using positive drift of the process $S_{t+\tau_{i,k}, 1} - S_{t+\tau_{i,k}, k}$ and analyzing an associated process as done in the proof for \eqref{eq:LBtildetaui},  
    we get that $\Pr(\tilde{\tau}_{i+1,k} \le \min\{\tau_{i,k} + n, \tau^+_1, \tau_{L+1,k} \} \mid \tau_{i,k}) \le 2r_i$.
    Finally, when the positive drift is constantly large, the probability that $  \tilde{\tau}_{L+1,k} \le \tau_L +Ln$
    is diminishing. 
    Therefore, we obtain \eqref{eq:dreamgoal} by a union bound and conclude the proof.
\end{proof}

\begin{proof}[\textbf{\emph{Proof of Theorem~\ref{thm:lowtempcoupling}}}]
By 
Lemma~\ref{lem:movefaraway}, after $T= O(n\log n)$ steps, with probability $1-o_\gamma(1)$, the proportions chain has $S_{T,1} \ge S_{T,k}+\delta$ for all $k=2,...,q$. From there, 
Lemma~\ref{lem:ODE-gets-you-to-mr} ensures that in a further $O(n)$ steps, the proportions vector is within an arbitrarily small $\rho$ distance from the fixed point corresponding to that phase, $(\mr, \frac{1-\mr}{q-1},...,\frac{1-\mr}{q-1})$. 
Finally, Lemma~\ref{lem:converge2mr} and Lemma~\ref{lem:supercriticalcouple} imply coupling of the proportions chain from there to that of a dynamics initialized from $\pi^{\ord,1}$. 
\end{proof}

Let us finally describe how one obtains the case of $m_0 = 1/q$, where we are using a fully uniform at random initialization for the low-temperature Potts model. 

\begin{lemma}\label{lem:1/q-initialization}
    Suppose $\beta>\betas$ and $\|S_0 - (\frac{1}{q},...,\frac{1}{q})\|_1 = O(n^{-1/2})$. Then for all $\gamma$ large, there exists a $C_{\gamma}>0$ such that the hitting time to having $S_{t,i}>S_{t,j}+ \gamma n^{-1/2}$ for some $i$ and all $j\ne i$ is at most $C_{\gamma} n$ except with probability $1-o_{\gamma}(1)$. 
\end{lemma}
\begin{proof}
    It is sufficient to show that in some $O(n)$ steps possibly depending on $\gamma$, the process attains $S_{t,1} >S_{t,j} + \gamma n^{-1/2}$ for all $j>1$ (up to permutation of the coordinates). In order to show this, we show that uniformly over any initialization, one obtains some coordinate which is $\gamma n^{-1/2}$ larger than the rest in $n$ steps, with some $c_\gamma>0$ probability. This then gets boosted to $1-o_\gamma(1)$ probability after $C_\gamma n$ steps. There are two cases of initial proportions vector $S_0$ to consider: 
    \begin{enumerate}
        \item Starting from $S_0$ and evolving $S_{t}$ according only to the drift $d_{\beta}$ as in~\eqref{eq:deterministic-ODE-low-temp} for $n$ steps (taking the martingale increments to be zero), one of the coordinates becomes $2\gamma n^{-1/2}$ larger than all the others; 
        \item The complement, in which case there are at least 1 coordinate $i>1$ which is within $2\gamma n^{-1/2}$ of the maximal coordinate (assumed to be the first coordinate) under the drift of~\eqref{eq:deterministic-ODE-low-temp} after $n$ steps. 
    \end{enumerate}
    In the first case, by Azuma--Hoeffding's inequality, after $n$ steps, the second term in the right-hand side of ~\eqref{eq:stochastic-approximation} is at most  $\gamma \sqrt{n}$ with probability $1-e^{ - \Omega(\gamma^2)}$, whence following the reasoning between~\eqref{eq:stochastic-approximation}--\eqref{eq:discrete-Gronwall}, we have $\|S_{t} - \mathbf{S}_{t}\|_1\le \gamma n^{-1/2}$ for all $t\le n$ and we will have attained a configuration having $S_{T,1} >S_{T,i} + \gamma n^{-1/2}$ after $T= n$ steps. 
    
    Now consider the second case and let $\mathbf{S}_t$ be defined as in~\eqref{eq:deterministic-ODE-low-temp}. Firstly, by comparison of 
    \begin{align*}
        S_{t,1} - \mathbf{S}_{t,1} = \sum_{l<t}(S_{l+1,1} -S_{l,1} - \frac{1}{n}d_\beta(S_l))\,, 
    \end{align*}
    to a random walk with a variance strictly bounded away from zero, taking $n$ steps that are $O(1/n)$ sized, there is a uniformly positive probability $p_{\gamma,1}$ that this is in $[2q\gamma n^{-1/2},3q\gamma n^{-1/2}]$. 
    Also, conditionally on a typical realization of such a sequence, $S_{t,1}$ only changed $n/q + o(n)$ many times. On the remaining $(1-\frac{1}{q})n- o(n)$ steps, the increments of $S_{t,2} - \mathbf{S}_{t,2}$ still have a uniformly lower bounded variance. This leads to another uniformly positive probability $p_{\gamma,2}$ that it has $S_{t,2} - \mathbf{S}_{t,2} \in [-\gamma n^{-1/2},\gamma n^{-1/2}]$ (the conditioning on the event for $S_{t,1}$ only has a negligible effect on the drift for $S_{t,2}$ because of Lipschitz continuity of $d_{\beta}$). Repeating this for the next $q-1$ coordinates, with the very last one having no variance remaining but deterministically having decreased because of what happened for the other coordinates, we find that there is at least probability $p_{\gamma,1} \cdots p_{\gamma,q-1}$ such that 
    \begin{align*}
        S_{t,1} - \mathbf{S}_{t,i} \ge -2\gamma n^{-1/2} - (q-1) \gamma n^{-1/2} + 2q\gamma n^{-1/2} = (q-1) \gamma n^{-1/2} \ge \gamma n^{-1/2}\,,
    \end{align*} 
    for all $i\ge 2$. Importantly, by continuity and compactness of the space, the variances are uniformly bounded from above and below, and so the above lower bound $p_{\gamma,1} \cdots p_{\gamma,q-1}$ is independent of the initial proportions vector. This argument can thus be repeated some $C_{\gamma}$ times to ensure a probability $1-o_\gamma(1)$ that one coordinate has $\gamma n^{-1/2}$ larger proportion than any of the other coordinates. The permutation symmetry of the initialization ensures the $q$ coordinates are equally likely to become this dominant one.  
\end{proof}

\subsection{Lower bound on the mixing time with different choices of \texorpdfstring{$m_0$}{m0}}

By combining the above quasi-equilibration results with metastability of the ordered and disordered phases, we show that if the initialization is the product measure with parameters not satisfying the conditions of Theorem~\ref{thm:Potts-Glauber}, then mixing is exponentially slow.

\begin{theorem}\label{thm:Potts-slow}
    For every $q>2$ and $\beta \in (\betau, \betas)$, if $m_*(\beta,q)$ and $\hat c_*(q)$ are as in Theorem~\ref{thm:Potts-Glauber}, then the Potts Glauber dynamics initialized from $\hat \nu^\otimes(m_0)$ with 
    \begin{enumerate}
        \item $\beta \in (\betau,\betac)$ and $m_0  > m_*(\beta,q)  - O(n^{-1/2})$,
        \item $\beta = \beta_c$ and $m_0 \ne m_*(\beta,q) + \hat c_*(q) n^{-1/2} + o(n^{-1/2})$,
        \item $\beta\in (\betac,\betas)$ and $m_0 < m_*(\beta,q) + O(n^{-1/2})$,
    \end{enumerate}
    takes $\exp(\Omega(n))$ time to reach $o(1)$ total-variation distance to stationarity. 
\end{theorem}

\begin{proof}
    We provide the details for the proof of item 1, the other cases following by similar reasoning. For any initialization parameter $m_0 > m_* - K n^{-1/2}$ for some $K = O(1)$, by Lemma~\ref{lem:existcstar} and~\ref{lem:Pottsexittime}, there is a positive probability $c_K>0$ that the process $\{\max_i S_{t,i}\}$ hits $m_* + \gamma n^{-1/2}$ before $m_* - \gamma n^{-1/2}$ (for sufficiently large $\gamma$) in some $t\le C_\gamma n$ many steps. By Lemma~\ref{lem:nondominantcoordinates}, with probability $1- o(1)$, at exit, the configuration satisfies the necessary conditions to apply the second part of Theorem~\ref{thm:PottsHitLeftTarget} and quasi-equilibrate to the ordered phase $\pi^\ord$. Putting these together, we find that for some $T_0 = O(n \log n)$,
    \begin{align}
        \|\mathbb P(X_{T_0} \in \cdot ) - \pi^\ord \|_\TV  \le 1- c_K + o_{\gamma,n}(1)\,. 
    \end{align}
    Next, we claim that a Potts dynamics chain initialized from $\pi^\ord$ retains total-variation distance $1-o(1)$ for exponentially many steps to $\pi$ when $\beta \in (\betau,\betac)$. To see this, we use that by Corollary~\ref{cor:equilibrium-estimates-Potts} the initialization from $\pi^\ord$ has proportions vector within $O(n^{-1/2})$ distance of $(\mr, \frac{1-\mr}{q-1},...,\frac{1-\mr}{q-1})$ (up to permutations). This is a stable fixed point of the drift function $d_\beta$ when $\beta>\betau$ by Lemma~\ref{lem:driftanalysis}. Following the Taylor expansion and martingale argument used in Lemma~\ref{lem:nondominantcoordinates}, in particular the application of Lemma~\ref{lem:mtghittingtime}, it takes $\exp(\Omega(n))$ time to leave an $\epsilon$-neighborhood of the stable fixed point $(\mr, \frac{1-\mr}{q-1},...,\frac{1-\mr}{q-1})$. (Note that the details of this last stage of reasoning are provided in the proof of \cite[Theorem 3]{CDLLPS-mean-field-Potts-Glauber}.) Combining the above, we find that for some $C>0$, for all $ T_0\le t \le e^{ n/C}$, 
    \begin{align*}
        \mathbb P(|\max_i S_{t,i}  - \mr| \le \epsilon) \ge c_K - o_{\gamma,n}(1)\,.
    \end{align*}
    By Lemma~\ref{lem:equilibrium-estimates}, when $\beta \in (\betau, \betac)$, since $\mr \ne 1/q$, one has $\mu(|\max_i S_{t,i}  - \mr|<\epsilon) =o(1)$ for small $\epsilon$, so the above bound implies that for $\gamma$ large, the total-variation to stationarity is at least $c_K/2$. 
\end{proof}

\subsection{Proof of Theorem~\ref{thm:Potts-Glauber}}
With all the above ingredients at hand, we are in position to complete the proof of Theorem~\ref{thm:Potts-Glauber}. 

\begin{proof}[\textbf{\emph{Proof of Theorem~\ref{thm:Potts-Glauber}}}]
    For item 1 of the theorem, fix $\beta\in (\betau,\betac)$, and suppose $m_0 = m_* - \omega(n^{-1/2})$, where $m_*$ is specified in Lemma~\ref{lem:driftanalysis}. 
   If $\sigma_0$ is generated according to $\hat \nu^\otimes(m_0)$, with probability $1-o(1)$, by concentration of multinomial random variables, $S_0$ has one coordinate that is within $O(n^{-1/2})$ of $m_0$ and the other coordinates are all within $O(n^{-1/2})$ of $\frac{1-m_0}{q-1}$. Without loss of generality, permute the coordinates so that it is the first coordinate that is close to $m_0$. Then, 
    by Theorem~\ref{thm:PottsHitLeftTarget},  
   $\|\Pr_{\sigma_0} (S(\sigma_T) \in \cdot) -\Pr_{\sigma_0'} (S(\sigma_T') \in \cdot) \|_{\TV} = o(1)$, where $\sigma_0'\sim \pi^{\dis}$. Furthermore, by Lemma~\ref{cor:equilibrium-estimates-Potts}, when $\beta<\betac$, we have $\|\pi^\dis - \pi\|_{\TV} = o(1)$, so it follows from the triangle inequality that
   \begin{align*}
          & \|\Pr_{\sigma_0\sim \hat \nu^\otimes(m_0)} (S(\sigma_T) \in \cdot) - \pi(S(\sigma) \in \cdot) \|_{\TV}  \\
    &\le  \|\Pr_{\sigma_0\sim \hat \nu^\otimes(m_0)} (S(\sigma_T) \in \cdot) - \Pr_{\sigma_0'\sim \pi^{\dis}} (S(\sigma_T') \in \cdot) \|_{\TV} + \|\pi^\dis - \pi\|_{\TV}\\
    & = o(1).
   \end{align*}
   Since $\hat \nu^\otimes(m_0)$ and $\pi$ are invariant under permutation of vertices, we obtain that 
   \[   \|\Pr_{\hat \nu^{\otimes}(m_0)}(\sigma_T \in \cdot) - \pi \|_{TV} = o(1). \]
For item 3, the proof also follows from Theorem~\ref{thm:PottsHitLeftTarget} and Lemma~\ref{cor:equilibrium-estimates-Potts} by a symmetrical argument, upon noticing that  $\|\pi - \pi^\ord\|_{\TV} = o(1)$, and $\pi^{\ord}$ is a $(1/q,...,1/q)$ mixture of $(\pi^{\ord,i})_{i\in [q]}$, and each of the $q$ coordinates are equally likely under the initialization $\hat \nu^{\otimes}(m_0)$ to dominate.  

We proceed to item 2 where $\beta =\betac$.
For any target $\epsilon>0$, we take $\gamma=\gamma(\epsilon)$ to be sufficiently large.
By Lemma~\ref{cor:equilibrium-estimates-Potts}, the stationary distribution $\pi$ is within $o(1)$ total-variation distance of a $(\xi,1-\xi)$ mixture of $\pi^\dis$ and $\pi^\ord$ for $\xi$ defined in~\eqref{eq:xi}. By Lemma~\ref{lem:existcstar}, there is a constant $\hat{c}_{*}(\xi)$ such that if $m_0 = m_* + \hat c_* n^{-1/2} + o(n^{-1/2})$, 
then $\|\Pr_{\hat \nu^{\otimes}(m_0)}(\tau_{\gamma}^- <\tau_{-\gamma}^+)  - \xi\|_{\TV} \le \epsilon$. 
Moreover, by Lemma~\ref{lem:Pottsexittime}, the minimum of $\tau_{\gamma}^-$ and $\tau_{\gamma}^+$ is $O(n)$ with probability $1-O(\gamma^{-1})$; 
by Lemma~\ref{lem:nondominantcoordinates}, if $T_1 = \tau_{\gamma}^- \wedge \tau_{\gamma}^+$ and $m_1 = S_{T_1,1}$ then $\| S_{T_1}- (m_1, \frac{1-m_1}{q-1},...,\frac{1-m_1}{q-1})\|_1 = O(n^{-1/2}\log n)$ with probability $1-O(n^{-2})$.
Finally, applying Theorem~\ref{thm:PottsHitLeftTarget} and Lemma~\ref{cor:equilibrium-estimates-Potts} to the configuration at time $T_1$ as in the off-critical case, together with the spin symmetry for convergence to ordered phases,  implies that there exists $T=T_1+ O(n\log n ) $ such that 
\begin{align*}
    \|\Pr_{\hat \nu^{\otimes}(m_0)}(\sigma_T \in \cdot) - ((1-\xi) \pi^\ord + \xi\pi^\dis)\|_{\TV} \le \exp(-\Omega(\gamma^2)) + \epsilon + O(\gamma^{-1}) + O(n^{-2})\,,
\end{align*}
which will be less than $2\epsilon$ for $\gamma$ large. 

It remains to discuss Item 4. If $m_0 \ge \frac{1}{q} + \omega(n^{-1/2})$, then together with spin symmetry and the vertex-permutation invariant of the initialization, 
Theorem~\ref{thm:lowtempcoupling}
implies that
$\|\Pr_{\hat \nu^{\otimes}(m_0)} (\sigma_T \in \cdot) - \pi^{\ord}\|_{\TV} = o(1)$ for  $T=O(n\log n)$. If $\frac{1}{q} \le m_0 \le \frac{1}{q} + \omega(n^{-1/2})$ then we first apply Lemma~\ref{lem:1/q-initialization} to get a coordinate (by symmetry a uniform at random one) to obtain $\gamma n^{-1/2}$ separation from the rest, then apply Theorem~\ref{thm:lowtempcoupling} to get the same. 
On the other hand, by 
Lemma~\ref{cor:equilibrium-estimates-Potts} we have $\|\pi^\ord - \pi\|_{\mathrm{TV}} = o(1)$, and thus item 4 follows from the triangle inequality. 

The slow mixing results are exactly the statement of Theorem~\ref{thm:Potts-slow}.
\end{proof}

We also include a proof of Theorem~\ref{thm:simulated-annealing-Potts} on initializations from the Potts Gibbs measure at a different temperature, to justify why it follows from the above arguments.

\begin{proof}[\textbf{\emph{Proof of Theorem~\ref{thm:simulated-annealing-Potts}}}]
    It is evident at this point that the only properties of the initialization used in our proof of Theorem~\ref{thm:Potts-Glauber} were on the size of its largest color class relative to $m_* n$, and the differences between its non-dominant color counts being $O(\sqrt{n})$. 
    The requisite properties for samples from $\pi_{\beta_0}$ are given by Corollary~\ref{cor:equilibrium-estimates-Potts}. 
    For the choice of $d_*(\beta,q)$, in item 1 when $\beta \in (\betau,\betac)$, 
    we choose $d_*(\beta,q)= \betac$, 
    and in item 3, the choice is $d_*(\beta,q)= \inf\{d: \frac{\thetar(d)(q-1)+1}{q} >m_*(\beta)\}$, as the typical size of the largest coordinate in a sample from $\pi_d$ is $\frac{1}{q} (\thetar(d) (q-1)+1)$ per Corollary~\ref{cor:equilibrium-estimates-Potts}. 
    Finally note that any $\pi_{\beta_0}$ could work for item 4 since the largest coordinate of a sample is always at least $1/q$.
\end{proof}

\bibliographystyle{plain} 
\bibliography{random_start}

\end{document}